\documentclass[11pt, draft]{amsart}
\usepackage{amssymb, amstext, amscd, amsmath, amssymb}
\usepackage{mathtools, xypic, paralist, color, dsfont, rotating}
\usepackage{verbatim}
\usepackage{enumerate,enumitem}
\usepackage{float}
\usepackage{setspace}

\usepackage{tikz}

\usepackage{tikz-cd}

\reversemarginpar

\usepackage{tikz}

\floatstyle{boxed} 
\restylefloat{figure}

\numberwithin{equation}{section}

\usepackage[a4paper]{geometry}
\usepackage{quoting}
\quotingsetup{vskip=.1in}
\quotingsetup{leftmargin=.17in}
\quotingsetup{rightmargin=.17in}

\let\OLDthebibliography\thebibliography
\renewcommand\thebibliography[1]{
  \OLDthebibliography{#1}
  \setlength{\parskip}{0pt}
  \setlength{\itemsep}{2pt plus 0.5ex}
}

\makeatletter
\def\@cite#1#2{{\m@th\upshape\bfseries%
[{#1\if@tempswa{\m@th\upshape\mdseries, #2}\fi}]}}
\makeatother
\theoremstyle{plain}
\newtheorem{theorem}{Theorem}[section]
\newtheorem{corollary}[theorem]{Corollary}
\newtheorem{proposition}[theorem]{Proposition}
\newtheorem{lemma}[theorem]{Lemma}
\theoremstyle{definition}
\newtheorem{definition}[theorem]{Definition}
\newtheorem{example}[theorem]{Example}

\newtheorem{remark}[theorem]{Remark}

\theoremstyle{remark}

\mathtoolsset{centercolon}
  \newcommand{\A}{{\mathcal{A}}}
  \newcommand{\B}{{\mathcal{B}}}
  
  \newcommand{\D}{{\mathcal{D}}}
  \newcommand{\E}{{\mathcal{E}}}
  \newcommand{\F}{{\mathcal{F}}}
  \newcommand{\G}{{\mathcal{G}}}

  \newcommand{\I}{{\mathcal{I}}}
  \newcommand{\J}{{\mathcal{J}}}
  \newcommand{\K}{{\mathcal{K}}}
\renewcommand{\L}{{\mathcal{L}}}
  \newcommand{\M}{{\mathcal{M}}}
  
\renewcommand{\O}{{\mathcal{O}}}
\renewcommand{\P}{{\mathcal{P}}}
  \newcommand{\Q}{{\mathcal{Q}}}
  \newcommand{\R}{{\mathcal{R}}}
\renewcommand{\S}{{\mathcal{S}}}
  \newcommand{\T}{{\mathcal{T}}}

  \newcommand{\X}{{\mathcal{X}}}
  \newcommand{\Y}{{\mathcal{Y}}}
  \newcommand{\Z}{{\mathcal{Z}}}
\def\al{\alpha}

\def\de{\delta}

\def\ka{\kappa}
\def\la{\lambda}

\def\si{\sigma}

\newcommand{\bC}{\mathbb{C}}

\newcommand{\bN}{\mathbb{N}}

\newcommand{\fC}{{\mathfrak{C}}}

\newcommand{\fH}{{\mathfrak{H}}}

\newcommand{\fK}{{\mathfrak{K}}}

\newcommand{\fS}{{\mathfrak{S}}}

\newcommand{\fX}{{\mathfrak{X}}}

\newcommand{\foral}{\text{ for all }}
\newcommand{\qand}{\quad\text{and}\quad}

\newcommand{\qiff}{\quad\text{if and only if}\quad}
\newcommand{\qfor}{\quad\text{for}\quad}

\newcommand{\ca}{\mathrm{C}^*}

\newcommand{\cenv}{\mathrm{C}^*_{\textup{env}}}
\newcommand{\tenv}{\mathcal{T}_{\textup{env}}}

\newcommand{\ol}{\overline}
\newcommand{\wt}{\widetilde}

\newcommand{\cl}[1]{\mathcal{#1}}
\newcommand{\bb}[1]{\mathbb{#1}}
\newcommand{\ad}{\operatorname{ad}}

\newcommand{\env}{\operatorname{\textup{env}}}

\newcommand{\id}{{\operatorname{id}}}

\newcommand{\spn}{\operatorname{span}}

\newcommand{\sca}[1]{\left\langle#1\right\rangle} 
\newcommand{\nor}[1]{\left\Vert #1\right\Vert} 
\newcommand{\un}[1]{{\underline{#1}}} 
\newcommand{\quo}[2]{{\raisebox{.1em}{$#1$}\left/ \, \raisebox{-.1em}{$#2$}\right.}} 

\addtocontents{toc}{\protect\setcounter{tocdepth}{1}}

\begin{document}

\title[Morita equivalence for operator systems]{Morita equivalence for operator systems}

\author[G.K. Eleftherakis]{George K. Eleftherakis}
\address{Department of Mathematics\\Faculty of Sciences\\University of Patras\\26504 Patras\\Greece}
\email{gelefth@math.upatras.gr}

\author[E.T.A. Kakariadis]{Evgenios T.A. Kakariadis}
\address{Department of Mathematics\\ National and Kapodistrian University of Athens\\ Athens\\ 1578 84\\ Greece}
\email{evkakariadis@math.uoa.gr}

\author[I.G. Todorov]{Ivan G. Todorov}
\address{School of Mathematical Sciences\\ University of Delaware\\ 501 Ewing Hall\\ Newark\\ DE 19716\\ USA}
\email{todorov@udel.edu}



\begin{abstract}
We define $\Delta$-equivalence for operator systems and show that it is identical to stable isomorphism. 
We define $\Delta$-contexts and bihomomorphism contexts and show that two operator systems are $\Delta$-equivalent if and only if they can be placed in a $\Delta$-context, equivalently, in a bihomomorphism context.
We show that nuclearity for a variety of tensor products is an invariant for $\Delta$-equivalence and that function systems are $\Delta$-equivalent precisely when they are order isomorphic. 
We prove that $\Delta$-equivalent operator systems have equivalent categories of representations. 
As an application, we characterise $\Delta$-equivalence of graph operator systems in combinatorial terms.
We examine a notion of Morita embedding for operator systems, showing that mutually $\Delta$-embeddable operator systems have orthogonally complemented $\Delta$-equivalent corners when represented in the double dual of their C*-envelopes. 
\end{abstract}

\maketitle

\tableofcontents

\section{Introduction}

Morita equivalence was introduced in the 1950's, by K.\ Morita, and provides a rigorous framework that allows to identify rings with equivalent categories of representations. 
It replaces the notion of an isomorphism by a weaker, more flexible, identification that takes into account the variety of concrete manifestations of the abstract object rather than its precise internal structure \cite{Bas62}.

Morita equivalence was first studied in the context of non-commutative analysis by M. Rieffel \cite{Rie74jpaa}, who defined functional analytic versions of the concept for C*-algebras and von Neumann algebras. 
Following the advent of operator space theory \cite{ER00}, Morita equivalence of nonselfadjoint operator algebras was defined and examined in depth in \cite{BMP00}, where several of the fundamental \lq\lq Morita Theorems'' \cite{Bas62} were proved (see also \cite{Ble01ms, Ble01jpaa}). 
In \cite{bk}, similar developments were pursued for nonselfadjoint dual operator algebras, inspiring further work in the dual context in \cite{blecherII, bk-TAMS, bk1}. 

The idea of taking into account the selfadjoint structure, as far as it is present in any nonselfdjoint algebra, inspired the notion of TRO-equivalence \cite{Ele12}, which led to an equivalence of Morita type, called $\Delta$-equivalence, generalising Rieffel's strong Morita equivalence of C*-algebras \cite{Rie74jpaa}, and examined for operator spaces, operator algebras, and their dual counterparts in \cite{Ele08, Ele14, Ele18, Ele19, EK17, EP+, EP08, EPT10}.
It turns out, in particular, that $\Delta$-equivalence is well-suited to describe stable isomorphism of operator algebras and operator spaces (that is, the isomorphism of their amplified copies obtained by tensoring with -- depending on the setting -- either the C*-algebra of compact operators or the von Neumann algebra of all bounded linear operators on a Hilbert space). 

While Morita equivalence for nonselfadjoint operator algebras and operator spaces has reached a fairly mature stage and aspects of it are well-understood, no analogous development has taken place for operator systems -- that is, unital selfadjoint operator spaces. 
Nevertheless, there are good reasons to undertake such a development. 
Indeed, operator systems -- introduced by W. Arveson in \cite{Arv69} -- have played a cornerstone role in building quantised functional analysis \cite{ER00, Pau02}, and are currently enjoying a surge of interest, both from a purely theoretical perspective \cite{DK19, KPTT11, KPTT13, KKM21} and in applications to the area of quantum information theory, where they are studied as non-commutative graphs \cite{DSW13, {Sta16}, BTW21, TT20, bhtt}. 
Stable isomorphism of operator systems has further proved essential in recent developments in non-commutative geometry and mathematical physics \cite{CS20}. 
We note that the authors of \cite{CS20} pose explicitly the question of developing a Morita theory of operator systems, parallel to the analogous theories for C*-algebras, operator algebras and operator spaces, cf. comments prior to \cite[Proposition 2.37]{CS20}.

In the present paper, we address this, initiating Morita theory in the operator system category. 
Some of its aspects -- notably the characterisation of stable equivalence (Theorem \ref{T:stable}) -- readily follow from the theory, already developed for general operator spaces \cite{EK17}. 
We undertake a further step by providing an abstract characterisation of $\Delta$-equivalent operator  systems, in terms of Morita contexts, in the spirit of \cite{Bas62}, \cite{Rie74jpaa} and \cite{BMP00} (Theorem \ref{th_ac}). 
We exhibit categorical implications of $\Delta$-equivalence, showing that the natural categories of representations of $\Delta$-equivalent operator systems have to be equivalent (Theorems \ref{B0} and \ref{B02}). 
This parallels the C*-algebra case \cite{Rie74jpaa}. 
We show that $\Delta$-equivalence reduces to a (complete) order isomorphism in the case of function systems (the operator system counterparts of commutative C*-algebras) and, perhaps most intriguingly, we uncover a relation between $\Delta$-equivalence and non-commutative graph isomorphisms \cite{Sta16}. 

The paper is organised as follows. 
After collecting the necessary preliminaries and setting notation in Section \ref{s_prel}, we introduce, in Section \ref{s_Delta}, TRO-equivalent and $\Delta$-equivalent operator systems. 
In the case the operator systems are C*-algebras, the latter concept reduces to strong Morita equivalence \cite{Rie74jpaa}. 
Making full use of the existing operator space theory, we show that $\Delta$-equivalence is identical to stable isomorphism. 
The definition of $\Delta$-equivalence utilises concrete completely order isomorphic copies of the operator systems in an essential way. 
We define two entirely abstract versions of the concept: \emph{$\Delta$-contexts} are modeled on C*-algebraic imprimitivity bimodules \cite{Rie74jpaa} (see also \cite{BMP00} for the corresponding nonselfadjoint variants), while \emph{bihomomorphism contexts} are an abstraction of mutual homomorphic embeddings of non-commutative graphs
(see \cite{Sta16}). 
We show that two operator systems are $\Delta$-equivalent precisely when they can be placed in a bihomomorphism (equivalently, in a $\Delta$-) context.
The introduction of bihomomorphism contexts can be seen as a first step towards the development of an 
abstract, representation-free, view on non-commutative graph homomorphisms. 
The latter notion, introduced in \cite{Sta16}, has gained importance both as an early sample of 
non-commutative combinatorics \cite{weaver} 
and as a first step towards the quantisation of graph homomorphisms carried out in 
\cite{TT20, bhtt}.

In Section \ref{ss_ci}, we examine the categorical implications of $\Delta$-equivalence. 
We consider two natural categories of representations of an operator system, and show that if two operator systems are 
$\Delta$-equivalent then the corresponding categories of representations are equivalent via natural transformations. 

Section \ref{s_mei} contains further consequences of $\Delta$-equivalence. 
In particular, we find an operator system version of the well-known fact that strongly Morita equivalent C*-algebras have isomorphic ideal lattices. 
In the operator system category, ideals are replaced by kernels in the operator systems \cite{KPTT13} that are bimodules over the multiplier algebras \cite{blm}. 
We show that $\Delta$-equivalence preserves nuclearity with respect to several types of tensor products, including all associative ones. 
As a consequence, we show that nuclearity, exactness, the operator system local lifting property (OSLLP), the weak expectation property (WEP), and the double commutant expectation property (DCEP) are invariants for $\Delta$-equivalence.
This extends results of M. Forough and M. Amini \cite{FA18} from the category of C*-algebras to the category of operator systems.

Section \ref{s_fs} is restricted to the case of function systems. In addition to the aforementioned order isomorphism characterisation of $\Delta$-equivalence, we define the notion of a centre of an arbitrary operator system and show that $\Delta$-equivalence restricts to order isomorphism between the corresponding centres. 

Section \ref{s_NCgraphs} specifies the aforementioned connection between non-commutative graph homomorphisms and $\Delta$-equivalence. 
We show that $\Delta$-equivalence coincides with TRO-equivalence (in their irreducible inclusion representations), and also with bihomomorphisms implemented by the same TRO (Theorem \ref{t_ncg} and Proposition \ref{p_coho}, respectively).
In particular, we determine when graph operator systems are $\Delta$-equivalent in graph-theoretic terms, namely, if and only if they are pullbacks of isomorphic graphs (Corollary \ref{c_graphs}).

Finally, in Section \ref{s_emb}, we consider a Morita-type embedding for operator systems, that represents ``half'' of the $\Delta$-equivalence relations. 
This relation is transitive but not anti-symmetric up to $\Delta$-equivalence (Remark \ref{R:abelian}).
While two mutually $\Delta$-embeddable operator systems $\cl S$ and $\cl T$ are not necessarily $\Delta$-equivalent, we show that $\cl S$ and $\cl T$ possess two orthogonally complemented $\Delta$-equivalent corners as represented in the double duals of their C*-envelopes (Theorem \ref{T:anti-sym}).
Similar concepts of embeddings have been studied by G. Eleftherakis for dual operator algebras and dual operator spaces \cite{Ele18, Ele19}.
In those cases the embedding relation turns out to be a partial order when restricted to von Neumann algebras and weak stable isomorphisms \cite[Remark 4.6, Remark 4.7]{Ele19}.

\subsection*{Acknowledgements}

We would like to thank Mick Gielen and Walter van Suijlekom for their helpful comments.
We would like to thank Alexandros Chatzinikolaou, Gage Hoeffer, Nikolaos Koutsonikos-Kouloumpis and Ioannis Apollonas Paraskevas for their helpful comments and corrections.

Evgenios Kakariadis acknowledges support from EPSRC (grant No.\ EP/ T02576X/1).
Ivan Todorov was supported by NSF grants CCF-2115071 and DMS-2154459.
The authors acknowledge support from the Heilbronn Institute for Mathematical Research (HIMR) and the UKRI/EPSRC Additional Funding Programme for Mathematical Sciences.

\section{Preliminaries}\label{s_prel}

\subsection{Operator spaces}

By the term ``closed'' we will mean closed in the norm topology.
If $\fH$ is a linear space in a normed space and $\fX$ is a linear space that acts on $\fH$ on the left then we write $[\fX \fH]$ for the closed linear span of the set $\{x h \mid x \in \fX, h \in \fH\}$.

If $H$ and $K$ are Hilbert spaces, we write as usual $\cl B(H,K)$ for the space of all bounded linear 
operators from $H$ into $K$, and set $\cl B(H) = \cl B(H,H)$. 
An \emph{operator space} is a (not necessarily closed) subspace of $\B(H,K)$ for some Hilbert spaces $H$ and $K$, while an \emph{operator system} is an operator subspace of $\B(H)$ for some Hilbert space $H$  that contains the identity operator $I_H$ on $H$ and is closed under taking adjoints.
An \emph{operator algebra} is an operator subspace of some $\B(H)$ that is closed under multiplication, and  a \emph{C*-algebra} is a closed selfadjoint operator algebra. 
Given a subset $\fS \subseteq \cl B(H)$, we write ${\rm C}^*(\fS)$ for the unital C*-algebra, generated by $\fS$. 

We say that an operator space $\X \subseteq \B(H, K)$ \emph{acts non-degenerately} (resp. \emph{is non-degenerate}) if $[\X H] = K$ and $[\X^* K] = H$ (resp. $I_H \in [\X^* \X]$ and $I_K \in [\X \X^*]$). 
We denote by $C(\X)$ (resp. $R(\X)$) the column (resp. row) operator space with entries in $\cl X$ of sufficiently large cardinality that will be clear from the context; whenever we wish to specify that the cardinality is $k\in \bb{N}$, we write 
$C_k(\X)$ (resp. $R_k(\X)$). 
We will require also the following notion of semi-units.

\begin{definition}\label{D:semiunit}
Let $\X \subseteq \B(H, K)$ be an operator space.
We say that the pair $((\un{x}_i)_i, (\un{y}_i)_i)$ (resp. $((\un{z}_i)_i, (\un{w}_i)_i)$) of 
nets of elements in $C(\cl X)$ (resp. $R(\cl X)$) is a \emph{semi-unit for $\X$} (resp. $\cl X^*$) if $\un{x}_i$ and $\un{y}_i$ (resp. $\un{z}_i$ and $\un{w}_i$) are finitely supported and
\[
\lim_i \un{x}_i^* \un{y}_i = I_H \ \ \ (\mbox{resp. } \lim_i \un{z}_i \un{w}_i^* = I_K).
\]
\end{definition}

The abstract versions of the notions of an operator space, an operator system and an operator algebra (with the extra assumption of a contractive approximate identity for the case of operator algebras) are rather well-known and will be used freely in the paper; we refer the reader to \cite{blm, Pau02, Pis03} for details on this, as well as for other basic concepts of operator space theory. 
An abstract operator space will be called \emph{non-degenerate} if it admits a completely isometric map onto a concrete non-degenerate operator space. For operator systems $\cl S$ and $\cl T$, we write $\cl S\simeq_{\rm c.o.i.}\cl T$ if there exists a unital complete order isomorphism from $\cl S$ onto $\cl T$. 

\begin{remark}\label{R:nd reduct}
Let $\X \subseteq \B(H, K)$, and set $K' := [\X H]$ and $H' := [\X^* K]$.
With respect to the decompositions $H = H' \oplus (H')^{\perp}$ and $K = K' \oplus (K')^{\perp}$, one can write
\[
x = \begin{bmatrix} P_{K'} x |_{H'} & 0 \\ 0 & 0 \end{bmatrix} \foral x \in \X;
\]
moreover, the compression $P_{K'} \X |_{H'} \subseteq \B(H', K')$ of $\cl X$ acts non-degenerately.
Indeed, by definition, $[\X H'] \subseteq [\X H] \subseteq K'$.
On the other hand, let $\eta \in K'$ with $\eta \perp [\X H']$.
Then $[\X^* \eta] \perp H'$.
However $[\X^* \eta] \subseteq [\X^* K] \subseteq H'$ and so $[\X^* \eta] = \{0\}$.
Consequently, $\eta \perp [\X H] = K'$, and thus $\eta =0$.
Therefore $K' = [\X H'] = [P_{K'} \X |_{H'} (H')]$.
A similar argument applies to $H'$, and thus $P_{K'} \X |_{H'}$ acts non-degenerately.
\end{remark}

For an operator system $\S$, there exists a unital completely isometric map $\iota_{\max} \colon \S \to \cl B(H)$ so that the (so called \emph{maximal}) C*-algebra $\ca_{\max}(\S) := \ca(\iota_{\max}(\S))$ has the following (universal) property: 
for every unital completely positive map $\phi \colon \S \to \B(K)$ 
there exists a $*$-homomorphism $\wt{\phi} \colon \ca(\iota_{\max}(\S)) \to \ca(\phi(\S))$
such that $\wt{\phi}\circ \iota_{\max} = \phi$. 
The latter object was introduced by E. Kirchberg and S. Wasserman \cite{KW98}, 
and was recently studied in the larger category of 
\lq\lq non-unital operator systems'' 
(that is, matricially ordered operator spaces that admit a completely isometric complete order embedding into $\cl B(H)$ for some Hilbert space $H$) in \cite{KKM21}. 

We will denote by $\fK$ the C*-algebra of compact operators on $\ell^2$.
We will use the notation $\otimes$ for the minimal (that is, norm closed spatial) tensor product of operator spaces, and will write $\odot$ for the algebraic tensor product. 

\subsection{Ternary rings of operators}

A \emph{ternary ring of operators} (TRO) $\M \subseteq \B(H, K)$ is an operator subspace such that $\M \M^* \M \subseteq \M$.
If $\M$ is closed then it is an $\A$-$\B$-imprimitivity bimodule, in the sense of Rieffel \cite{Rie74jpaa}, for the C*-algebras
\[
\A = [\M \M^*] \qand \B = [\M^* \M],
\]
realising a strong Morita equivalence between $\A$ and $\B$. 
In this case, there are nets  
\[
(\underline{m}_1^\mu)_\mu \subseteq R(\M)
\qand
(\underline{m}_2^\la)_ \la \subseteq C(\M)
\]
of finitely supported strings of elements in $\M$ such that, if
\[
a_\mu := \underline{m}_1^\mu \cdot (\underline{m}_1^\mu)^* \in \A
\qand
b_\la := (\underline{m}_2^{\la})^* \cdot \underline{m}_2^{\la} \in \B,
\]
then the net $(a_\mu)_{\mu}$ (resp. $(b_\la)_{\la}$) is a contractive approximate identity for $\cl A$ (resp. $\cl B$). 
Consequently they define a left and a right semi-unit for $\M$ and $\M^*$.
Thus a closed TRO $\M$ satisfies the identity $[\M \M^* \M] = \M$.

In the case where $\A$ and $\B$ admit countable approximate identities, the corresponding nets can be replaced by sequences and can be assumed to have increasing supports.
Therefore we obtain contractions
$\underline{m}_1 \in R(\M)$
and
$\underline{m}_2 \in C(\M)$
over $\bb{N}$, such that
\[
\sum_{n=1}^{\infty} m_{1,n} m_{1,n}^* m = m
\qand
\sum_{n=1}^{\infty} m_{2,n}^* m_{2,n} m^* = m^* \ \
\foral
m \in \M;
\]
(see e.g. \cite[Lemma 2.3]{Bro77}).
We isolate the following standard fact about closed TRO's.

\begin{remark}\label{R:ind rep}
Let $\M \subseteq \B(H, K)$ be a closed TRO, and set $\A = [\M \M^*]$ and $\B = [\M^* \M]$.
A $*$-representation $\si \colon \B \to \B(K)$ induces a completely contractive map
\[
\phi \colon \M \to \B(K, \M \otimes_\B K); \ \phi(m) \eta = m \otimes \eta,
\]
where $\M \otimes_\B K$ is the completion of the balanced algebraic tensor product $\mathcal{M} \odot_\B K$ with inner product defined by
\[
\sca{m_1 \otimes \eta_1, m_2 \otimes \eta_2} = \sca{\eta_1, \si(m_1^* m_2) \eta_2},
\ \ \ \ m_1, m_2 \in \M, \eta_1, \eta_2 \in K.
\]
Consequently, the mapping $\pi(a)$, given by $\pi(a) (m \otimes \eta) = (am) \otimes \eta$ defines a $*$-representation $\pi \colon \A \to  \B(\M \otimes_\B K)$.
In particular, $\phi$ is a TRO morphism of $\M$ such that $\phi(m_1) \phi(m_2)^* = \pi(m_1 m_2^*)$ and $\phi(m_1)^* \phi(m_2) = \si(m_1^* m_2)$ for all $m_1, m_2 \in \M$.
\end{remark}

\subsection{TRO envelopes}

Let $\X$ be an operator space.
For a completely isometric map $i \colon \X \to \M$ into a TRO $\M$, let $\T(i(\X))$ be the TRO densely spanned by the products
\[
i(x_1) i(x_2)^* i(x_3) i(x_4)^* \cdots i(x_{2n})^* i(x_{2n+1})
\qfor
n\geq 0, x_1,\dots, x_{2n+1} \in \X;
\]
we say that $(\T(i(\X)),i)$ is \emph{a TRO extension of $\X$}.
There is a particular embedding of $\X$ in the injective envelope $\I(\S(\X))$ of a canonical operator system $\S(\X)$, associated with $\cl X$. 
If $\X$ is unital, we set $\S(\X) = \X + \X^*$, and if not, we set
\[
\S(\X) = \left\{ \begin{bmatrix} \la & x_1 \\ x_2 & \mu \end{bmatrix} \mid x_1, x_2 \in \X, \la, \mu \in \bC\right\}
\]
(the latter object is called the \emph{Paulsen operator system} of $\cl X$).
Writing 
\begin{equation}\label{eq_I12}
\I(\S(\X)) = \begin{bmatrix} \I_{11}(\X) & \I_{12}(\X) \\ \I_{21}(\X)^* & \I_{22}(\X) \end{bmatrix}
\end{equation}
in a matrix form, the injective envelope $\I(\X)$ of $\X$ is defined to be the corner $\I_{12}(\X)$; in particular, it coincides with $\I(\X + \X^*)$ when $\X$ is unital.
The TRO extension of $\X$ inside $\I(\cl S(\X))$ will be denoted by $\tenv(\X)$.
M. Hamana \cite{Ham99} showed that $\tenv(\X)$ possesses the following universal property: given any TRO extension $(\M, j)$ of $\X$ there exists a (necessarily unique) surjective triple morphism $\theta \colon \M \to \tenv(\X)$ such that $\theta(j(x)) = x$, for all $x\in \cl X$.
The operator space $\tenv(\X)$ is called \emph{the ternary} (or \emph{TRO}) \emph{envelope} of $\X$.

If $\X$ is unital then the embedding $\X \hookrightarrow \I(\X)$ is unital and the well-known Choi-Effros construction \cite{CE77} endows $\I(\X)$ with a C*-algebraic structure; thus, in this case, the TRO envelope is a C*-algebra, denoted by $\cenv(\X)$.
The existence of the C*-envelope was also established by M. Hamana \cite{Ham79}, following Arveson's quantisation program \cite{Arv69}.
For an alternative proof of Hamana's Theorem, using boundary subsystems, the reader is directed to \cite{Kak11-2}.
In fact, the arguments in \cite{Kak11-2} suffice to show the existence of the TRO envelope as well, even though this is not mentioned therein.

By \cite[Corollary 4.6.12]{blm}, if $\I(\X)$ is an injective envelope of an operator space $\X$ then $\I(\X) \otimes \B(\ell^2)$ is an injective envelope of $\X \otimes \fK$.
Therefore, if $\X$ is unital or an operator algebra then 
\begin{equation}\label{eq_fK}
\tenv(\X \otimes \fK) \simeq \tenv(\X) \otimes \fK \simeq \cenv(\X) \otimes \fK.
\end{equation}
In particular $\tenv(\X \otimes \fK)$ is a C*-algebra and has the co-universal property of the C*-envelope, in the sense that if $\phi \colon \X \otimes \fK \to \B(H)$ is a completely isometric map then there is a $*$-epimorphism $\ca(\phi(\X \otimes \fK)) \to \cenv(\X) \otimes \fK$ fixing $\X \otimes \fK$.
Another proof is provided in \cite[Proposition 2.37]{CS20}.

There is an alternative way of obtaining the C*-envelope of a unital operator space $\X$, described for operator algebras by M. Dritschel and S. McCullough \cite{DM05}, and simplified for operator systems by W. Arveson \cite{Arv08}.
A \emph{dilation} $\phi' \colon \X \to \B(K)$ of a unital completely contractive map $\phi \colon \X \to \B(H)$ is a completely contractive map with $H \subseteq K$ such that $\phi(x) = P_H \phi'(x)|_H$ for every $x \in \X$.
The map $\phi$ is called \emph{maximal} if it only admits trivial dilations.
Then the C*-envelope is (canonically) $*$-isomorphic to the C*-algebra generated by the image of a maximal unital completely isometric map.

It is not always true that every $*$-representation $\pi \colon \cenv(\X) \to \B(H)$ is the unique extension of $\pi|_{\X}$.
(An example can be constructed for the nonselfajoint operator algebra generated by the Cuntz isometries $\{s_n\}_{n \in \bN}$ of $\O_\infty$; see \cite{DS18, KR20} for more examples.)
An operator space $\X$ is called \emph{hyperrigid} if whenever $\pi$ is a $*$-representation of $\cenv(\X)$, its restriction $\pi|_\X$ to $\X$ is a maximal map. 
A key fact in these considerations is that the maximal maps completely norm $\X$.
Realising $\X \subseteq \B(H)$, maximal maps have the property that they extend uniquely to a completely positive map on $\ca(\X)$ that is a $*$-representation.
Initially Arveson considered maps with the additional property that the extension is irreducible.
We will refer to those as \emph{Choquet} representations.
Arveson proved that the Choquet representations completely norm a separable operator system; 
K. Davidson and M. Kennedy \cite{DK19} established the latter fact in full generality.

\subsection{Multipliers of operator spaces}

For a linear space $\cl V$, we often use the notation $\mbox{}_{\cl A}\cl V_{\cl B}$ to designate the fact that $\cl V$ is a left $\cl A$-module and a right $\cl B$-module. 
Employing the completely isometric inclusions
\[
\X \hookrightarrow \S(\X) \hookrightarrow \I(\S(\X))
\]
and resorting to the decomposition (\ref{eq_I12}) of $\I(\S(\X))$, we write $\M_l(\X)$ for 
\emph{the operator space of left multipliers}
of $\X$ \cite[Section 4.5]{blm}.
Thus, as an operator algebra, $\M_l(\X)$ is completely isometrically isomorphic to
\[
\I\M_l(\X) := \{a \in \I_{11}(\X) \mid a\X \subseteq \X\}.
\]
The diagonal of $\M_l(\X)$ is denoted by $\A_l(\X)$ and is a unital C*-algebra.
In particular,
\[
\A_l(\X) \simeq \{a \in \I_{11}(\X) \mid a\X \subseteq \X \text{ and } a^* \X \subseteq \X \}
\]
via a $*$-isomorphism.
Similar facts hold for the operator space $\M_r(\X)$ of the right multipliers and its diagonal $\A_r(\X)$.
The maps
\[
\A_l(\X) \times \X \to \X; \ (a, x) \mapsto ax 
\qand 
\X \times \A_r(\X) \to \X; \ (x, b) \mapsto xb
\]
are completely contractive and $\X$ is an operator $\A_l(\X)$-$\A_r(\X)$-bimodule \cite[Parargaph 4.6.6]{blm}.
The C*-bimodule ${}_{\A_l(\X)} \X_{\A_r(\X)}$ is non-degenerate and $\A_l(\X)$ and $\A_r(\X)$ act faithfully on $\X$.
If $\S$ is an operator system, we moreover have that
\[
\A_l(\S) = \A_r(\S) = \{a \in \I(\S) \mid a \S \subseteq \S \text{ and } a^* \S \subseteq \S\};
\]
we write $\A_\S$ for the latter C*-algebra.
Letting $\iota_{\env} \colon \S \to \cenv(\S) \subseteq \I(\S)$ be the canonical embedding, we note that the unitality condition yields that $\A_\S \subseteq \iota_{\env}(\S)$; we can thus consider $\cl A_{\cl S}$ as being contained in $\cl S$. 

Recall that, if $\cl A$ is a unital C*-algebra, an operator system $\cl S$ is called an \emph{operator $\cl A$-system} if $\cl S$ is a C*-bimodule over $\cl A$ and 
\[
A^*\cdot M_m(\cl S)^+ \cdot A\subseteq M_n(\cl S)^+ \ \mbox{ whenever } A\in M_{m,n}(\cl A).
\]
We refer the reader to \cite[Section 15]{Pau02} for further details on operator C*-systems and note that an operator system $\cl S$ is automatically an operator $\cl A_{\cl S}$-system. 

\subsection{Morita equivalence for operator spaces}

Following \cite{Ele08, Ele12, Ele14, EP08, EPT10}, G. Eleftherakis and E. Kakariadis \cite{EK17}, and more recently G. Eleftherakis and E. Papapetros \cite{EP+}, examined the notion of Morita equivalence for closed operator spaces.
We say that the closed operator spaces $\X \subseteq \B(H_1, H_2)$ and $\Y \subseteq \B(K_1, K_2)$ are \emph{TRO-equivalent} if there exist TRO's $\M_1 \subseteq \B(H_1, K_1)$ and $\M_2 \subseteq \B(H_2, K_2)$ such that
\[
\X = [\M_2^* \Y \M_1]
\qand
\Y = [\M_2 \X \M_1^*];
\]
in this case, we say that $\cl X$ and $\cl Y$ are TRO-equivalent \emph{via} $\cl M_1$ \emph{and} $\cl M_2$. 
We say that $\cl X$ and $\cl Y$ are \emph{$\Delta$-equivalent} if there exist completely isometric maps $\phi \colon \X \to \B(H_1, H_2)$ and $\psi \colon \Y \to \B(K_1, K_2)$ such that $\phi(\X)$ and $\psi(\Y)$ are TRO-equivalent.
It was shown in \cite{EK17} that TRO equivalence is an equivalence relation (the adverb ``strongly'', used in \cite{EK17} when referring to these equivalences, is dropped herein).
The operator spaces $\cl X$ and $\cl Y$ are called \emph{stably isomorphic} if $\cl X\otimes\fK\simeq \cl Y\otimes\fK$ via a complete isometry. 

We collect some results from \cite{EK17} which will be used subsequently:
\begin{enumerate}
\item Stably isomorphic closed operator spaces are $\Delta$-equivalent \cite[Theorem 4.3]{EK17}.
\item If two closed operator spaces $\X$ and $\Y$ are TRO-equivalent via $\M_1$ and $\M_2$ and the 
C*-algebras $[\M_1^* \M_1]$, $[\M_1 \M_1^*]$, $[\M_2^* \M_2]$ and $[\M_2 \M_2^*]$ have countable approximate identities then $\X$ and $\Y$ are stably isomorphic \cite[Corollary 4.7]{EK17}.
\item Separable closed operator spaces are $\Delta$-equivalent if and only if they are stably isomorphic
\cite[Corollary 4.8]{EK17}.
\item $\Delta$-equivalent closed operator spaces have $\Delta$-equivalent TRO envelopes \cite[Theorem 5.10]{EK17}.
\end{enumerate}

For future reference, we note that the proof of item (iv) proceeds as follows.
If $\X$ and $\Y$ are $\Delta$-equivalent then we can choose $\phi$ so that $\phi(\X) \subseteq \tenv(\X)$ and there exists a completely isometric map $\psi$ of $\Y$ such that $\psi(\Y)$ is TRO-equivalent to $\phi(\X)$ and that $\T(\psi(\Y))$ is TRO-isomorphic to $\tenv(\Y)$.
We also isolate the following remark from \cite{EK17} for further use.

\begin{remark}\label{R:m.a.i.}
Suppose that $\X\subseteq \B(H_1, H_2)$ and $\Y\subseteq \B(K_1, K_2)$ are operator spaces that are 
TRO-equivalent via $\M_1 \subseteq \B(H_1, K_1)$ and $\M_2 \subseteq \B(H_2, K_2)$, and fix the column contractions
\[
(\underline{m}_1^\mu)_\mu \subseteq C(\M_1)
\qand
(\underline{m}_2^\la)_\la \subseteq C(\M_2),
\]
so that the nets with elements
\[
a_\mu := (\underline{m}_1^\mu)^* \cdot \underline{m}_1^\mu
\qand
b_\la := (\underline{m}_2^\la)^* \cdot \underline{m}_2^\la
\]
are contractive approximate identities for $[\M_1^* \M_1]$ and $[\M_2^* \M_2]$, respectively.
If $k_\mu$ (resp. $k_\la$) is the support of $(\underline{m}_1^\mu)_\mu$ (resp. $(\underline{m}_2^\la)_\la$) 
then the maps 
\[
\phi_{\la, \mu} \colon \B(H_1, H_2) \to M_{k_\la, k_\mu}(\B(K_1, K_2)) ; \ x \mapsto \underline{m}_2^\la \cdot x \cdot (\underline{m}_1^\mu)^*
\]
and
\[
\psi_{\la, \mu} \colon M_{k_\la, k_\mu}(\B(K_1, K_2)) \to \B(H_1, H_2) ; \ [y_{i,j}] \mapsto (\underline{m}_2^\la)^* \cdot [y_{i,j}] \cdot \underline{m}_1^\mu
\]
are completely contractive, and satisfy
\[
(\psi_{\la, \mu} \circ \phi_{\la, \mu})(x) =  b_\la \cdot x \cdot a_\mu \ \foral x \in \X.
\]
Therefore $\left((\psi_{\la, \mu} \circ \phi_{\la, \mu})|_\X\right)_{\la,\mu}$ 
converges norm-pointwise to $\id_\X$, that is, the diagram
\[
\xymatrix@C=1.5cm{
\X \ar[dr]_{\phi_{\la, \mu}} \ar[rr]^{\id_\X} & & \X \\
& M_{k_{\la}, k_{\mu}}(\Y) \ar[ur]_{\psi_{\la, \mu}} &
}
\]
approximately commutes. 
If $\X$ and $\Y$ are unital then we may assume that the nets are sequences and that the row contractions are supported on the same entries $k_n$, for every $n \in \bN$.
If $\X$ is an $\A$-$\B$-bimodule over C*-algebras $\A \subseteq [\M_1^* \M_1]$ and $\B \subseteq [\M_2^* \M_2]$ then this scheme applies to $\id_\A$ and $\id_\B$, so that the triple $(\id_{\cl B},\id_{\cl X},\id_{\cl A})$ is approximated 
in the same fashion. 
\end{remark}

\section{$\Delta$-equivalence and its characterisations}\label{s_Delta}

\subsection{TRO-equivalence}

We start by considering concrete operator systems (that may not be complete).
We first define TRO-equivalence in this context and make some preparatory remarks.
We fix Hilbert spaces $H$ and $K$. 

\begin{definition}\label{D:TRO}
Two concrete operator systems $\cl S\subseteq \cl B(H)$ and $\cl T\subseteq \cl B(K)$ are called \emph{TRO-equivalent} (denoted $\cl S\sim_{\rm TRO}\cl T$), if there exists a non-degenerate TRO $\cl M\subseteq \cl B(H,K)$ such that 
\begin{equation}\label{eq_TROeq}
\cl M^*\cl T \cl M \subseteq \ol{\cl S}
\qand 
\cl M \cl S \cl M^* \subseteq \ol{\cl T}.
\end{equation}
\end{definition}

If (\ref{eq_TROeq}) holds true, we say that $\cl S$ and $\cl T$ are TRO-equivalent \emph{via} $\cl M$.
It is clear that, in this case, $\M^* \M \subseteq \ol{\S}$ and $\M \M^* \subseteq \ol{\T}$.
Furthermore,
\[
\M \M^* \T \M \M^* \subseteq \M \ol{\S} \M^* \subseteq \ol{\T},
\]
and likewise for $\S$.
Since $\M$ is non-degenerate we have that $1_\S \in [\M^* \M]$ and $1_\T \in [\M \M^*]$, and therefore we obtain the bimodule conditions
\[
\M \M^* \T \cup \T \M \M^* \subseteq \ol{\T} \qand \M^* \M \S \cup \S \M^* \M \subseteq \ol{\S}.
\]

\begin{remark} \label{R:unit}
The requirement that $[\M^* \M]$ and $[\M \M^*]$ be unital is a necessary part of TRO-equivalence.
This is apparent, in particular, when the operator systems are complete.
In general, one can start with a TRO $\M$ so that $[\M^* \M]$ and $[\M \M^*]$ act non-degenerately on the complete operator systems.
Then \cite[Lemma 4.9]{EK17} implies that $[\M^* \M]$ and $[\M \M^*]$ have to be necessarily unital.
\end{remark}

The following proposition allows a reduction to complete operator systems.

\begin{proposition}\label{P:complete}
The concrete operator systems $\S\subseteq \cl B(H)$ and $\T \subseteq \cl B(K)$ are TRO-equivalent if and only if there exists a TRO $\M \subseteq \B(H, K)$ such that $[\M^* \T \M] = \ol{\S}$ and $[\M \S \M^*] = \ol{\T}$.
\end{proposition}

\begin{proof}
Suppose that ${\S}$ and ${\T}$ are TRO equivalent via a non-degenerate $\M$; thus, $1_\S = I_H \in [\M^* \M]$ and $1_\T = I_K \in [\M \M^*]$.
The inclusions
\[
\S \subseteq [\M^* \M \S \M^* \M] \subseteq [\M^* \T \M] \subseteq \ol{\S},
\qand
\T \subseteq [\M \M^* \T \M \M^*] \subseteq [\M \S \M^*] \subseteq \ol{\T},
\]
complete one direction.
Conversely, suppose that $\M \subseteq \B(H, K)$ is a TRO such that 
\[
[\M^* \T \M] = \ol{\S} \qand [\M \S \M^*] = \ol{\T}.
\]
We will show that $\M$ is non-degenerate.
To this end, let $(\un{m}^\mu)_{\mu} \subseteq C(\M)$ be a net so that the net $(a_\mu)_\mu \subseteq [\M^* \M]$ with elements $a_\mu := (\un{m}^\mu)^* \un{m}^\mu$ is a contractive approximate identity of the C*-algebra $[\M^* \M]$.
By definition, 
\[
\lim_\mu a_\mu \cdot s = s  \foral s \in \M^* \T \M
\]
in the norm topology, and so by passing to the norm closure this holds for all $s \in \ol{\S} = [\M^* \T \M]$.
In particular, this holds for $s = 1_\S$, showing that $1_\S \in [\M^* \M]$.
Likewise, $1_\T \in [\M \M^*]$ and the proof is complete.
\end{proof}

The condition that $\M$ is a non-degenerate TRO can be further relaxed by using instead non-degenerate operator spaces.

\begin{proposition}\label{P:tro os}
Let $\S \subseteq \B(H)$ and $\T \subseteq \B(K)$ be operator systems and $\X \subseteq \B(H, K)$ be a non-degenerate operator space such that
\[
\X^* \T \X \subseteq \ol{\S}
\qand
\X \S \X^* \subseteq \ol{\T}.
\] 
Let $\A := \ca(\X^* \X)$ and $\M := [\X \A]$. 
Then $\M$ is a TRO, and $\S$ and $\T$ are TRO-equivalent via $\cl M$. 
\end{proposition}

\begin{proof}
Let $\P = [\X^*\X]$ and $\Q = [\X\X^*]$. 
Since $\X$ is assumed non-degenerate, $\P$ and $\Q$ contain the corresponding identities, and thus they are operator systems.
Moreover, 
\[
\P \S \P \subseteq [\X^*\X \S \X^*\X] \subseteq [\X^* \T \X] \subseteq \ol{\S};
\]
similarly, $\Q \T \Q \subseteq \T$. 
We have that $\P^2\subseteq \ol{\S}$, and that $\S \P \subseteq \ol{\S}$, since $\P$ is unital. 
Assuming $\P^n\subseteq \ol{\S}$ for some $n$, we have $\P^{n+1} \subseteq \ol{\S} \P \subseteq \ol{\S}$.
Noting that $\A = \ca(\P)$, we have 
$\S \cdot \A \subseteq \ol{\S}$ and, by taking adjoints, we conclude that 
$\ol{\S}$ is an $\cl A$-bimodule. 
Letting $\M = [\X\A]$, we have 
\[
\M\M^* \M \subseteq [\X\A \X^* \X\A]
\subseteq [\X\A \P \A] \subseteq [\X \A] = \M,
\]
that is, $\M$ is a (closed) TRO. 
Since $\X$ is non-degenerate and $\A$ is unital, $\M$ is non-degenerate. 
In addition, 
\[
\M^*\T \M \subseteq [\A \X^*\T \X \A] \subseteq [\A \S \A] \subseteq \ol{\S},
\]
and, similarly, 
\[
\M \S \M^* \subseteq [\X \A\S \A \X^*] \subseteq [\X \S \X^*] \subseteq \ol{\T}.
\]
Thus, $\S\sim_{\rm TRO}\T$ via $\M$ and the proof is complete.
\end{proof}

\begin{proposition}\label{P:transitive}
TRO-equivalence of (concrete) operator systems is an equivalence relation.
\end{proposition}

\begin{proof}
It suffices to show transitivity. 
Let $\S \subseteq \B(H)$, $\T \subseteq \B(K)$ and $\R \subseteq \B(L)$ be such that $\S \sim_{\rm TRO} \T$ via $\M_1$ and $\T \sim_{\rm TRO} \R$ via $\M_2$.
Let $\D$ be the $*$-algebra generated by $\M_1 \M_1^* \cup \M_2^* \M_2$ and  $\M_3 := [\M_2 \D \M_1]$.
We have
\[
\M_3 \M_3^* \M_3 \subseteq \M_2 \D \M_1 \M_1^* \D \M_2^* \M_2 \D \M_1 \subseteq \M_2 \D \M_1 \subseteq \M_3,
\]
that is, $\cl M_3$ is a TRO. 
Moreover,
\[
\M_1 \M_1^* \T \cup \T \M_1 \M_1^* \subseteq \ol{\T}
\qand
\M_2^* \M_2 \T \cup \T \M_2^* \M_2 \subseteq \ol{\T},
\]
and so $\ol{\T}$ is a $\D$-bimodule.
Therefore
\[
\M_3^* \R \M_3 \subseteq \M_1^* \D \M_2^* \R \M_2 \D \M_1 \subseteq \M_1^* \D \ol{\T} \D \M_1 \subseteq \M_1^* \ol{\T} \M_1 \subseteq \ol{\S};
\]
similarly, $\M_3 \S \M_3^* \subseteq \ol{\R}$, and the proof is complete.
\end{proof}

\subsection{$\Delta$-equivalence}

We now proceed to a notion of equivalence that does not depend on the concrete representation of the operator systems.

\begin{definition}\label{D:delta}
Two operator systems $\cl S$ and $\cl T$ are called \emph{$\Delta$-equivalent} (denoted $\cl S\sim_{\Delta}\cl T$) if there exist Hilbert spaces $H$ and $K$ and unital complete order embeddings $\phi \colon \cl S\to \cl B(H)$ and $\psi \colon \cl T \to \cl B(K)$ such that $\phi(\cl S)\sim_{\rm TRO} \psi(\cl T)$. 
\end{definition}

\begin{proposition}\label{p_coarser}
$\Delta$-equivalence of operator systems is an equivalence relation, coarser than unital complete order isomorphism.
\end{proposition}

\begin{proof}
The proof follows the lines of \cite[Theorem 3.11]{EK17}, the necessary changes requiring only the fact that TRO-equivalence is an equivalence relation.
If $\phi \colon \S \to \T$ is a unital complete order isomorphism, then $\phi(\S) \sim_{\rm TRO} \T$ via $\M = \bC$, and so $\S \sim_\Delta \T$.
\end{proof}

\begin{theorem}\label{T:stable} 
Let $\cl S$ and $\cl \T$ be operator systems. 
The following are equivalent:
\begin{enumerate}
\item $\cl S \sim _{\Delta} \cl T$ as operator systems;
\item $\S \otimes \fK \simeq \T \otimes \fK$ via a completely positive completely isometric isomorphism;
\item $\S \otimes \fK \simeq \T \otimes \fK$ via a hermitian completely isometric isomorphism.
\end{enumerate} 
\end{theorem}

\begin{proof}
By Proposition \ref{P:complete}, we may assume that $\S$ and $\T$ are complete operator systems.
Item (i) implies item (ii) as in the proof of \cite[Corollary 4.10]{EK17} as the spaces are unital and the constructed isomorphisms in this case are completely positive and completely isometric. 
Assume item (iii) holds and suppose that $\S \subseteq \B(H)$ and $\T \subseteq \B(K)$.
Then $\S \sim_{\rm TRO} \S \otimes \fK$ as operator spaces (and systems) via the TRO $\M := R(\bC)$.
Thus, we have TRO-equivalences of operator spaces 
\[
\S \sim_{\rm TRO} \S \otimes \fK \simeq \T \otimes \fK \sim_{\rm TRO} \T
\]
via a single TRO each. 
Let $\X = \S \otimes \fK$ and consider $\X \subseteq \I(\X)$.
By \cite[Lemma 3.10]{EK17}, there exists a completely isometric map $\phi$ of $\S$ and a TRO $\M_1$ such that
\[
[\M_1 \phi(\S) \M_1^*] = \X \qand [\M_1^* \X \M_1] = \phi(\S).
\]
Hence $\phi(\S)$ and $\X$ are TRO equivalent as operator systems.
Similarly, there exists a completely isometric map $\psi$ defined on $\T$ and a TRO $\M_2$ such that
\[
[\M_2 \psi(\T) \M_2^*] = \X \qand [\M_2^* \X \M_2] = \psi(\T);
\]
thus, $\psi(\T)$ and $\X$ are TRO equivalent as operator systems.
Since $\cl X \simeq \T \otimes \fK$, Propositions \ref{P:transitive} and \ref{p_coarser} imply that $\phi(\S) \sim_{\rm TRO} \psi(\T)$, and therefore $\S \sim_{\Delta} \T$ as operator systems.
\end{proof}

We can choose representations for $\Delta$-equivalence that in addition respect C*-envelopes.
The next proposition is an operator system version of \cite[Theorem 5.10]{EK17}.

\begin{proposition}\label{P:reduction}
Let $\cl S$ and $\cl T$ be operator systems. 
If $\S \sim_{\Delta} \T$ then there exist 
complete order embeddings $\phi \colon \S \to \B(H)$ and $\psi \colon \T \to \B(K)$, 
and a non-degenerate TRO $\M \subseteq \B(H, K)$, such that $\phi(\S) \sim_{\rm TRO} \psi(\T)$ via $\M$, and in addition
\[
\phi(\A_\S) = [\M^* \M] \qand \psi(\A_\T) = [\M \M^*].
\]
Furthermore, 
\[
\cenv(\S) \simeq \ca(\phi(\S)) \qand \cenv(\T) \simeq \ca(\psi(\T)),
\]
and consequently $\cenv(\S) \sim_{\Delta} \cenv(\T)$.
\end{proposition}

\begin{proof}
By Proposition \ref{P:complete}, we may assume, without loss of generality, that $\S$ and $\T$ are complete.
Let $\phi \colon \S \to \I(\S)$ be the unital completely isometric embedding in the injective envelope $\cl I(\cl S)$ of $\S$
and note that $\cenv(\S) \simeq \ca(\phi(\S))$. 
By \cite[Theorem 5.10]{EK17}, there exists a unital complete isometry $\psi \colon \T \to \B(K)$ and a closed TRO $\M$ such that $\phi(\S) \sim_{\rm TRO} \psi(\T)$ via $\M$.
It is moreover shown in \cite{EK17} that $\cenv(\T) \simeq \ca(\psi(\T))$, and that 
\[
[\M^* \ca(\psi(\T)) \M] = \ca(\phi(\S)) \qand [\M \ca(\phi(\S)) \M^*] = \ca(\psi(\T)).
\]
By abuse of notation, we write $\phi(\A_\S)$ and $\psi(\A_\T)$ for the C*-algebras with the property
\[
\A_\S \simeq \phi(\A_\S) \subseteq \ca(\phi(\S))
\qand
\A_\T \simeq \psi(\A_\T) \subseteq \ca(\psi(\T)).
\]
A direct verification shows that  
\[
[\M^* \psi(\A_\T) \M] \subseteq \phi(\A_\S) \qand [\M \phi(\A_\S) \M^*] \subseteq \psi(\A_\T).
\]
Since $[\M^* \M]$ and $[\M \M^*]$ are unital, applying these inclusions successively promotes them to equalities.
The subspace $\L := [\psi(\A_\T) \M \phi(\A_\S)]$ is a (closed) TRO since
\begin{align*}
\L \L^* \L 
& \subseteq [\psi(\A_\T) \cdot \M \phi(\A_\S) \M^* \cdot \psi(\A_\T) \M \phi(\A_\S)] 
= 
[\psi(\A_\T) \cdot \psi(\A_\T) \cdot \psi(\A_\T) \M \phi(\A_\S)] = \L.
\end{align*}
Using a similar argument, we obtain the required
\[
\phi(\A_\S) = [\L^* \L]
\qand
\psi(\A_\T) = [\L \L^*],
\]
and hence the fact that $\phi(\S) \sim_{\rm TRO} \psi(\T)$ via $\L$.
\end{proof}

Of particular interest are the operator systems $\cl S$ with $\cl A_{\cl S} = \bb{C} \cdot 1_S$.
Following \cite{EPT10}, we call such operator systems \emph{rigid}.
The $\Delta$-equivalence for this class coincides with complete order isomorphism.
More generally we have the following corollary.

\begin{corollary} \label{C:rigid}
Let $\S \subseteq \B(H)$ and $\T \subseteq \B(K)$ be operator systems.
If $\S$ is rigid, then $\S \sim_{\Delta} \T$ if and only if there exists $k\in \bb{N}$ such that $\ol{\T} \simeq_{\rm c.o.i.} M_k(\ol{\S})$.
Moreover, if $\S \sim_{\Delta} \T$ then $\A_\T \simeq M_k$.
\end{corollary}

\begin{proof}
Without loss of generality, assume that $\cl S$ and $\cl T$ are complete.
Suppose that $\S \sim_{\Delta} \T$.
By Proposition \ref{P:reduction}, there exist unital complete order embeddings $\phi \colon \cl S\to \cl B(H)$ and $\psi \colon\cl T\to \cl B(K)$ such that 
\[
\ca(\phi(\S)) \simeq \cenv(S) \qand \ca(\psi(\T)) \simeq \cenv(\T),
\]
and $\phi(\S) \sim_{\rm TRO} \psi(\T)$ via a TRO $\M$ with
\begin{equation}\label{eq_asatm}
[\M^* \M] \simeq \A_\S \qand [\M \M^*] \simeq \A_\T.
\end{equation}
Since $\phi$ and $\psi$ are unital complete order embeddings, we suppress notation and replace $\phi(\S)$ by $\S$ and $\psi(\T)$ by $\T$.
By the identification on the left in (\ref{eq_asatm}), $\cl M$ is a right Hilbert module over $\bC$, that is, a Hilbert space.
On the other hand, for the C*-algebra $\K(\M)$ of compact operators on $\M$ we have
\[
\K(\M) \simeq [\M \M^*] \simeq \A_\T.
\]
Hence $\K(\M)$ is unital, and thus $\M$ is finite dimensional.
From this we derive that $\A_\T \simeq M_k$ for $k = \dim \M$.

Let $\{m_j\}_{j=1}^k$ be a family with 
\[
I_K = \sum_{i=1}^k m_i m_i^*
\qand
m_i^* m_j= \de_{i,j} I_H, \ i,j = 1,\dots,k.
\]
Define 
\[
\mathcal Y := [\mathcal M \mathcal S] = [\mathcal T \mathcal M],
\]
and let the completely contractive maps
\[
\phi \colon \mathcal T \rightarrow C_k(\mathcal Y); \ \phi(t) = \begin{bmatrix} t m_1 \\ \vdots \\ t m_k \end{bmatrix}
\qand
\psi \colon C_k(\mathcal Y)\rightarrow \mathcal T; \ \psi( \begin{bmatrix} y_1 \\ \vdots \\ y_k \end{bmatrix} ) = \sum_{i=1}^k y_i m_i^*.
\]
Since
\[
(\psi \circ \phi)(t) = \sum_{i=1}^k t m_i m_i^*= t, \ \ \ t\in \cl T, 
\]
we have that $\phi$ is a complete isometry.
For $(y_1, \dots, y_k)^t\in  C_k(\mathcal Y)$, let $t := \sum_{i=1}^k y_i m_i^* \in \T$, and observe that, 
since $m_i^* m_j = \de_{i,j} I_H$, we moreover have
\[
\phi(t) = \left(\sum_{i=1}^k y_i m_i^* m_1, \dots, \sum_{i=1}^k y_i m_i^* m_k\right)^t = (y_1, \dots, y_k)^t.
\] 
Therefore $\phi$ is onto $C_k(\mathcal Y)$, and thus $\mathcal T \simeq C_k(\mathcal Y)$. 
Similar arguments give that $\mathcal Y \simeq R_k(\mathcal S)$, and hence 
\[
\mathcal T \simeq C_k(R_k(\mathcal S))\simeq M_k(\mathcal S).
\]

Conversely suppose that $\T \simeq M_k(\S)$, and so $\T \sim_{\Delta} M_k(\S)$.
On the other hand, recall that $\S \sim_{\Delta} M_k(\S)$ by the TRO $\M := C_k$, since
\[
[C_k \cdot \S \cdot R_k] = M_k(\S)
\qand
[R_k \cdot M_k(\S) \cdot C_k] = \S.
\]
Due to the transitivity of $\Delta$-equivalence we get that $\T \sim_{\Delta} \S$, and the proof is complete.
\end{proof}

\begin{example}
Corollary \ref{C:rigid} covers the operator systems considered in \cite{CS20}.
The operator system $C(S^1)^{(n)} \subseteq M_n$ therein consists of $n \times n$ complex-valued Toeplitz matrices, that is, matrices of the form
\[
T :=
\begin{bmatrix}
t_0 & t_{-1} & \cdots & t_{-n+2} & t_{-n+1} \\
t_1 & t_0 & t_{-1} & & t_{-n+2} \\
\vdots & t_1 & t_0 & \ddots & \vdots \\
t_{n-2} & & \ddots & \ddots & t_{-1} \\
t_{n-1} & t_{n-2} & \cdots & t_1 & t_0
\end{bmatrix}.
\]
These operator systems are rigid and finite-dimensional, by definition.
Thus, if $\T \sim_{\Delta} C(S^{(1)})^{(n)}$, then $\T$ is finite dimensional with $\A_\T \simeq M_k$, and $\T \simeq M_k(C(S^{(1)})^{(n)})$.
\end{example}

\subsection{Morita contexts for operator systems}

In this subsection, we provide an abstract characterisation of $\Delta$-equivalence through appropriate versions of Morita contexts, in the spirit of \cite{Bas62, BMP00, Rie74jpaa}.
In what follows, a trilinear map $\lambda \colon V_1\times V_2\times V_3 \to V$ is lifted to a trilinear map (denoted in the same way)
\[
\lambda \colon M_{k_1,k_2}(V_1) \times M_{k_2,k_3}(V_2)\times M_{k_3,k_{4}}(V_r) \to M_{k_1,k_{4}}(V),
\]
by letting 
\[
\lambda\left((v_{i_1,i_2}^{(1)}),(v_{i_2,i_3}^{(2)}),(v_{i_3,i_{4}}^{(r)})\right) 
= \left(\sum_{i_2 = 1}^{k_2} \sum_{i_3 =1}^{k_3} \lambda\left(v_{i_1,i_2}^{(1)},v_{i_2,i_3}^{(2)},v_{i_3,i_{4}}^{(r)}\right)\right).
\]
If $\cl A_i$ are algebras, $i = 1,2,3,4$, and $V_i$ is an $\cl A_i,\cl A_{i+1}$-bimodule, $i = 1,2,3$, while $V$ is an $\cl A_1,\cl A_4$-module, we say that $\lambda$ is a module map over $\cl A_1,\cl A_2, \cl A_3$ and $\cl A_4$ if 
\[
\lambda(a_1\cdot v_1\cdot a_2, v_2\cdot a_3, v_3\cdot a_4) 
= 
a_1\cdot \lambda(v_1, a_2 \cdot v_2, a_3 \cdot v_3)\cdot a_4,
\]
for all $v_i\in V_i$, $a_j\in \cl A_j$, $i = 1,2,3$, $j = 1,2,3,4$.

Let $\X$ be an operator space and $\S$ and $\T$ be operator systems.
We say that a trilinear map
\[
[ \cdot, \cdot, \cdot] \colon \X^* \times \T \times \X \longrightarrow \S
\]
is \emph{positive}, if 
\[
[a^*, \T^+, a] \subseteq \S^+ \foral  a \in \X.
\]
Recall that, in an operator system, the positive elements span the real vector space of all hermitian elements.
Hence positive trilinear maps are automatically compatible with taking adjoints in the sense that 
\[
[a^*, t, b]^* = [b^*, t^*, a] \ \foral a, b \in \X, t \in \T.
\]
We say that the trilinear map $[\cdot, \cdot, \cdot]$ is \emph{completely positive} if the induced trilinear map from $M_{n,m}(\cl X^*)\times M_m(\cl T) \times M_{m,n}(\cl X)$ into $M_n(\cl S)$ is positive for all $n,m\in \bb{N}$. 

\begin{definition}\label{d_DelM}
Let $\S$ and $\T$ be abstract operator systems and $\M$ be a TRO.
We say that the quintuple $\big( \S, \T, \M, [\cdot, \cdot, \cdot], (\cdot, \cdot, \cdot) \big)$ is a \emph{$\Delta$-pre-context} if:

\begin{itemize}
\item[(i)] the C*-algebras $[\M^* \M]$ and $[\M \M^*]$ are unital;

\item[(ii)] $\S$ is a C*-bimodule over $[\M^* \M]$ and 
$\T$ is a C*-bimodule over $[\M \M^*]$; 

\item[(iii)] 
$[\cdot,\cdot,\cdot] \colon \M^* \times \T \times \M \longrightarrow \S$ and 
$(\cdot,\cdot,\cdot) \colon \M \times \S \times \M^* \longrightarrow \T$
are completely bounded completely positive maps, modular over $[\M^*\M]$ and $[\M \M^*]$
on the outer variables (with unital module actions), and

\item[(iv)] the associativity relations
\[
(m_1, [m_2^*, t, m_3], m_4^*) = (m_1m_2^*) \cdot t \cdot (m_3 m_4^*)
\]
and
\[
[m_1^*, (m_2,  s, m_3^*), m_4] = (m_1^* m_2) \cdot s \cdot (m_3^* m_4)
\]
hold for all $s \in \S$, $t \in \T$ and all $m_1, m_2, m_3, m_4 \in \M$.
\end{itemize}

\vspace{2pt}

A $\Delta$-pre-context is called a \emph{$\Delta$-context} if the trilinear maps $[\cdot, \cdot, \cdot]$ and $(\cdot, \cdot, \cdot)$ are completely contractive and the relations
\begin{equation}\label{eq_move}
(m_1, 1_\S, m_2^*) = (m_1 m_2^*) \cdot 1_\T
\qand
[m_1^*, 1_\S, m_2] = (m_1^* m_2) \cdot 1_\S
\end{equation}
hold for all $m_1, m_2 \in \M$.
\end{definition}

There are some immediate consequences of the $\Delta$-context relations which we record below.

\begin{remark}\label{R:properties}
Let $\big( \S, \T, \M, [\cdot, \cdot, \cdot], (\cdot, \cdot, \cdot) \big)$ be a $\Delta$-context.
Then the following hold for all $s\in \cl S$, $t\in \cl T$ and all $m_i, n_i \in \cl M$, $i = 1,2,3$:
\begin{enumerate}
\item
$[m_1^*,(m_2,[m_3^*,t,n_3],n_2^*), n_1] = 
[m_1^*m_2 m_3^*, t, n_3 n_2^* n_1]$;

\vspace{4pt}

\item
$(m_1,[m_2^*,(m_3, s, n_3^*), n_2],n_1^*) = 
(m_1m_2^* m_3, s,n_3^*n_2n_1^*)$;

\vspace{4pt}

\item
$(m_1, 1_\S, m_2^*) = 1_\T \cdot m_1 m_2^*$;

\vspace{4pt}

\item
$[m_1^*, 1_\T, m_2] = 1_\S \cdot m_1^* m_2$;

\vspace{4pt}

\item
$[m_1^*,(m_2,1_{\cl S},n_2^*),n_1]
= [m_1^*,1_{\cl T}, m_2 n_2^* n_1]
= [m_1^*m_2 n_2^*,1_{\cl T}, n_1]$;

\vspace{4pt}

\item
$(m_1,[m_2^*,1_{\cl T},n_2], n_1^*) 
= (m_1, 1_{\cl S}, m_2^* n_2 n_1^*)
= (m_1 m_2^* n_2, 1_{\cl S}, n_1^*)$.
\end{enumerate}

\noindent 
First we show that 
\begin{equation}\label{eq_bimre}
[m_1^* a, t, b m_2] = [m_1^*, a \cdot t \cdot b, m_2] \ \mbox{ and } \ 
(m_1 c, s, d m_2^*) = (m_1, c \cdot s \cdot d, m_2^*)
\end{equation}
for all $s \in \S$, $t \in \T$, $a, b \in [\M^* \M]$, and $c, d \in [\M \M^*]$.
Indeed, for a contractive approximate identity $(\un{m}_i^* \un{n}_i)_i$ of $[\M^* \M]$ converging to $1_{[\M^* \M]}$ 
we have
\begin{align*}
(m_1, (m_2^* m_3) \cdot s, m_4^*)
& =
(m_1, (m_2^* m_3) \cdot s \cdot 1_{[\M^* \M]} , m_4^*) 
=
\lim_i (m_1, (m_2^* m_3) \cdot s \cdot (\un{m}_i^* \un{n}_i), m_4^*) \\
& =
\lim_i (m_1, [m_2^*, (m_3, s, \un{m}_i^*), \un{n}_i], m_4^*) 
=
\lim_i (m_1 m_2^*) \cdot (m_3, s, \un{m}_i^*) \cdot (\un{n}_i m_4^*) \\
& =
\lim_i (m_1 m_2^* m_3, s, \un{m}_i^* \un{n}_i m_4^*)
=
(m_1 (m_2^* m_3), s, m_4^*),
\end{align*}
as required; the other relations in (\ref{eq_bimre}) are shown similarly.
Items (i), and (ii) are consequences of (\ref{eq_bimre}).  
Items (iii) and (iv) follow after taking adjoints in (\ref{eq_move}), e.g.
\[
1_\T \cdot m_1 m_2^* = (m_2 m_1^* \cdot 1_\T)^* = (m_2, 1_\S, m_1^*)^* = (m_1, 1_\S, m_2^*).
\]
For item (v) we have by using the definition that
\begin{align*}
[m_1^*,(m_2,1_{\cl S},n_2^*),n_1]
& =
(m_1^* m_2) \cdot 1_{\cl S} \cdot (n_2^*n_1) \\
& =
[m_1^*, 1_{\S}, m_2] \cdot (n_2^* n_1) 
=
[m_1^*, 1_{\S}, m_2 n_2^* n_1].
\end{align*}
By taking adjoints we have the symmetrical relation in (v).
Item (vi) follows in a similar way.
\end{remark}

Motivated by Proposition \ref{P:tro os}, we next relax the conditions from Definition \ref{d_DelM}.
This would allow us to reveal a link with non-commutative graph theory \cite{BTW21, DSW13, Sta16}, 
made precise subsequently. 

\begin{definition}\label{D:hom}
Two concrete operator systems $\cl S\subseteq \cl B(H)$ and $\cl T \subseteq \cl B(K)$ are called 
\emph{(concretely) bihomomorphically equivalent} if there exists a (concrete) operator space $\X \subseteq \cl B(H,K) $ such that $\X$ and $\X^*$ are non-degenerate, and 
\[
\X^*\cl T \X \subseteq \ol{\cl S} \qand \X \cl S \X^*\subseteq \ol{\cl T}.
\]
Two abstract operator systems $\cl S$ and $\cl T$ will be called \emph{bihomomorphically equivalent} if there exist Hilbert spaces $H$ and $K$, and unital complete order embeddings $\phi \colon \cl S\to \cl B(H)$ and $\psi \colon \cl T\to \cl B(K)$ such that the concrete operator systems $\phi(\cl S)$ and $\psi(\cl T)$ are bihomomorphically equivalent.
We write $\cl S \leftrightarrows \cl T$ to denote that $\S$ and $\T$ are bihomomorphically equivalent.
\end{definition}

Note that Proposition \ref{P:tro os} asserts that two concrete operator systems $\cl S$ and $\cl T$
are bihomomorphically equivalent if and only if $\cl S \sim_{\rm TRO} \cl T$.

\begin{definition}\label{d_DelMho}
Let $\S$ and $\T$ be abstract operator systems and $\X$ be an abstract operator space.
We say that the quintuple $\big( \S, \T, \X, [\cdot, \cdot, \cdot], (\cdot, \cdot, \cdot) \big)$ is a 
\emph{bihomomorphism pre-context} if:

\begin{itemize}
\item[(i)] $\X$ is non-degenerate;

\item[(ii)] 
$[\cdot,\cdot,\cdot] \colon \X^* \times \T \times \X \longrightarrow \S$ and 
$(\cdot,\cdot,\cdot) \colon \X \times \S \times \X^* \longrightarrow \T$
are completely bounded completely positive maps such that
\[
[\X^*, 1_\T, \X] \subseteq \A_{\S} \qand (\X, 1_\S, \X^*) \subseteq \A_{\T};
\]

\item[(iii)] the associativity relations
\[
[x_1^*, (x_2, s, x_3^*), x_4] = [x_1^*, 1_\T, x_2] \cdot s \cdot [x_3^*, 1_\T, x_4]
\]
and
\[
(x_1, [x_2^*, t, x_3], x_4^*) = (x_1, 1_\S, x_2^*) \cdot t \cdot (x_3, 1_\S, x_4^*)
\]
hold for all $s \in\S, t \in \T$ and all $x_1, x_2, x_3, x_4 \in \X$.
\end{itemize}

A bihomomorphism pre-context is called a \emph{bihomomorphism context} if the trilinear maps $[\cdot,\cdot,\cdot]$ and $(\cdot,\cdot,\cdot)$ are completely contractive and 
there exist semi-units $((\un{x}_i)_i, (\un{y}_i)_i)$ and $((\un{z}_i)_i, (\un{w}_i)_i)$ over $\cl X$ and $\cl X^*$, respectively, such that 
\begin{equation}\label{eq_adjxi}
\lim_i [\un{x}_i^*, 1_\T \otimes I, \un{y}_i] =  1_\S
\ \ \mbox{ and } \ \ 
\lim_i (\un{z}_i, 1_\S \otimes I, \un{w}_i^*) = 1_\T.
\end{equation}
\end{definition}

Recall that by definition the semi-units involved in the biholomorphism context are nets of finitely supported contractive columns; see Definition \ref{D:semiunit}.
Note, in addition, that conditions (\ref{eq_adjxi}) automatically imply
$\lim_i [\un{y}_i^*, 1_\T \otimes I, \un{x}_i] =  1_\S$
and
$\lim_i (\un{w}_i, 1_\S \otimes I, \un{z}_i^*) = 1_\T$, by the positivity of the trilinear maps. 

\begin{remark}\label{R:approx}
Let $\big( \S, \T, \X, [\cdot, \cdot, \cdot], (\cdot, \cdot, \cdot) \big)$ be a bihomomorphism context, and let $((\un{x}_i)_i, (\un{y}_i)_i)$ and $((\un{z}_i)_i, (\un{w}_i)_i)$ be semi-units for $\X$ and $\X^*$, respectively. 
Since $1_\T = 1_{\A_{\T}}$, we have
\[
\lim_i (\un{z}_i, [\un{w}_i^*, t, \un{z}_i], \un{w}_i^*) = \lim_i (\un{z}_i, 1_\S \otimes I, \un{w}_i^*) \cdot t \cdot (\un{z}_i, 1_\T \otimes I, \un{w}_i^*) = t
\]
for all $t \in \T$, and a similar relation holds for $( (\un{x}_i)_i, (\un{y}_i)_i )$ and elements $s \in \S$.
\end{remark}

\begin{theorem}\label{th_ac} 
Let $\cl S$ and $\cl T$ be (abstract) operator systems. 
The following are equivalent:
\begin{enumerate}
\item $\cl S\sim_{\Delta}\cl T$;
\item $\cl S \leftrightarrows \cl T$;
\item there exists a $\Delta$-context for $\cl S$ and $\cl T$;
\item there exists a bihomomorphism context for $\cl S$ and $\cl T$.
\end{enumerate}
\end{theorem}

\begin{proof}
Without loss of generality, we may assume that $\S$ and $\T$ are complete operator systems.
Indeed, for items (i) and (ii) this is derived from Proposition \ref{P:complete}.
For items (iii) and (iv) we note that if $\S$ and $\T$ admit a $\Delta$- or a bihomomorphic context then so do their completions by extension of the continuous trilinear maps.

\smallskip

The equivalence [(i)$\Leftrightarrow$(ii)] is shown in Proposition \ref{P:tro os}.

\smallskip

For the implication [(iv)$\Rightarrow$(ii)] we adapt a well-known Morita equivalence machinery.
Fix a unital complete order embedding $\phi \colon \S \to \I(\S) \subseteq \B(H)$ in the injective envelope of $\S$.
Let $((\un{x}_i)_i, (\un{y}_i)_i)$ and $((\un{z}_i)_i, (\un{w}_i)_i)$ be semi-units for $\X$ and $\X^*$, respectively, each of whose elements have finite support of length $n_i$. 
Equip the algebraic tensor product $\X \odot H$ with the sesquilinear form given by 
\[
\sca{x_1 \otimes h_1, x_2 \otimes h_2} := \sca{\phi( [x_2^*, 1_\T, x_1] ) h_1, h_2}_H.
\]
We write $\X \wt{\otimes} H$ for the Hausdorff completion of $\X \odot H$, 
obtained by passing to the quotient with respect to the kernel of $\sca{\cdot, \cdot}$ and completing, 
and set $K := \X \wt{\otimes} H$ for brevity and suppress the notation for the quotient map. 
The form $\sca{\cdot, \cdot}$ yields an inner product on $K$, which will be denoted by $\sca{\cdot, \cdot}_K$. 
For $x \in \X$, let $\theta(x) \colon H\to K$ be the map given by $\theta(x)(h) = x \otimes h$, and note that 
\begin{align}\label{eq_nthet}
\left\| \left(\theta(x_{i,j})\right)_{i,j} \begin{bmatrix} h_1 \\ \vdots \\ h_n \end{bmatrix} \right\|_{K^{(n)}}^2
& =
\left\| \left(\sum_{j=1}^n x_{i,j} \otimes h_j\right)_{i=1}^n\right\|_{K^{(n)}}^2  \\\nonumber
& =
\left\langle \phi^{(n)}\left( \left( \sum_{k=1}^n [x_{k,i}^*, 1_\T, x_{k,j}] \right)_{i,j} \right) \begin{bmatrix} h_1 \\ \vdots \\ h_n \end{bmatrix}, \begin{bmatrix} h_1 \\ \vdots \\ h_n \end{bmatrix} \right\rangle_{H^{(n)}} \\ \nonumber
& = 
\left\| \phi^{(n)}\left( \left(\sum_{k=1}^n [x_{k,i}^*, 1_\T, x_{k,j}] \right)_{i,j} \right)^{1/2} \begin{bmatrix} h_1 \\ \vdots \\ h_n \end{bmatrix} \right\|_{H^{(n)}}^2. \nonumber
\end{align}
Since $\phi$ is completely isometric and the trilinear maps are completely contractive,
\begin{align*}
\left\| \left(\theta(x_{i,j}) \right)_{i,j}\right \|^2
& =
\left\| \phi^{(n)}\left( \left( \sum_{k=1}^n [x_{k,i}^*, 1_\T, x_{k,j}] \right)_{i,j}\right)^{1/2} \right\|^2\\
& \leq
\left\| \left(\sum_{k=1}^n [x_{k,i}^*, 1_\T, x_{k,j}]\right)_{i,j} \right\| 
\leq
\left\| \left(x_{i,j}\right)_{i,j}\right\|^2.
\end{align*}
Therefore $\theta$ defines a completely contractive map
$\theta \colon \X \rightarrow \B(H,K)$,
and a straightforward calculation shows that
\begin{equation}\label{eq_xst}
\theta(x)^* (x' \otimes h') = \phi( [x^*, 1_\T, x'] ) h' \ \foral x,x'\in \cl X, h'\in H.
\end{equation}

Next we define a unital completely positive map $\psi \colon \T \to \B(K)$ such that
\begin{equation}\label{eq_psi(t)}
\sca{\psi(t) (x_1 \otimes h_1), x_2 \otimes h_2}_K = \sca{ \phi( [x_2^*, t, x_1] ) h_1, h_2}_H.
\end{equation}
Towards this end first we consider a positive element $t \in \T$ and define the sesquilinear form
\[
\sca{\cdot, \cdot}_t \colon \X \odot H \times \X \odot H \to \bC
\]
given by
\[
\sca{x_1 \otimes h_1, x_2 \otimes h_2}_t := \sca{ \phi( [x_2^*, t, x_1] ) h_1, h_2}_H.
\]
Since $t \in \T^+$ it follows that $\sca{\cdot, \cdot}_t$ is positive semidefinite.
Moreover, since $0 \leq t \leq \|t\| 1$ and the trilinear map $[\cdot,\cdot,\cdot]$ is completely positive, 
for $u = \sum_{k=1}^n x_k \otimes h_k \in \X \odot H$ we have
\begin{align*}
\sca{u, u}_t
& =
\left\langle \phi^{(n)} \bigg( \big[\begin{bmatrix} x_1^* & \cdots & 0 \\ \vdots & & \vdots \\ x_n^* & \cdots & 0 \end{bmatrix}, t \otimes I_n, \begin{bmatrix} x_1 & \cdots & x_n \\ \vdots & & \vdots \\ 0 & \cdots & 0 \end{bmatrix} \big] \bigg) 
\begin{bmatrix} h_1 \\ \vdots \\ h_n \end{bmatrix}, \begin{bmatrix} h_1 \\ \vdots \\ h_n \end{bmatrix} \right\rangle_{H^{(n)}} \\
& \leq 
\|t\|
\left\langle \phi^{(n)} \bigg( \big[\begin{bmatrix} x_1^* & \cdots & 0 \\ \vdots & & \vdots \\ x_n^* & \cdots & 0 \end{bmatrix}, 1_\T \otimes I_n, \begin{bmatrix} x_1 & \cdots & x_n \\ \vdots & & \vdots \\ 0 & \cdots & 0 \end{bmatrix} \big] \bigg) 
\begin{bmatrix} h_1 \\ \vdots \\ h_n \end{bmatrix}, \begin{bmatrix} h_1 \\ \vdots \\ h_n \end{bmatrix} \right\rangle_{H^{(n)}} \\ 
& =
\|t\| \sca{u,u}.
\end{align*}
In particular if $\sca{u,u} = 0$ then $\sca{u,u}_t = 0$ and thus we obtain an induced sesquilinear form (denoted in the same way $\sca{\cdot, \cdot}_t$) by passing to quotients.
In particular this form satisfies
\[
\sca{u,u}_t \leq \|t\| \cdot \sca{u,u}_K \foral u \in K.
\]
By the Cauchy-Schwarz inequality for positive sesquilinear forms we get that 
\[
|\sca{u,v}_t| \leq \sca{u,u}^{1/2}_t  \sca{v,v}^{1/2}_t
\leq \|t\| \sca{u,u}_K^{1/2} \sca{v,v}_K^{1/2}
= \|t\| \|u\|_K \|v\|_K,
\]
and thus the Riesz representation Theorem induces a bounded linear map $\psi(t) \in \B(K)$.
Since the positive elements span $\T$, we derive a well-defined linear map $\psi(t)$, $t \in \T$,
satisfying (\ref{eq_psi(t)}); note that the map $\psi \colon \cl T\to \B(K)$ is linear and unital. 
Similar computations show that the map $\psi$ is completely contractive and completely positive.

Let 
$L = \X^* \wt{\otimes} K \equiv \X^* \wt{\otimes} (\X \wt{\otimes} H)$. 
Applying the construction from the previous paragraphs to $\psi$, we obtain a completely contractive map
$\theta' \colon \X^* \rightarrow \B(K,L)$
and a unital completely positive map $\phi' \colon \S \rightarrow \B(L)$.
Note that 
\begin{equation}\label{eq_phi'}
\sca{ \phi'(s) (x_1^* \otimes y_1 \otimes h_1), x_2^* \otimes y_2 \otimes h_2}_L
=
\sca{ \phi( [y_2^*, (x_2, s, x_1^*), y_1] ) h_1, h_2}_H,
\end{equation}
for all $x_i, y_i\in \cl X$, $s\in \cl S$, $h_i\in H$, $i = 1,2$. 
Using the fact that $\phi$ is a unital 
$\A_{\S}$-bimodule map, for $x_i, y_i\in \cl X$, $h_i\in H$, $i = 1,2$, we have 
\begin{align*}
\sca{x_1^* \otimes y_1 \otimes h_1, x_2^* \otimes y_2 \otimes h_2}_L
& =
\sca{ \psi((x_2, 1_\S, x_1^*))(y_1 \otimes h_1), y_2 \otimes h_2}_K \\
& =
\sca{ \phi([y_2^*, (x_2, 1_\T, x_1^*), y_1]) h_1, h_2 }_H \\
& =
\sca{ \phi( [y_2^*, 1_\T, x_2] \cdot 1_\S \cdot [x_1^*, 1_\T, y_1] ) h_1, h_2}_H \\
& =
\sca{ \phi( [x_1^*, 1_\T, y_1] ) h_1, \phi( [x_2^*, 1_\T, y_2] ) h_2 }_H.
\end{align*}
It follows that the map
\[
U_H \colon L \rightarrow H; \ x^* \otimes y \otimes h \mapsto \phi( [x^*, 1_\T, y] ) h,
\]
is well-defined and isometric.
Furthermore, $U_H$ is onto $H$; indeed for $h \in H$ we have that
\begin{equation}\label{eq_hlim}
h = \phi(1_\S) h = \lim_i \phi( [\un{x}_i^*, 1_\T \otimes I, \un{y}_i] )h 
= \lim_i \sum_{k=1}^{n_i} U_H(x_{i,k}^* \otimes y_{i,k} \otimes h).
\end{equation}
We will show that
\begin{equation}\label{eq_UH}
U_H \phi'(s) = \phi(s) U_H \foral s \in \S
\qand
U_H \theta'(x^*) = \theta(x)^* \foral x \in \X.
\end{equation}
Indeed, using (\ref{eq_phi'}) and (\ref{eq_hlim}), we have 
\begin{align*}
\sca{U_H \phi'(s) (x_1^* \otimes x_2 \otimes h_2), h_1}_H
& =
\lim_i \sum_{k=1}^{n_i} \sca{ \phi'(s) (x_1^* \otimes x_2 \otimes h_2), x_{i,k}^* \otimes y_{i,k} \otimes h_1}_L \\
& =
\lim_i \sca{ \phi( [\un{y}_i^*, (\un{x}_i, s, x_1^*), x_2] ) h_2, h_1}_H \\
& = 
\lim_i \sca{ \phi( [\un{y}_i^*, 1_\T \otimes I, \un{x}_i] \cdot s \cdot [x_1^*, 1_\T, x_2] ) h_2, h_1}_H \\
& = 
\sca{ \phi(1_{\A_{\S}} \cdot s) \phi( [x_1^*, 1_\T, x_2] ) h_2, h_1}_H \\
& = 
\sca{\phi(s) U_H (x_1^* \otimes x_2 \otimes h_2), h_1}_H.
\end{align*}
With regard to $\theta$ and $\theta'$, a direct observation using (\ref{eq_xst}) gives 
\[
U_H \theta'(x^*) (x' \otimes h) = U_H (x^* \otimes x' \otimes h) = \phi( [x^*, 1_\T, x'] ) h = \theta(x)^* (x' \otimes h).
\]

To complete the proof of this implication, we will show that $\psi$ is completely isometric, that $\theta(\X)$ is non-degenerate, and that $\phi(\S)$ and $\psi(\T)$ are bihomomorphically equivalent via $\theta(\X)$.
First note that 
\begin{align*}
\sca{\theta(x_2)^* \psi(t) \theta(x_1) h_1, h_2}_H
& =
\sca{\psi(t) (x_1 \otimes h_1), x_2 \otimes h_2}_K
=
\sca{ \phi( [x_2^*, t, x_1] ) h_1, h_2 }_H,
\end{align*}
implying 
\[
\theta(x_2)^* \psi(t) \theta(x_1) = \phi( [x_2^*, t, x_1] ) \foral x_1, x_2 \in \X, t \in \T;
\]
thus, $\theta(\X)^* \psi(\T) \theta(\X) \subseteq \phi(\S)$.
A similar computation applied to $\theta'$ and $\phi'$ shows that
\[
\theta'(\X^*) \phi'(\S) \theta'(\X) \subseteq \psi(\T),
\]
and therefore
\[
\theta(\X) \phi(\S) \theta(\X)^* 
= 
\theta'(\X^*)^* U_H^* U_H \phi'(\S) 
U_H^* U_H \theta'(\X^*)
=
\theta'(\X^*)^* \phi'(\S) \theta'(\X^*) \subseteq \psi(\T).
\]
In order to show that $\psi$ is completely isometric, we use Remark \ref{R:approx} for $\T$ and the fact that $\phi$ is completely isometric to obtain
\begin{align*}
\|t\|
& =
\lim_i \| (\un{z}_i, [\un{w}_i^*, t, \un{z}_i], \un{w}_i^*) \|
\leq
\limsup_i \| [\un{w}_i^*, t, \un{z}_i] \| \\
& =
\limsup_i \| \phi( [\un{w}_i^*, t, \un{z}_i] ) \|
=
\limsup_i \| \theta^{(n_i)}(\un{w}_i)^* \psi(t) \theta^{(n_i)}(\un{z}_i) \|
\leq
\| \psi(t) \|,
\end{align*}
where we have also used that the semi-units and the trilinear maps are contractive.
A similar argument works for all matrix norms.
Finally, for the non-degeneracy of $\theta(\X)$, we notice that 
\begin{align*}
\lim_i \theta^{(n_i)}(\un{x}_i)^* \theta^{(n_i)}(\un{y}_i) 
& = 
\lim_i \theta^{(n_i)}(\un{x}_i)^* \left(\psi(1_\T) \otimes I_{n_i}\right) \theta^{(n_i)}(\un{y}_i) \\
& = 
\lim_i \phi( [\un{x}_i^*, 1_\T \otimes I, \un{y}_i] ) = \phi(1_\S) = I_H,
\end{align*}
in norm, where we used the unitality of $\psi$ and $\phi$.

\smallskip

The implication [(iii)$\Rightarrow$(i)] follows in a similar way to the implication [(iv) $\Rightarrow$ (ii)].
More precisely, let $(H,\phi,\pi)$ be a C*-representation of the operator $[\M^* \M]$-system $\cl S$.
Since
\[
\phi( [m_1^*, 1_{\cl T}, m_2] ) = \phi( m_1^* m_2 \cdot 1_\T ) = \pi(m_1^* m_2) \phi(1_\T) = \pi(m_1^* m_2)
\]
whenever $m_1, m_2 \in \M$, 
the construction used to show the implication [(iv) $\Rightarrow$ (ii)] gives that
\[
\M \wt{\otimes} H \simeq \M \otimes_{[\M^* \M]} H.
\]
The proof then proceeds in the same way by recalling that, for a non-degenerate TRO $\M$,
we have $\M^* \otimes_{[\M \M^*]} \otimes \M \otimes_{[\M^* \M] } H \simeq H$.

\smallskip

For the implication [(i)$\Rightarrow$(iii)], let $\phi$ and $\psi$ be unital complete order embeddings as in Proposition \ref{P:reduction} such that the concrete operator systems $\phi(\S)$ and $\psi(\T)$ are TRO-equivalent via $\M$ with $\phi(\A_\S) = [\M^* \M]$ and $\psi(\A_\T) = [\M \M^*]$.
Moreover
\[
\ca(\phi(\S)) \simeq \cenv(\S) \qand \ca(\psi(\T)) \simeq \cenv(\T),
\]
and thus $\phi$ (resp. $\psi$) is an $\A_\S$-bimodule map (resp. $\A_\T$-bimodule map).
The maps $\phi$ and $\psi$ induce bimodule structure over $[\M^* \M]$ and $[\M \M^*]$ in the sense that
\[
m_1^* m_2 \cdot s := \phi^{-1}(m_1^* m_2) \cdot s
\qand
m_1 m_2^* \cdot t := \psi^{-1}(m_1 m_2^*) \cdot t,
\]
for all $m_1, m_2 \in \M$, $s \in \S$ and $t \in \T$, where the multiplication on the right hand sides is given by the multiplier algebras.
Set 
\[
[m_1^*, t, m_2] := \phi^{-1}( m_1^* \psi(t) m_2 )
\qand 
(m_3, s, m_4^*) := \psi^{-1}( m_3 \phi(s) m_4^*),
\]
where $s \in \S$, $t \in \T$ and $m_1, m_2, m_3, m_4 \in \X$; thus, the trilinear maps
$[\cdot,\cdot,\cdot] \colon \M^* \times \T \times \M \longrightarrow \S$ and 
$(\cdot,\cdot,\cdot) \colon \M \times \S \times \M^* \longrightarrow \T$ are 
completely bounded and completely positive, and moreover modular over $[\M^* \M]$ and $[\M \M^*]$ on the outer variables.
This shows item (iii) of Definition \ref{d_DelM}.
A routine computation gives item (iv) of Definition \ref{d_DelM}, and unitality of $\phi$ and $\psi$ gives that $\big( \S, \T, \M, [\cdot, \cdot, \cdot], (\cdot, \cdot, \cdot) \big)$ is a $\Delta$-context.
Indeed for the latter we have that
\[
(m_1, 1_\S, m_2^*) = \psi^{-1}( m_1 \phi(1_\S) m_2^* ) = \psi^{-1}(m_1 m_2^*) = \psi^{-1}(m_1 m_2^*) \cdot 1_\T = m_1 m_2^* \cdot 1_\T.
\]

\smallskip

The implication [(i)$\Rightarrow$(iv)] follows in the same way for the same trilinear maps given as above by the maps $\phi$ and $\psi$ of Proposition \ref{P:reduction}.
Indeed, unitality of $\phi$ and $\psi$ and that $\phi(\A_\S) = [\M^* \M]$ and $\psi(\A_\T) = [\M \M^*]$ yield item (ii) of Definition \ref{d_DelMho}, while item (iv) follows from the fact that $\phi$ and $\psi$ are bimodule maps. 
The unitality of $\phi$ and $\psi$ and non-degeneracy of $\M$ now imply that $\big( \S, \T, \M, [\cdot, \cdot, \cdot], (\cdot, \cdot, \cdot) \big)$ is a bihomomorphism context.
Note here that by construction the TRO $\M$ is non-degenerate as $[\M^* \M]$ and $[\M \M^*]$ are unital C*-algebras.
\end{proof}

\section{Categorical implications}\label{ss_ci}

In this section, we obtain consequences of $\Delta$-equivalence for categories of representations of operator systems. 
We consider two categories, which have the same objects but different morphisms. 
Let $\cl S$ be an operator system. 
A \emph{C*-representation} of $\cl S$ is a pair $(H,\phi)$, where $H$ is a Hilbert space and $\phi \colon \cl S\to \cl B(H)$ is a unital completely positive map whose restriction $\pi_{\phi}$ to $\cl A_{\cl S}$ is a *-homomorphism. 
By \cite[Paragraph 1.3.12]{blm}, combined with Arveson's Extension Theorem, 
\[
\phi(a\cdot s) = \pi_{\phi}(a)\phi(s) \foral s\in \cl S, a\in \cl A_{\cl S},
\]
that is, the triple $(H,\phi,\pi_{\phi})$ is a C*-representation of the operator $\cl A_{\cl S}$-system $\cl S$. 
When we want to emphasise the *-homomorphism $\pi = \pi_{\phi}$, we will write $(H,\phi,\pi)$ in the place of $(H,\phi)$. 
Two C*-representations $(H_1,\phi_1,\pi_1)$ and $(H_2,\phi_2,\pi_2)$ are called \emph{unitarily equivalent} if there exists a unitary operator $U \colon H_1\to H_2$ such that $U\phi_1(s) = \phi_2(s)U$, $s\in \cl S$. 
An object of the category ${\rm Rep}_{{\rm C^*}}(\cl S)$ is a unitary equivalence class of a C*-representation of $\cl S$.
We will identify an object of ${\rm Rep}_{{\rm C^*}}(\cl S)$ with a particular representative from its class without further mention. 

Given two objects $\Gamma_1 = (H_1,\phi_1, \pi_1)$ and $\Gamma_2 = (H_2,\phi_2, \pi_2)$ of ${\rm Rep}_{\rm C^*}(\cl S)$, a morphism from $\Gamma_1$ to $\Gamma_2$ is a completely bounded map 
$\Phi \colon \phi_1(\cl S)\rightarrow \phi_2(\cl S)$ satisfying the relations
\[
\Phi(\phi_1(a\cdot s \cdot b)) = \Phi(\pi_1(a) \phi_1(s)\pi_1(b)) = \pi_2(a)\Phi(\phi_1(s))\pi_2(b) \foral a, b \in \cl A_{\cl S}, s \in \cl S.
\]
We denote the set of such morphisms by ${\rm Hom}_{\cl S}(\Gamma_1, \Gamma_2)$, and note that ${\rm Rep}_{\rm C^*}(\cl S)$ is an additive category.
A functor $\cl F \colon {\rm Rep}_{\rm C^*}(\cl S) \to {\rm Rep}_{\rm C^*}(\cl T)$ is called \emph{completely contractive} 
if $\|\cl F(\Phi)\|_{\rm cb}\leq \|\Phi\|_{\rm cb}$ for every morphism $\Phi$ of ${\rm Rep}_{\rm C^*}(\cl S)$ 
(see \cite{Ble01ms}). 

Suppose that $\S \sim_{\Delta} \T$ and $\cl M$ is a TRO as in Proposition \ref{P:reduction}. 
Theorem \ref{th_ac} and Remark \ref{R:ind rep}
yield a completely contractive functor $\cl F \colon {\rm Rep}_{\rm C^*}(\cl S)\to {\rm Rep}_{\rm C^*}(\cl T)$, in the following way.
For an object $(H_\phi, \phi, \pi)$  from ${\rm Rep}_{\rm C^*}(\S)$ let $H_{\psi}$ be the completion of $\cl M\otimes_{\cl A_{\cl S}} H$ with respect to the semi-inner product given by $\langle a\otimes\xi,b\otimes \eta \rangle = \langle \phi(b^*a)\xi, \eta \rangle$ and $\psi$ be given by letting
\[
\left\langle \psi(t)(a\otimes\xi),b\otimes \eta \right\rangle = 
\left\langle \phi(b^* \hspace{-0.1cm} \cdot\hspace{-0.05cm}  t
\hspace{-0.05cm} \cdot \hspace{-0.05cm} a)\xi,\eta\right\rangle.
\]
By the proof of Theorem \ref{th_ac}, $(H_\psi, \psi)$ is an object of ${\rm Rep}_{\rm C^*}(\T)$; 
write $\sigma := \pi_{\psi} = \psi|_{\cl A_{\cl T}}$, and set $\cl F\left((H_\phi, \phi, \pi)\right) = (H_\psi, \psi, \si)$.
Note that
\begin{equation}\label{eq_psiphi}
\psi(m \cdot s \cdot n^*)(l\otimes\xi) = m\otimes\phi(s\cdot n^*l)(\xi) \ \foral m,n,l\in \cl M, s\in \cl S,
\end{equation}
and that the map $\nu \colon \cl M\rightarrow \cl B(H,K)$,
given by 
\[
\nu(m)(\xi) = m \otimes\xi \ \foral m\in\cl M, \xi\in H,
\]
satisfies the identity 
$$\nu(m)^*(l\otimes\xi) = \phi(m^*l)(\xi), \ \ \ m,l\in \cl M, \xi\in H,$$
implying that $\nu$ is a ternary morphism. 
We also note the relations 
\begin{enumerate}
\item $\psi(m \cdot s \cdot n^*) = \nu(m)\phi(s)\nu(n)^*$,
\item $\pi(l^*n) = \nu(l)^*\nu(n)$, and 
\item $\phi(m^* \cdot t \cdot n) = \nu(m)^*\psi(t)\nu(n)$,
\end{enumerate}
for all $m,n, l \in \cl M$, $s\in \cl S$, and $t\in \cl T$.
Thus the *-homomorphism $\sigma \colon \cl A_{\cl T}\rightarrow \cl B(K)$ satisfies the identities
\[
\sigma(mn^*) = \nu(m)\nu(n)^* \foral m,n \in \cl M
\]
and 
\[
\psi(\cl T) = [\nu(\cl M)\phi(\cl S)\nu (\cl M)^*], \;\; \phi(\cl S) = [\nu(\cl M)^*\psi(\cl T)\nu(\cl M)],
\]
\[
\pi(\cl A_{\cl S}) = [ \nu (\cl M)^*\nu(\cl M)],\;\;\sigma(\cl A_{\cl T}) = [ \nu(\cl M)\nu(\cl M)^*].
\]

Let $(H_i, \phi_i, \pi_i), i=1,2$ be objects of the category ${\rm Rep}_{\rm C^*}(\cl S)$ and $\Phi \colon \phi_1(\cl S)\rightarrow \phi_2(\cl S)$ be a morphism. Write $\nu_i$ and $\sigma_i$ for the 
corresponding maps, arising from $(H_i, \phi_i, \pi_i)$, $i = 1,2$. 
Set $(K_i, \psi_i, \pi_i) = \cl F((H_i,\phi_i,\si_i))$, $i=1,2$.
Let $n_i,m_i \in \cl M, s_i \in \cl S$, $i = 1,\dots,k$. 
There exists a contractive approximate identity of $[\cl M\cl M^*]$ of the form $(q_\lambda q_\lambda^*)_{\lambda}$ where $q_\lambda$ is a finite row over $\cl M$ for every $\lambda$.  
Suppressing the notation for the ampliation of $\nu_1$, we note that, for $\epsilon > 0$, there exists $q = q_{\lambda_0}$ such that 
\begin{align*}
\nor{\sum_{i=1}^k\nu_2(n_i)\Phi \left(\phi_1(s_i)\right)\nu_2(m_i)^*}-\epsilon
& \leq 
\nor{\sum_{i=1}^k\nu_2(q)\nu_2(q)^*\nu_2(n_i)\Phi \left(\phi_1(s_i)\right)\nu_2(m_i)^*\nu_2(q)\nu_2(q)^*}\\
& \leq 
\nor{\Phi\left(\nu_1(q)^*\sum_{i=1}^k\nu_1(n_i)\phi_1(s_i)\nu_1(m_i)^*\nu_1(q)\right)}\\
& \leq 
\|\Phi\|_{\rm cb} \nor{\nu_1(q)^*\left(\sum_{i=1}^k\nu_1(n_i)\phi_1(s_i)\nu_1(m_i)^*\right)\nu_1(q)}\\
& \leq 
\|\Phi\|_{\rm cb} \nor{\sum_{i=1}^k\nu_1(n_i)\phi_1(s_i)\nu_1(m_i)^*}.
\end{align*}
Therefore 
\[
\nor{\sum_{i=1}^k\nu_2(n_i)\Phi (\phi_1(s_i))\nu_2(m_i)^*}
\leq 
\|\Phi\|_{\rm cb} \nor{\sum_{i=1}^k\nu_1(n_i)\phi_1(s_i)\nu_1(m_i)^*}.
\]
We conclude that the linear map 
\[
\Psi \colon \psi_1(\cl T)\rightarrow \psi_2(\cl T); \ \nu_1(n)\phi_1(s)\nu_1(m)^*\mapsto \nu_2(n)\Phi(\phi_1(s))\nu_2(m)^*,
\]
is well-defined, bounded, and $\|\Psi\| \leq \|\Phi\|_{\rm cb}$.
Similarly $\Psi$ is completely bounded and $\|\Psi\|_{\rm cb}\leq\|\Phi\|_{\rm cb}$.

We set $\cl F(\Phi) := \Psi$ as defined above, and note that $\cl F(\Phi)$ is unital when $\Phi$ is unital, and completely positive when $\Phi$ is completely positive.
If $m,l,n,k\in \cl M$ and $s\in \cl S$ then 
\begin{align*}
\Psi(\psi_1((ml^*n) \cdot s \cdot k^*))
& = 
\Psi(\nu_1(ml^*n)\phi_1(s)\nu_1(k)^*)
= 
\nu_2(ml^*n)\Phi(\phi_1(s))\nu_2(k)^*\\
& =  
\sigma_2(ml^*)\nu_2(n)\Phi(\phi_1(s))\nu_2(k)^*
=  
\sigma_2(ml^*)\Psi(\nu_1(n)\phi_1(s)\nu_1(k)^*)\\
& =   
\sigma_2(ml^*)\Psi(\psi_1(n \cdot s \cdot k^*)).
\end{align*}
Thus,
\[
\Psi(\psi_1(bt))=\sigma_2(b)\Psi(\psi_1(t)) \foral b\in  \cl A_{\cl T}, t\in \cl T.
\]
Similarly,
\[
\Psi(\psi_1(tb))=\Psi(\psi_1(t))\sigma_2(b) \foral b \in \cl A_{\cl T},  t\in \cl T,
\]
and hence $\Psi\in {\rm Hom}\left((K_1, \psi_1, \si_1), (K_2, \psi_2, \si_2)\right)$.

We have thus defined a completely contractive functor $\cl F \colon {\rm Rep}_{\rm C^*}(\cl S) \rightarrow {\rm Rep}_{\rm C^*}(\cl T)$.
Let $\cl G \colon {\rm Rep}_{\rm C^*}(\cl T) \rightarrow {\rm Rep}_{\rm C^*}(\cl S)$ be the completely contractive 
functor arising from the adjoint $\cl M^*$ of $\cl M$ in an analogous way.  

\begin{theorem}\label{B0} 
Let $\cl S, \cl T$ be operator systems with $\cl S \sim_\Delta \cl T$. 
Then 
\[
\cl G \circ \cl F\cong {\rm Id}_{{\rm Rep}_{\rm C^*}(\cl S)} 
\qand
\cl F \circ \cl G\cong {\rm Id}_{{\rm Rep}_{\rm C^*}(\cl T)},
\]
up to natural equivalence.  
In particular, the categories ${\rm Rep}_{\rm C^*}(\cl S)$ and ${\rm Rep}_{\rm C^*}(\cl T)$ are equivalent via completely contractive functors.
\end{theorem}

\begin{proof}
We fix objects $(H_i, \phi_i, \pi_i)$ in ${\rm Rep}_{\rm C^*}(\cl S)$, $i = 1,2$, and let
\[
(K_i,\psi_i,\sigma_i)=\cl F((H_i, \phi_i, \pi_i)), \ (L_i, \zeta_i,\tau_i)=\cl G((K_i,\psi_i,\sigma_i)), \ 
i=1, 2.
\]
We also fix 
\[
\Phi\in {\rm Hom}_{\cl S}((H_1,\phi_1,\pi_1),(H_2,\phi_2,\pi_2)).
\]
Recalling the notation from the proof of Theorem \ref{th_ac}, we set $U_i := U_{H_i}$, $i=1,2$. 
We will show that the diagram
\begin{center}
\begin{tikzcd}[column sep=huge,row sep=huge]
\zeta_1(\cl S) \arrow[d,"(\cl G\circ \mathcal{F})(\Phi)" ] \arrow[r, "{\rm Ad}_{U_1}"] \arrow[d, black]
&\phi_1(\cl S) \arrow[d, "\Phi" black] \\
\zeta_2(\cl S)  \arrow[r, black, "{\rm Ad}_{U_2}" black]
& \phi_2(\cl S)
\end{tikzcd}
\end{center}
is commutative; that is, that 
\begin{equation}\label{eq_AdU1}
\Phi(U_1\zeta_1(s)U_1^*) = U_2(\cl G\circ \cl F)(\Phi)(\zeta_1(s))U_2^* \ \foral s \in \cl S.
\end{equation}

Let $\cl M$ be the TRO arising in Proposition \ref{P:reduction} and, for $i=1,2$, consider the ternary morphism
\[
\rho_i \colon \cl M\rightarrow \cl B(L_i, K_i) \textup{ given by } \rho_i(m)(n^*\otimes(l\otimes\xi))=\psi_i(mn^*)(l\otimes\xi).
\]
We have 
\[
\rho_i(m)(n^*\otimes(l\otimes\xi))=m\otimes\phi_i(n^*l)(\xi)=m\otimes U_i(n^*\otimes(l\otimes\xi));
\]
thus, 
\[
\rho_i(m)(\omega)=m\otimes U_i(\omega) \foral \omega \in L_i, \ i = 1,2.
\]
For $m, n \in \cl M$ and $\xi\in H_i$, we have 
\begin{align*} 
\sca{\rho_i(m)(\omega), n\otimes\xi}
& =
\sca{m\otimes U_i(\omega),n\otimes\xi }
=
\sca{\phi_i(n^*m)U_i(\omega), \xi}\\ 
& = 
\sca{U_i(\omega), \nu_i(m)^*(n\otimes\xi)}
=
\sca{\nu_i(m)U_i(\omega), n\otimes\xi}.
\end{align*}
Thus, 
\[
\rho_i(m)=\nu_i(m)U_i, \ \ i = 1,2.
\]

Let $m,n,k,l,p,q \in \cl M$, $s\in \cl S$ and $\xi\in H_2$. 
Then
\begin{align*}
& \hspace{0cm}
U_2 (\cl G \circ \cl F)(\Phi)(\rho_1(p)^*\nu_1(k)\phi_1(s)\nu_1(l)^*\rho_1(q)))U_2^*(\phi_2(m^*n)(\xi)) = \\
& \hspace{4cm} =
U_2 (\cl G\circ \cl F) (\Phi)(\rho_1(p)^*\nu_1(k)\phi_1(s)\nu_1(l)^*\rho_1(q)))U_2^*\nu_2(m)^*(n\otimes\xi)\\ 
& \hspace{4cm} =
U_2\rho_2(p)^*\cl F(\Phi)(\nu_1(k)\phi_1(s)\nu_1(l)^*)\rho_2(q)U_2^*\nu_2(m)^*(n\otimes\xi)\\
& \hspace{4cm} =
\nu_2(p)^*\nu_2(k)\Phi(\phi_1(s))\nu_2(l)^*\nu_2(q)\nu_2(m)^*(n\otimes\xi)\\ 
& \hspace{4cm} =
\Phi(\nu_1(p)^*\nu_1(k)\phi_1(s)\nu_1(l)^*\nu_1(q))\nu_2(m)^*(n\otimes\xi)\\ 
& \hspace{4cm} =
\Phi(\nu_1(p)^*\nu_1(k)\phi_1(s)\nu_1(l)^*\nu_1(q))(\phi_2(m^*n)(\xi))\hspace{-0.03cm}.
\end{align*}
Therefore 
\[
U_2 \left( (\cl G \circ \cl F) ( \Phi ) ( \rho_1(p)^* \nu_1(k) \phi_1(s) \nu_1(l)^* \rho_1(q) ) \right) U_2^*
=  
\Phi \left(\nu_1(p)^* \nu_1(k)\phi_1(s)\nu_1(l)^*\nu_1(q) \right).
\]
Since 
\[
\Phi(\nu_1(p)^*\nu_1(k)\phi_1(s)\nu_1(l)^*\nu_1(q)) 
=  
\Phi(\nu_1(p)^*\psi_1(k \cdot s \cdot l^*)\nu_1(q)) 
= 
\Phi(\phi_1(p^*ksl^*q)),
\]
we have that 
\[
U_2 \left( (\cl G\circ \cl F) (\Phi) (\rho_1(p)^*\nu_1(k)\phi_1(s)\nu_1(l)^*\rho_1(q)) \right) U_2^*=\Phi(\phi_1(p^*ksl^*q)).
\]
Since 
\[
\zeta_1(\cl S) = [\rho_1(\cl M)^*\nu_1(\cl M)\phi_1(\cl S)\nu_1(\cl M)^*\rho_1(\cl M)],
\]
we easily see that 
\[
U_2(\cl G\circ \cl F) (\Phi) (\zeta_1(s)) U_2^*=\Phi(\phi_1(s)) \foral s\in \cl S.
\]
Using (\ref{eq_UH}), we conclude that (\ref{eq_AdU1}) holds true, showing that $\cl G\circ \cl F\cong {\rm Id}_{{\rm Rep}_{\rm C^*}(\cl S)}$. 
By symmetry we also have that $\cl F \circ \cl G\cong {\rm Id}_{{\rm Rep}_{\rm C^*}(\cl T)}$, and the proof is complete. 
\end{proof}

\begin{theorem}\label{P:M-fun0}
Let $\S$ and $\T$ be operator systems with $\S \sim_{\Delta} \T$. 
The functors $\cl F$ and $\cl G$ preserve the maximal representations and the Choquet representations.
In addition, $\S$ is hyperrigid if and only if $\T$ is hyperrigid.
\end{theorem}

\begin{proof}
Using Proposition \ref{P:reduction}, we assume that 
\[
\phi \colon \S \to \B(H) \qand \psi \colon \T \to \B(K)
\] 
are unital complete order isomorphisms such that 
\[
\cenv(\S) \simeq \ca(\phi(\S)) \qand \cenv(\T) \simeq \ca(\psi(\T)),
\]
and that $\M$ is a closed TRO such that $\phi(\S) \sim_{\rm TRO} \psi(\T)$ via $\M$, 
\[
\phi(\A_\S) = [\M^* \M] = [\M^* \psi(\A_\T) \M] \qand \A_\T \simeq [\M \M^*] = [\M \phi(\A_\S) \M^*].
\]
Thus $\ca(\phi(\S)) \sim_{\rm TRO} \ca(\psi(\T))$ and $\phi(\A_\S) \sim_{\rm TRO} \psi(\A_\T)$ via $\M$, and as these are C*-algebras, the TRO-equivalence is a strong Morita equivalence.

Let $\rho$ be a maximal unital completely positive map of $\S$. 
By Arveson's Invariant Principle we may assume $\rho \colon \phi(\S) \to \B(L)$.
Maximality implies that $\rho$ has a (unique) unital completely positive extension to a $*$-representation of $\ca(\phi(\S))$ and thus it induces a C*-representation of $\cl S$. 
Let $\tau := \cl F(\rho)$; thus, 
$\tau \colon \psi(\T) \to \B(\M \widetilde{\otimes} L)$
be the induced unital completely positive map (see the proof of Theorem \ref{th_ac}).
By \cite{DM05}, there exists a maximal dilation $\wt{\tau}$ of $\tau$ acting on some $K \supseteq \M \widetilde{\otimes} L$.
We will show that $\wt{\tau}$ is trivial.

Towards this end, consider $\cl G(\wt{\tau})$ and notice that by construction the functors $\cl G$ and $\cl F$ preserve dilations.
Therefore
\[
\cl G(\wt{\tau}) = \begin{bmatrix} (\cl G \circ \cl F)(\rho) & \ast \\ \ast & \ast \end{bmatrix} 
= \begin{bmatrix} (\cl G \circ\cl F)(\rho) & 0 \\0 & \ast \end{bmatrix},
\]
since $\cl G \circ \cl F(\rho)$ is unitarily equivalent to the maximal $\rho$, and thus maximal itself.
Applying $\F$ to the direct sum we get that $\wt{\tau}$ is unitarily equivalent to 
\[
(\cl F \circ \cl G)(\wt{\tau}) = \begin{bmatrix} (\cl F \circ \cl G \circ \cl F)(\rho)& 0 \\ 0 & \ast \end{bmatrix},
\]
where $(\F \circ \G \circ \F)(\rho)$ is unitarily equivalent to $\tau$.
The associated unitaries preserve the direct sum (and one is an extension of the other); 
after applying them, we see that $\wt{\tau}$ is a trivial dilation of $\tau$.
Hence the maximal representations are preserved.

Note that the functors preserve irreducibility of the unique C*-extensions, and thus the Choquet representations are also preserved.
Similar arguments, which are left to the reader, can be applied to show that if $\S$ is hyperrigid then so is $\T$.
\end{proof}

The second category we consider, denoted ${\rm Hmod}(\cl S)$, has the same objects as ${\rm Rep}_{{\rm C^*}}(\cl S)$, but given  objects $\Gamma_1 = (H_1,\phi_1)$ and $\Gamma_2 = (H_2,\phi_2)$, the set ${\rm Int}_{\cl S}(\Gamma_1, \Gamma_2)$ of morphisms from $\Gamma_1$ to $\Gamma_2$ is defined by letting 
\[
{\rm Int}\mbox{}_{\cl S}(\Gamma_1, \Gamma_2) = \{T\in \cl B(H_1,H_2) \mid T\phi_1(s) = \phi_2(s)T \mbox{ for all } s\in \cl S\};
\]
we call the elements $T\in {\rm Int}\mbox{}_{\cl S}(\Gamma_1, \Gamma_2)$ \emph{intertwiners} of the pair $(\Gamma_1,\Gamma_2)$. 
Keeping the notation already defined for the category ${\rm Rep}_{\rm C^*}(\cl S)$, given objects $\Gamma_i = (H_i,\phi_i,\pi_i)$, $i = 1,2$, and an element $T\in {\rm Int}_{\cl S}(\Gamma_1, \Gamma_2)$, set $(K_i,\psi_i) = \cl F(\Gamma_i)$, $i = 1,2$, and let 
$\tilde{T} \colon \cl M\odot H_1\to \cl M\odot H_2$
be the linear map, 
given by $\tilde{T}(m\otimes \xi) = m\otimes T\xi$, $m\in \cl M$, $\xi\in H_1$. 
Since $T$ is an intertwiner, if $m\in \cl M$, $a\in \cl A_{\cl S}$ and $\xi\in H_1$ then 
\begin{align*}
\tilde{T}\left((m \cdot a) \otimes \xi - m \otimes (a \cdot \xi)\right)
& = 
(m \cdot a) \otimes T \xi - m \otimes T \pi_1(a)\xi\\
& = 
(m \cdot a) \otimes T \xi - m \otimes \pi_2(a)T \xi\\
& = 
(m \cdot a) \otimes T \xi - m \otimes \left(a\cdot T \xi\right).
\end{align*}
Thus, $\tilde{T}$ induces a linear map (denoted in the same way) 
$\tilde{T} \colon \cl M\otimes_{[\cl M^*\cl M]} H_1 \to \cl M\otimes_{[\cl M^*\cl M]} H_2$.
Standard arguments, similar to (\ref{eq_nthet}), show that $\tilde{T}$ is bounded and hence induces a linear map (denoted in the same way) $\tilde{T} \colon K_1\to K_2$. 
In addition, if $m,n,l,\in \cl M$, $s\in \cl S$ and $\xi\in H_1$, then, using (\ref{eq_psiphi}), 
we have 
\begin{align*}
\tilde{T}\psi(msn^*)\left(l\otimes \xi\right) 
& = 
\tilde{T}\left(m\otimes \phi(sn^*l)\xi\right) 
= 
m\otimes T\phi(sn^*l)\xi\\
& = 
m\otimes \phi(sn^*l)T\xi
=
\psi(msn^*)(l\otimes T\xi)
= 
\psi(msn^*)\tilde{T}\left(l\otimes \xi\right),
\end{align*}
showing that $\tilde{T}\in {\rm Int}_{\cl T}((K_1,\psi_1), (K_2,\psi_2))$. 
We set $\cl F(T) := \tilde{T}$. 
Techniques, similar to those in Theorem \ref{B0}, show the following theorem.

\begin{theorem}\label{B02}
Let $\cl S, \cl T$ be operator systems with $\cl S \sim_\Delta \cl T$. 
We have that
\[
\cl G \circ \cl F\cong {\rm Id}_{{\rm Hmod}(\cl S)} 
\qand
\cl F \circ \cl G\cong {\rm Id}_{{\rm Hmod}(\cl T)},
\]
up to natural equivalence.  
In particular, the categories ${\rm Hmod}(\cl S)$ and ${\rm Hmod}(\cl T)$ are equivalent through completely contractive functors.
\end{theorem}

\begin{remark}\label{r_Cstarco}
Let $\cl A$ be a unital C*-algebra, considered as an operator system. 
As $\A$ coincides with its multiplier algebra, the category ${\rm Hmod}(\cl S)$ coincides with the category of Hermitian modules considered by Rieffel in \cite{Rie74jpaa}; in particular, its objects are (equivalence classes of) pairs $(H,\pi)$, where $H$ is a Hilbert space and $\pi \colon \cl A\to \cl B(H)$ is a unital *-homomorphism. 
Since Morita and strong Morita equivalence of C*-algebras are genuinely 
distinct notions (see e.g. \cite{Beer82}), we have that the converse of Theorem \ref{B02} does not hold true. 
We are not aware if the converse of Theorem \ref{B0} holds true. 
\end{remark}

\section{Morita equivalence invariants}\label{s_mei}

In this section, we examine the behaviour of tensor products 
with respect to $\Delta$-equivalence, 
and lift a well-known ideal lattice isomorphism property from the C*-algebra to the operator system case. 

\subsection{Isomorphism of lattices of kernels}\label{ss_kernels}

We write  ${\rm Ker}(\cl S)$ for the set of all \emph{kernels} of an operator system $\cl S$, that is, the set of subspaces $\cl J\subseteq \cl S$ for which there exists a Hilbert space $H$ and a unital completely positive map $\phi \colon \cl S\to \cl B(H)$ such that $\cl J = \ker(\phi)$. 
By \cite[Proposition 3.1]{KPTT13}, these are precisely the subspaces $\cl J \subseteq\cl S$ for which there exists a subset $A$ of the state space $S(\cl S)$ of $\cl S$ such that 
\[
\cl J = \cl J_A := \{x\in \cl S \mid f(x) = 0, \ f\in A\}.
\]
For a subspace $\J\subseteq \cl S$, we write
\[
{\rm hull}(\cl J) := \{f\in S(\cl S) \mid \cl J\subseteq \ker(f)\};
\]
thus, if $\cl J$ is a kernel, then $\cl J = \cl J_{{\rm hull}(\cl J)}$. 
For $\cl J_i\in {\rm Ker}(\cl S)$ and $A_i = {\rm hull}(\cl J_i)$, $i = 1,2$, set 
\begin{equation}\label{eq_latt}
\cl J_1 \wedge \cl J_2 := \cl J_{A_1\cup A_2}  
\qand 
\cl J_1 \vee \cl J_2 := \cl J_{A_1\cap A_2}.
\end{equation}
Note that $\cl J_1\wedge \cl J_2 = \cl J_1\cap \cl J_2$ and that $\cl J_1 \vee \cl J_2$ is the smallest  kernel containing $\cl J_1$ and $\cl J_2$; equivalently, $\J_1 \vee \J_2$ is the intersection of all kernels that contain $\J_1 + \J_2$.
When equipped with the operations (\ref{eq_latt}), the set ${\rm Ker}(\cl S)$ becomes a lattice. 

Let ${\rm Ker}_{\cl A_{\cl S}}(\cl S)$ be the set of all subspaces $\cl J$ for which there exists an element $(H_\phi, \phi)$ of ${\rm Rep}_{\rm C^*}(\cl S)$ such that $\cl J = \ker(\phi)$.
By definition, the elements of ${\rm Ker}_{\cl A_{\cl S}}(\cl S)$ are $\cl A_{\cl S}$-bimodules. 
Since $\A_\S$ is unital, we have, in particular, that
\[
[\A_\S \J \A_\S] = \J, \foral \J \in {\rm Ker}_{\A_\S}(\S).
\]
Let $\cl J_i\in {\rm Ker}_{\cl A_{\cl S}}(\cl S)$, say $\cl J_i = \ker(\phi_i)$ for some $(H_i,\phi_i)$ from ${\rm Rep}_{\rm C^*}(\cl S)$, $i = 1,2$. 
Letting $H = H_1 \oplus H_2$, $\phi(x) = \phi_1(x) \oplus \phi_2(x)$ and  $\pi(a) = \pi_1(a) \oplus \pi_2(a)$, $x\in \cl S$, $a\in \cl A_{\cl S}$, we have that $(H,\phi)$ is an object from ${\rm Rep}_{\rm C^*}(\cl S)$ and $\cl J_1\cap \cl J_2 = \ker(\phi)$. 
Set $\cl J_1\wedge \cl J_2 = \cl J_1\cap \cl J_2$. 
Similarly to the previous paragraph, write  $\cl J_1\vee \cl J_2$ for the smallest element of ${\rm Ker}_{\cl A_{\cl S}}(\cl S)$ containing $\cl J_1 + \cl J_2$, turning ${\rm Ker}_{\cl A_{\cl S}}(\cl S)$ into a lattice with respect to the operations $\wedge$ and $\vee$ (the element $\cl J_1\vee \cl J_2$ exists since ${\rm Ker}_{\cl A_{\cl S}}(\cl S)$ is a complete semi-lattice, when equipped with the operation $\wedge$). 
Note that, in the case where $\cl S = \cl A_{\cl S}$ (that is, $\cl S$ is a C*-algebra), 
${\rm Ker}_{\cl A_{\cl S}}(\cl S)$ coincides with the lattice of all closed ideals of $\cl S$. 

\begin{theorem}\label{th_latt}
Let $\cl S$ and $\cl T$ be operator systems. 
If $\S \sim_{\Delta} \T$ then the lattices ${\rm Ker}_{\cl A_{\cl S}}(\cl S)$ and ${\rm Ker}_{\cl A_{\cl T}}(\cl T)$ are isomorphic. 
\end{theorem}

\begin{proof}
Keeping the notation of Theorem \ref{B0}, let
\[
\cl F \colon {\rm Rep}_{\rm C^*}(\cl S) \to {\rm Rep}_{\rm C^*}(\cl T)
\qand
\cl G \colon {\rm Rep}_{\rm C^*}(\cl T) \to {\rm Rep}_{\rm C^*}(\cl S)
\]
be the functors specified therein.
For $\cl J\in {\rm Ker}_{\cl A_{\cl S}}(\cl S)$, write $\cl J = \ker(\phi)$ for some $(H_\phi, \phi)\in {\rm Rep}_{\rm C^*}(\cl S)$, and let $\tilde{\cl F}(\cl J) = \ker(\cl F(\phi))$.
It can be seen, e.g., using item (iv) from Section \ref{ss_ci}, 
that 
$y\in \ker(\cl F(\phi))$ if and only if $b^* y a\in \ker(\phi)$ for all $a,b\in \cl M$; thus the restriction
\[
\tilde{\cl F} \colon {\rm Ker}_{\cl A_{\cl S}}(\cl S) \to {\rm Ker}_{\cl A_{\cl T}}(\cl T)
\]
is a well-defined map.
In addition, it can be seen from the construction that
\[
\M \widetilde{\otimes}_{[\cl M^*\cl M]} (H_1 \oplus H_2) 
\simeq (\M \widetilde{\otimes}_{[\cl M^*\cl M]} H_1) \oplus (\M \widetilde{\otimes}_{[\cl M^*\cl M]} H_2),
\]
and hence $\cl F$ preserves direct sums in the sense that
\[
\cl F(H_1\oplus H_2) \simeq \cl F(H_1)\oplus \cl F(H_2)
\qand
\cl F(\phi_1\oplus\phi_2) \simeq \cl F(\phi_1)\oplus \cl F(\phi_2),
\]
up to unitary equivalence.
Thus, 
\begin{align*}
\tilde{\cl F}(\cl J_1\cap \cl J_2) 
& = \ker(\cl F(\phi_1\oplus \phi_2)) 
= \ker\left(\cl F(\phi_1)\oplus \cl F(\phi_2)\right)\\
& =
\ker\left(\cl F(\phi_1)\right) \cap \ker\left(\cl F(\phi_2)\right)
= \tilde{\cl F}(\cl J_1)\cap \tilde{\cl F}(\cl J_2).
\end{align*}
In particular, $\tilde{\cl F}$ is order-preserving, and hence preserves the operation $\vee$. 
We conclude that $\tilde{\cl F}$ is a lattice homomorphism. 

By symmetry, $\tilde{\cl G}$ is a lattice homomorphism. 
Note, in fact, that
\begin{equation}\label{eq_FG}
\tilde{\cl F}(\cl J) = [\cl M\cl J\cl M^*] \ \mbox{ and } \tilde{\cl G}(\cl I) = [\cl M^*\cl I\cl M].
\end{equation}
Since $\A_\S \simeq [\M\M^*]$ we have that
\[
(\tilde{\cl G} \circ \tilde{\cl F})(\J) = [\M^* \M \J \M^* \M] = [\A_\S \J \A_\S] = \J,
\]
and therefore $\tilde{\cl F}$ and $\tilde{\cl G}$ are mutual inverses.
\end{proof}

\begin{remark}
An alternative approach for the lattice isomorphism in Theorem \ref{th_latt}
would be to use the formulas (\ref{eq_FG}) directly as the definition of the lattice homomorphisms, once we have fixed the TRO equivalent isometric representations of $\S$ and $\T$.
One can directly check that indeed this provides a lattice isomorphism.
For the operation $\wedge$ let $\J_1, \J_2 \in {\rm Ker}_{\A_\S}$ and note that
\[
[\M (\J_1 \cap \J_2) \M^*] \subseteq [\M \J_1 \M^*] \cap [\M \J_2 \M^*].
\]
On the other hand for $x \in [\M \J_1 \M^*] \cap [\M \J_2 \M^*]$ we have that
\[
\M^* x \M \subseteq [\M^* \M \J_1 \M^* \M] \cap [\M^* \M \J_2 M^* \M] = \J_1 \cap \J_2.
\]
Therefore, using the fact that $[\M \M^*]$ is unital, we obtain
\[
x \in [\M \M^*] x [\M \M^*] \subseteq [\M (\J_1 \cap \J_2) \M^*],
\]
giving the required equality
\[
[\M (\J_1 \cap \J_2) \M^*] = [\M \J_1 \M^*] \cap [\M \J_2 \M^*].
\]
Likewise, for $\vee$, we have by definition that $\J_1 \vee \J_2 \supseteq \J_1 + \J_2$ and thus
\[
[\M (\J_1 \vee \J_2) \M^*] \supseteq [\M \J_1 \M^*] \vee [\M \J_2 \M^*].
\]
On the other hand
\[
[\M^* \cdot [\M \J_1 \M^*] \vee [\M \J_2 \M^*] \cdot \M] \supseteq [\M^* \M \J_i \M^* \M] = \J_i,
\]
and thus
\[
[\M^* \cdot [\M \J_1 \M^*] \vee [\M \J_2 \M^*] \cdot \M] \supseteq \J_1 \vee \J_2.
\]
Applying again $\M$ and using the TRO property $[\M \M^* \M] = \overline{\M}$, we derive that
\begin{align*}
[\M \J_1 \M^*] \vee [\M \J_2 \M^*]
& = 
[\M \M^* \cdot [\M \J_1 \M^*] \vee [\M \J_2 \M^*] \cdot \M \M^*]
\supseteq 
[\M (\J_1 \vee \J_2) \M^*]
\end{align*}
and therefore the equality
\[
[\M (\J_1 \vee \J_2) \M^*] = [\M \J_1 \M^*] \vee [\M \J_2 \M^*].
\]
\end{remark}

\subsection{Tensor products}

Tensor products of operator systems have been studied by A. Kavruk, V. Paulsen, I. Todorov and M. Tomforde \cite{KPTT11, KPTT13}; 
we recall some necessary background. 
Let $\S$ and $\R$ be two operator systems and write $\S \odot \R$ for their algebraic tensor product.
There are several operator system structures $\tau$ of interest one can endow $\S \odot \R$ with, and we will write $\S \otimes_\tau \R$ for the completion of $\S \odot \R$ with respect to the norm of the $\tau$-tensor product.

The \emph{maximal tensor product $\S \otimes_{\max} \R$} is characterised by the property that it linearises 
every jointly completely positive bilinear map from $\S \times \R$ into an operator system, 
to a completely positive map.
The \emph{minimal} (or \emph{spatial}) \emph{tensor product} $\S \otimes \R$ is given as the (concrete) closed operator subsystem of $\B(H \otimes K)$ arising from (any) unital complete order embeddings $\S \hookrightarrow \B(H)$ and $\R \hookrightarrow \B(K)$.
These are examples of associative tensor products.

In the \emph{commuting tensor product} $\S \otimes_{\rm  c} \R$, an element 
$u\in M_n(\S \odot \R)$ is positive if it has positive images under the maps $(\phi\cdot\psi)^{(n)}$ (where $(\phi\cdot\psi)(x\otimes y) = \phi(x)\psi(y)$), for all unital completely positive maps $\phi$ and $\psi$ with commuting ranges, defined on $\S$ and $\R$, respectively. 
In \cite[Lemma 2.5]{KPTT13} it is shown that:
\begin{enumerate}
\item if $\A$ is a C*-algebra then the inclusion $\S \odot \A \hookrightarrow \S \otimes_{\max} \A$ extends to a complete order isomoprhism on $\S \otimes_{\rm c} \A$;
\item the inclusion $\S \odot \R \hookrightarrow \ca_{\max}(\S) \otimes_{\max} \ca_{\max}(\R)$ extends to a complete order embedding on $\S \otimes_{\rm c} \R$;
\item the inclusion $\S \odot \R \hookrightarrow \S \otimes_{\rm c} \ca_{\max}(\R)$ extends to a complete order embedding on $\S \otimes_{\rm c} \R$.
\end{enumerate}
By combining items (i) and (iii) we thus derive the following complete order embedding in the case of a C*-algebra $\A$ and an operator system $\R$:
\begin{equation}\label{eq:c}
\A \otimes_{\rm c} \R \stackrel{{\rm (iii)}}{\hookrightarrow} \A \otimes_{\rm c} \ca_{\max}(\R) \stackrel{{\rm (i)}}{\simeq} \A \otimes_{\max} \ca_{\max}(\R),
\end{equation}
and thus
\begin{equation}\label{eq:c2}
\A \otimes_{\rm c} \R \simeq [\A \odot \iota_{\max}(\R)]^{-\A \otimes_{\max} \ca_{\max}(\R)}.
\end{equation}

The \emph{essential left tensor product} $\S \otimes_{\rm el} \R$ is defined by the requirement that the inclusion map $\S \odot \R \subseteq \I(\S) \otimes_{\max} \R$ lifts to a complete order embedding on $\S \otimes_{\rm el} \R$.
The \emph{essential right tensor product} $\S \otimes_{\rm er} \R$ is defined by the requirement that 
the inclusion map $\S \odot \R \subseteq \S \otimes_{\max} \I(\R)$ lifts to a complete order embedding on $\S \otimes_{\rm er} \R$.

If $\tau$ and $\tau'$ are operator system tensor products for which the identity map on $\S \odot \R$ lifts to a completely positive map $\cl S\otimes_\tau\cl R\to \cl S\otimes_{\tau'}\cl R$ for all operator systems $\cl S$ and $\cl R$ then we write $\tau' \leq \tau$.
In this case, we say that $\S$ is $(\tau', \tau)$-nuclear if for every operator system $\R$ the identity map on $\S \odot \R$ 
lifts to a complete order isomorphism  $\S \otimes_\tau \R \to \S \otimes_{\tau'} \R$.
We say that $\S$ is \emph{nuclear} if it is $(\min,\max)$-nuclear.

It was shown in \cite{KPTT11} that the aforementioned tensor product structures compare as follows:
\begin{align*}
\min \leq {\rm er}, {\rm el} \leq {\rm c} \leq \max.
\end{align*}
In \cite{KPTT13}, 
the nuclearity with respect to these tensor product structures is linked to 
approximation properties of an operator system $\S$.
In more detail, $\S$ is called \emph{exact} if for every C*-algebra $\A$ with an ideal $\I$ we have that the well-defined map
\[
\quo{\A \otimes \S}{\I \otimes \S} \rightarrow \left(\quo{\A}{\I}\right) \otimes \S
\]
is a completely isometric isomorphism.
We say that $\S$ has the \emph{operator system local lifting property} (OSLLP), if whenever $\A$ is a unital C*-algebra 
with a closed ideal $\I \lhd \A$ and canonical quotient map $q_\I \colon \A \to \A/\I$, and $\phi \colon \S \to \A/\I$ is a unital completely positive map, then for every finite-dimensional operator system $\E \subseteq \S$ there is a unital completely positive map $\psi \colon \E \to \A$ such that the diagram
\[
\xymatrix{
\A \ar[drr]^{q_\I} & & \\
\E \ar[u]^{\psi} \ar[rr]_{\phi|_\E} & & \A/ \I 
}
\]
is commutative.
We say that $\S$ has the \emph{weak expectation property} (WEP), if the canonical inclusion $\iota \colon \S \to \S^{dd}$ extends to a unital completely positive map $\phi \colon \I(\S) \to \S^{dd}$.
Here $\S^{dd}$ denotes the double dual operator system of $\S$.
We say that $\S$ has the \emph{double commutant expectation property} (DCEP), if for any unital complete order embedding $\iota \colon \S \to \B(H)$ there exists a unital completely positive extension $\phi \colon \I(\S) \to \iota(\S)'' \subseteq \B(H)$.
The following equivalences were shown in \cite{KPTT13}:
\begin{enumerate}
\item $\S$ is exact if and only if $\S$ is $(\min, {\rm el})$-nuclear.
\item $\S$ has the (OSLLP) if and only if $\S$ is $(\min, {\rm er})$-nuclear.
\item $\S$ has the (WEP) if and only if $\S$ is $({\rm el, \max})$-nuclear.
\item $\S$ has the (DCEP) if and only if $\S$ is $({\rm el}, {\rm c})$-nuclear.
\end{enumerate}
Using these characterisations, 
we will prove that $\Delta$-equivalence preserves the aforementioned operator system approximation properties.

We will need the following lemma.
Recall the following definition from \cite{KPTT11, KPTT13}.
A $\tau$-tensor product is called \emph{functorial} if every pair of unital completely positive maps $\phi_i \colon \S_i \to \T_i$, $i=1,2$, integrates to a unital completely positive map $\phi_1 \otimes \phi_2 \colon \S_1 \otimes_\tau \S_2 \to \T_1 \otimes_\tau \T_2$.

\begin{lemma}\label{l_nuc}
Let $\S$ be an operator system that is $(\tau', \tau)$-nuclear for $\tau' \leq \tau$ such that:
\begin{enumerate}
\item $\tau'$ and $\tau$ are associative tensor products; or
\item $\tau'$ is associative and $\tau = {\rm el}$; or
\item $\tau'$ is associative and $\tau = {\rm er}$; or
\item $\tau' = {\rm el}$ and $\tau$ is associative; or
\item $\tau' = {\rm er}$ and $\tau$ is associative; or
\item $(\tau', \tau) = ({\rm el}, {\rm c})$.
\end{enumerate}
Then the operator system $M_n(\S)$ is $(\tau', \tau)$-nuclear.
\end{lemma}

\begin{proof}
The nuclearity of $M_n$ implies
\[
M_n \otimes_{\tau} \cl S\simeq M_n \otimes \cl S, \ \ \ n\in \bb{N}. 
\]
Further, by \cite[Paragraph 4.2.10]{blm}, the embedding $M_n(\S) \hookrightarrow M_n(\I(\S))$ 
extends to a complete order isomorphism 
\begin{equation}\label{eq_IMnS}
\I(M_n(\S)) \simeq M_n(\I(\S)), \ \ \ n\in \bb{N}.
\end{equation}

We fix an operator system $\cl R$. 

\smallskip

\noindent 
(i) Since $M_n$ is nuclear, we have the complete order isomorphisms
\begin{align*}
M_n(\T) \otimes_\tau \R 
& \simeq 
(M_n \otimes \T) \otimes_\tau \R
\simeq
M_n \otimes_\tau \T \otimes_\tau \R 
\simeq
M_n \otimes (\T \otimes_\tau \R)\\
& \simeq 
M_n \otimes (\T \otimes_{\tau'} \R) 
\simeq
M_n \otimes_{\tau'} \T \otimes_{\tau'} \R
\simeq
(M_n \otimes \T) \otimes_{\tau'} \R
\simeq
M_n(\T) \otimes_{\tau'} \R.
\end{align*}

\smallskip

\noindent
(ii) Using (\ref{eq_IMnS}), the nuclearity of $M_n$ (and thus also preservation of inclusions by $M_n \otimes_{\max} -$), the associativity of $\tau'$ and the 
$(\tau', {\rm el})$-nuclearity of $\S$, we have the complete order isomorphisms
\begin{align*}
M_n(\S) \otimes_{\tau'} \R 
& \simeq 
M_n \otimes_{\tau'} (\S \otimes_{\tau'} \R)
\simeq
M_n \otimes_{\max} (\S \otimes_{\rm el} \R) \\
& \hookrightarrow
M_n \otimes_{\max} (\I(\S) \otimes_{\max} \R)
\simeq
M_n(\I(S)) \otimes_{\max} \R
\simeq
\I(M_n(\S)) \otimes_{\max} \R.
\end{align*}
Therefore we have the following diagram that fixes $M_n(\S) \odot \R$:
\[
\xymatrix{
M_n(\S) \otimes_{\rm el} \R \ar[rr] \ar[rrd] & & M_n(\S) \otimes_{\tau'} \R \ar[d] \\
& & \I(M_n(\S)) \otimes_{\max} \R,
}
\]
where the horizontal arrow is given by $\tau' \leq {\rm el}$, the diagonal arrow is a complete order embedding by the definition of $\otimes_{\rm el}$, and the vertical arrow was shown to be a complete order embedding.
This shows that the horizontal map is a complete order isomorphism.

\smallskip

\noindent
(iii) Using the nuclearity of $M_n$ (and thus also preservation of inclusions by $M_n \otimes_{\max} -$), the associativity of $\tau'$ and 
the $(\tau', {\rm er})$-nuclearity of $\S$, we have the complete order isomorphisms
\begin{align*}
M_n(\S) \otimes_{\tau'} \R 
& \simeq 
M_n \otimes_{\tau'} (\S \otimes_{\tau'} \R)
\simeq
M_n \otimes_{\max} (\S \otimes_{\rm er} \R) \\
& \hookrightarrow
M_n \otimes_{\max} (\S \otimes_{\max} \I(\R))
\simeq
M_n(S) \otimes_{\max} \I(\R).
\end{align*}
Therefore we have the following diagram that fixes $M_n(\S) \odot \R$:
\[
\xymatrix{
M_n(\S) \otimes_{\rm er} \R \ar[rr] \ar[rrd] & & M_n(\S) \otimes_{\tau'} \R \ar[d] \\
& & M_n(\S) \otimes_{\max} \I(\R),
}
\]
where the horizontal arrow is given by $\tau' \leq {\rm er}$, the diagonal arrow is a complete order embedding by the definition of $\otimes_{\rm er}$, and the vertical arrow was shown to be a complete order embedding.
This shows that the horizontal map is a complete order isomorphism.

\smallskip

\noindent
(iv) Using (\ref{eq_IMnS}), the nuclearity of $M_n$ (and thus also preservation of inclusions by $M_n \otimes_{\max} -$), the associativity of $\tau$ and the $({\rm el}, \tau)$-nuclearity of $\S$, we have the complete order isomorphisms
\begin{align*}
M_n(\S) \otimes_{\tau} \R 
& \simeq 
M_n \otimes_{\tau} (\S \otimes_{\tau} \R)
\simeq
M_n \otimes_{\max} (\S \otimes_{\rm el} \R) \\
& \hookrightarrow
M_n \otimes_{\max} (\I(\S) \otimes_{\max} \R)
\simeq
M_n(\I(S)) \otimes_{\max} \R
\simeq
\I(M_n(\S)) \otimes_{\max} \R.
\end{align*}
Therefore we have the following diagram that fixes $M_n(\S) \odot \R$:
\[
\xymatrix{
M_n(\S) \otimes_{\tau} \R \ar[rr] \ar[rrd] & & M_n(\S) \otimes_{\rm el} \R \ar[d] \\
& & \I(M_n(\S)) \otimes_{\max} \R,
}
\]
where the horizontal arrow is given by ${\rm el} \leq \tau$, the vertical arrow is a complete order embedding by the definition of $\otimes_{\rm el}$, and the diagonal arrow was shown to be a complete order embedding.
This shows that the horizontal map is a complete order isomorphism.

\smallskip

\noindent
(v) Using the nuclearity of $M_n$ (and thus also preservation of inclusions by $M_n \otimes_{\max} -$), the associativity of $\tau$ and 
the $({\rm er}, \tau)$-nuclearity of $\S$, we have the complete order isomorphisms
\begin{align*}
M_n(\S) \otimes_{\tau} \R 
& \simeq 
M_n \otimes_{\tau} (\S \otimes_{\tau} \R)
\simeq
M_n \otimes_{\max} (\S \otimes_{\rm er} \R) \\
& \hookrightarrow
M_n \otimes_{\max} (\S \otimes_{\max} \I(\R))
\simeq
M_n(S) \otimes_{\max} \I(\R).
\end{align*}
Therefore we have the following diagram that fixes $M_n(\S) \odot \R$:
\[
\xymatrix{
M_n(\S) \otimes_{\tau} \R \ar[rr] \ar[rrd] & & M_n(\S) \otimes_{\rm er} \R \ar[d] \\
& & M_n(\S) \otimes_{\max} \I(\R),
}
\]
where the horizontal arrow is given by ${\rm er} \leq \tau$, the vertical arrow is a complete order embedding by the definition of $\otimes_{\rm er}$, and the diagonal arrow was shown to be a complete order embedding.
This shows that the horizontal map is a complete order isomorphism.

\smallskip

\noindent
(vi) First we show that there is a complete order isomorphism
\begin{equation}\label{eq_neweq}
M_n(\S) \otimes_{\rm c} \R \simeq (M_n \otimes \S) \otimes_{\rm c} \R \simeq M_n \otimes (\S \otimes_{\rm c} \R)
\end{equation}
that extends the canonical inclusions of the algebraic tensor product.
Towards this end, by \cite[Lemma 2.5]{KPTT13} we have that
\[
\S \otimes_{\rm c} \R \hookrightarrow \ca_{\max}(\S) \otimes_{\max} \ca_{\max}(\R).
\]
Using the nuclearity of $M_n$ and the preservation of inclusions by $M_n \otimes -$, we have that
\begin{align*}
M_n \otimes (\S \otimes_{\rm c} \R) 
& \hookrightarrow
M_n \otimes (\ca_{\max}(\S) \otimes_{\max} \ca_{\max}(\R)) \\
& \simeq
M_n \otimes_{\max} (\ca_{\max}(\S) \otimes_{\max} \ca_{\max}(\R)) 
\simeq
(M_n \otimes_{\max} \ca_{\max}(\S)) \otimes_{\max} \ca_{\max}(\R),
\end{align*}
where we used the associativity of $\otimes_{\max}$.
By construction and \cite[Lemma 2.5]{KPTT13}, this map takes values in
\[
\left[(M_n \otimes_{\max} \ca_{\max}(\S)) \odot \iota_{\max}(\R) \right]^{-(M_n \otimes_{\max} \ca_{\max}(\S)) \otimes_{\max} \ca_{\max}(\R) }
\simeq
(M_n \otimes_{\max} \ca_{\max}(\S)) \otimes_{\rm c} \R,
\]
as $M_n \otimes_{\max} \ca_{\max}(\S)$ is a C*-algebra (see equation (\ref{eq:c2})).
In particular, we get a complete order embedding that takes values in
\[
\left[ M_n \odot \iota_{\max}(\S) \right]^{-M_n \otimes_{\max} \ca_{\max}(\S)}  \otimes_{\rm c} \R
\simeq
(M_n \otimes \S) \otimes_{\rm c} \R,
\]
where we used the nuclearity of $M_n$.
We thus obtain a complete order embedding that preserves the copies of $M_n \odot (\S \odot \R) = (M_n \odot \S) \odot \R$, and thus it is surjective.
The required identification (\ref{eq_neweq}) now follows. 

Now the proof follows as in item (iv).
That is, use (\ref{eq_IMnS}), the nuclearity of $M_n$ (and thus also preservation of inclusions by $M_n \otimes_{\max} -$), the associativity of the maximal tensor product and the $({\rm el}, {\rm c})$-nuclearity of $\S$, to get the complete order isomorphisms
\begin{align*}
M_n(\S) \otimes_{\rm c} \R 
& \simeq 
M_n \otimes (\S \otimes_{\rm c} \R)
\simeq
M_n \otimes_{\max} (\S \otimes_{\rm el} \R) \\
& \hookrightarrow
M_n \otimes_{\max} (\I(\S) \otimes_{\max} \R)
\simeq
M_n(\I(S)) \otimes_{\max} \R
\simeq
\I(M_n(\S)) \otimes_{\max} \R.
\end{align*}
Therefore we have the following diagram that fixes $M_n(\S) \odot \R$:
\[
\xymatrix{
M_n(\S) \otimes_{\rm c} \R \ar[rr] \ar[rrd] & & M_n(\S) \otimes_{\rm el} \R \ar[d] \\
& & \I(M_n(\S)) \otimes_{\max} \R,
}
\]
where the horizontal arrow is given by ${\rm el} \leq {\rm c}$, the vertical arrow is a complete order embedding by the definition of $\otimes_{\rm el}$, and the diagonal arrow was shown to be a complete order embedding.
This shows that the horizontal map is a complete order isomorphism.
\end{proof}

\begin{remark}
We note that item (vi) above would follow from item (iv), if $\otimes_{\rm c}$ is associative.
However, this is not known.
\end{remark}

\begin{theorem}\label{C:ten pro}
Let $\S$ and $\T$ be operator systems such that $\S \sim_{\Delta} \T$.
Let $(\tau', \tau)$ be a pair of functorial tensor products with $\tau' \leq \tau$ such that:
\begin{enumerate}
\item $\tau'$ and $\tau$ are associative tensor products; or
\item $\tau'$ is associative and $\tau = {\rm el}$; or
\item $\tau'$ is associative and $\tau = {\rm er}$; or
\item $\tau' = {\rm el}$ and $\tau$ is associative; or
\item $\tau' = {\rm er}$ and $\tau$ is associative; or
\item $(\tau', \tau) = ({\rm el}, {\rm c})$.
\end{enumerate}
Then $\S$ is $(\tau', \tau)$-nuclear if and only if $\T$ is $(\tau', \tau)$-nuclear.
\end{theorem}

\begin{proof}
Without loss of generality, 
assume that $\S \subseteq \B(H)$ and $\T \subseteq \B(K)$ are complete operator systems, 
TRO-equivalent by a closed TRO $\M \subseteq \B(H, K)$; thus, 
\[
[\M^* \T \M] = \S \qand [\M \S \M^*] = \T.
\]
Fix an operator system $\cl R$ and assume that $\cl S$ is $(\tau', \tau)$-nuclear. 
We claim that the map
\[
\T \otimes_{\tau} \R \ni t \otimes r \mapsto t \otimes r \in \T \otimes_{\tau'} \R
\]
is (unital and) completely isometric.
By definition it is completely positive, and thus it suffices to show that is has a completely positive inverse.

Towards this end, let $\un{m} \in R(\M)$ be a semi-unit for $\M$ such that
$\sum_{n=1}^{\infty} m_{n} m_{n}^* m = m$ for all $m \in \M$.
In particular, we have
$\sum_{n=1}^{\infty} m_n m_n^* t = t$, $t \in \T$, and also $\sum_{n=1}^{\infty} t m_n m_n^* = t$, $t \in \T$, by taking adjoints. 
Hence, defining
\[
\phi_n(t) = [m_{i}^* t m_{j}]_{i,j=1}^n
\qand
\psi_n([s_{ij}]) = \sum_{i, j = 1}^n m_{i} s_{i,j} m_{j}^*,
\]
we obtain an approximately commutative diagram 
\[
\xymatrix@C=1.5cm{
\T \ar[dr]_{\phi_{n}} \ar@{-->}[rr]^{\id_\T} & & \T \\
& M_{n}(\S) \ar[ur]_{\psi_{n}} &
}
\]
of completely positive completely contractive maps.
Let $\R$ be an operator system and, using Lemma \ref{l_nuc}, let 
$$\si_n \colon M_{n}(\S) \otimes_{\tau'} \R \to M_{n}(\S) \otimes_{\tau} \R$$
be a complete order isomorphism. 
Since $\tau$ and $\tau'$ are assumed to be functorial, 
the composition $\Phi_n := (\psi_n \otimes \id_\R) \circ \si_n \circ (\phi_n \otimes \id_\R)$, that is,
\[
\Phi_n \colon \T \otimes_{\tau'} \R \to M_{n}(\S) \otimes_{\tau'} \R \to M_{n}(\cl S) \otimes_{\tau} \R \to \T \otimes_\tau \R,
\]
is well-defined and completely contractive; by definition,
\[
\Phi_n(t \otimes r) = \sum_{i,j =1}^n m_{i} m_{i}^* t m_{j} m_{j}^* \otimes r \foral t\in \cl T, r \in \R, n\in \bb{N}.
\]
It follows that the sequence $(\Phi_n)_{n\in \bb{N}}$ converges in the point-norm topology, and hence in the Bounded-Weak topology \cite[Chapter 7]{Pau02}, 
to a completely positive completely contractive map $\Phi$.
Since 
\[
\sum_{n, k} m_n m_n^* t m_k m_k^* = t \foral t \in \T,
\]
we have that $\Phi(t \otimes r) = t \otimes r$, and thus $\Phi$ is the identity map with domain $\T \otimes_{\tau'} \R$ and range $\T \otimes_{\tau} \R$, and the proof is complete.
\end{proof}

Since $\max$, $\min$, ${\rm el}$ and ${\rm er}$ are functorial \cite{KPTT11}, we have the following corollary.

\begin{corollary}
Nuclearity, exactness, the (OSLLP), the (WEP), and the (DCEP) are invariants of $\Delta$-equivalence of operator systems.
\end{corollary}

\section{Function systems}\label{s_fs}

A \emph{function system} is an operator subsystem of the (commutative) C*-algebra ${\rm C}(X)$ of all continuous functions on a compact Hausdorff space $X$ that separates the points.
The prototypical example is the space $A(K)$ of continuous affine functions on a compact convex subset $K$ of a locally convex space.
In \cite[Theorem 4.4.3]{DK19}, K. Davidson and M. Kennedy 
showed that $\ca_{\max}(A(K)) \simeq {\rm C}(K)$ as (unital) C*-algebras.
The \emph{\v{S}ilov ideal} in ${\rm C}(K)$, corresponding to $A(K)$, arises from the largest open subset of $K$ that can be removed, retaining the complete order embedding of $A(K)$. 
More precisely, there exists a closed subset $\partial A(K) \subseteq K$ so that the restriction map
\[
{\rm C}(K) \to {\rm C}( \partial A(K) ); \ f \mapsto f|_{\partial A(K)}
\]
is a complete order embedding when restricted to $A(K)$. 
Since affine functions attain their supremum at the extreme points of $K$, we have that $\partial A(K) \subseteq \ol{\partial K}$.
Conversely, since $\ol{\partial K}$ is the smallest set on which all affine functions attain their supremum, we have that 
\[
\partial A(K) = \ol{\partial K}
\qand
\cenv(\S) \simeq {\rm C}( \ol{\partial K} );
\]
in particular, the C*-envelope of a function system is an abelian C*-algebra. 

\begin{theorem}\label{C:fun sys}
Let $\S$ and $\T$ be function systems.
Then $\S \sim_{\Delta} \T$ if and only if $\ol{\S} \simeq \ol{\T}$ up to complete order isomorphism.
\end{theorem}

\begin{proof}
Without loss of generality, we assume that $\S$ and $\T$ are complete operator systems.
Suppose that $\S \sim_{\Delta} \T$.
Proposition \ref{P:reduction} yields complete order isomorphisms
$\phi \colon \S \to \B(H)$ and $\psi \colon \T \to \B(K)$
and a closed TRO $\M$ so that $\phi(\S) \sim_{\rm TRO} \psi(\T)$ via $\M$,
\[
\cenv(\S) \simeq \ca(\phi(\S)) = [\M^* \ca(\psi(\T)) \M]
\qand
\cenv(\T) \simeq \ca(\psi(\T)) = [\M \ca(\phi(\S)) \M^*],
\]
while $\A_\S \simeq [\M^* \M]$ and $\A_\T \simeq [\M \M^*]$.
Since $\psi(\T)$ is unital, there exists a row contraction $\un{m}_1 = (m_{1, n})_n \in R(\M)$ such that
\[
\sum_{n} m_{1, n} m_{1, n}^* = 1_{\psi(\cl T)}.
\]
Let
\[
\vartheta \colon \ca(\phi(\S)) \to \ca(\psi(\T)); \ a \mapsto \un{m}_1 \cdot (a \otimes I_{\ell^2}) \cdot \un{m}_1^* = \sum_n m_{1, n} a m_{1, n}^*;
\]
then $\vartheta$ is a unital completely contractive map.
We first show that $\vartheta$ is a homomorphism.
By \cite[Proposition 1.2.4]{blm} it will follow that $\vartheta$ is a $*$-homomorphism.
To this end, we have that every $m_{1,n}^* m_{1,k}$ is in $\A_\S$ and thus commutes with $a \in \ca(\phi(\S))$.
Hence, if $a,b\in \ca(\phi(\cl S))$, then 
\begin{align*}
\vartheta(a)  \cdot \vartheta(b)
& = 
\sum_{n, k} m_{1, n} a (m_{1, n}^* m_{1, k}) b m_{1, k}^* \\
& =
\sum_{n, k} (m_{1, n} m_{1, n}^*) m_{1, k} a b m_{1, k}^* 
= 
\sum_{k} m_{1, k} a b m_{1, k}^*
=
\vartheta(a b).
\end{align*}

Similarly, there exists a column contraction $\un{m}_2 = (m_{2, n})_n \in C(\M)$ such that
\[
\sum_{n} m_{2,n}^* m_{2,n}  = 1_{\phi(\cl S)}.
\]
Likewise we can define the homomorphism
\[
\varphi \colon \ca(\psi(\T)) \to \ca(\psi(\S)); \ c \mapsto \sum_n m_{2,n}^* c m_{2,n}.
\]
Since $m_{2, n}^* m_{1, k}$ is in $\A_\S$, it commutes with every $a \in \ca(\phi(\S))$, and a direct computation yields
\begin{align*}
(\varphi \circ \vartheta)(a)
& = 
\sum_{n} \sum_{k} (m_{2, n}^* m_{1, k}) \cdot a \cdot m_{1, k}^* m_{2, n} 
 = 
\sum_{n} \sum_{k} a m_{2, n}^* m_{1, k} m_{1, k}^* m_{2, n} \\
& = 
a \sum_n m_{2, n}^* \cdot \left(\sum_{k} m_{1, k} m_{1, k}^*\right) \cdot m_{2, n} 
 =
a \sum_{n} m_{2, n}^* m_{2, n} = a.
\end{align*}
Hence $\varphi$ is the inverse of $\vartheta$ and thus $\vartheta$ (and likewise $\varphi$) is a $*$-isomorphism.
Their restrictions to $\phi(\S)$ and $\psi(\T)$ are unital completely isometric maps that provide the complete order isomorphism between $\S$ and $\T$.
\end{proof}

The proof of Theorem \ref{C:fun sys} requires only commutativity and can be extended to centres of operator systems.
We view $\S$ as sitting inside $\cenv(\S)$ and let the \emph{centre of $\S$} be defined by
\[
\Z(\S) := \Z(\cenv(\S)) \cap \S = \{x \in \S \mid \iota_{\env}(x) y = y \iota_{\env}(x) \foral y \in \cenv(\S) \}.
\]
Since $\S$ generates $\cenv(\S)$ as a C*-algebra, we have that 
\[
\Z(\S) = \{x \in \S \mid \iota_{\env}(x) \iota_{\env}(y) = \iota_{\env}(y) \iota_{\env}(x) \foral y \in \S\}.
\]

\begin{corollary}\label{C:centers}
Let $\S$ and $\T$ be operator systems.
If $\S \sim_{\Delta} \T$ then $\Z(\ol{\S}) \simeq_{\rm c.o.i.} \Z(\ol{\T})$.
Consequently an operator system $\S$ is $\Delta$-equivalent to a function system if and only if $\ol{\S} \sim_{\Delta} \Z(\ol{\S})$.
\end{corollary}

\begin{proof}
By Proposition \ref{P:reduction}, without loss of generality we assume that $\S$ and $\T$ are TRO-equivalent complete operator systems such that $\ca(\S) = \cenv(\S)$ and $\ca(\T) = \cenv(\T)$ are TRO-equivalent via the same closed TRO, say $\M$.
As in the proof of Theorem \ref{C:fun sys}, let $\un{m} = (m_n)_n \in C(\M)$ be a row contraction such that
$\sum_n m_n m_n^* = 1_{\cl T}$
and define the unital completely contractive map
\[
\vartheta \colon \S \to \T; s \mapsto \un{m} \cdot (s \otimes I_{\ell^2}) \cdot \un{m}^* = \sum_n m_{n} s m_{n}^*.
\]
It suffices to show that $\vartheta|_{\Z(\S)}$ has range inside $\Z(\T)$; the result then follows in the same way as in Theorem \ref{C:fun sys}.
Towards this end, let $l_1, l_2 \in \M$ and $a \in \ca(\S)$.
Then $m_n^* l_1 \in \S$ and for $s \in \Z(\S)$ we have
\begin{align*}
\vartheta(s) \cdot l_1 a l_2^* 
& =
\sum_n m_n s (m_n^* l_1) a l_2^*
=
\sum_n m_n m_n^* l_1 s a l_2^* \\
& =
l_1 a s l_2^*
=
\sum_n l_1 a s (l_2^* m_n) m_n^*
=
l_1 a l_2^* \sum_n m_n s m_n^*
=
l_1 a l_2^* \cdot \vartheta(s).
\end{align*}
Since $[\M \ca(\S) \M^*] = \ca(\T)$, we have that $\vartheta(s) \in \Z(\ca(\T))$ and thus $\vartheta(s)$ is in $\Z(\T)$.
\end{proof}

\section{Non-commutative graphs}\label{s_NCgraphs}

Recall that a \emph{non-commutative (NC) graph} \cite{DSW13} is an operator subsystem of $M_d$, for some $d\in \bb{N}$. 
The prototypical example comes from undirected graphs $G$ on a vertex set $[d] := \{1,\dots,d\}$. 
In this case we write $i\sim j$ if $\{i,j\}$ is an edge of $G$, and $i\simeq j$ if $i\sim j$ or $i = j$. 
Following \cite{DSW13}, let 
\begin{equation}\label{eq_SG}
\cl S_G = {\rm span}\{E_{i,j} \mid i\simeq j\}\subseteq M_d,
\end{equation}
for the canonical matrix unit system $(E_{i,j})_{i,j=1}^d$ of $M_d$. 
Since $G$ is undirected, $\cl S_G$ is an operator system; operator systems of this form are called \emph{graph operator systems}.

The C*-algebra generated by $\S_G$ contains all $E_{v_1, v_2}$ as long as there exists a path connecting the vertex $v_1$ with the vertex $v_2$.
Therefore $\ca(\S_G) = \oplus_{j = 1}^n M_{d_j}$, for some $d_j, n\in \bb{N}$ with $d_1 + \cdots + d_n = d$.
By \cite[Theorem 3.2]{op} in particular we have that $\cenv(\S_G) \simeq \oplus_{j = 1}^n M_{d_j}$, where the different components for $\ca(\S_G)$ correspond to the disjoint connected components of the graph $G$.
Therefore $\A_{\S_G} \simeq \oplus_{j = 1}^k M_{r_j}$, for some $r_j, k \in \bb{N}$ with $n \leq k$ and $r_1 + \cdots + r_k = d$.

We can extend these results to non-commutative graphs, up to irreducible representations.
We say that a non-commutative graph $\S \subseteq M_d$ \emph{acts reducibly} if there exists a non-trivial subspace $L \subseteq \bC^d$ such that
\begin{enumerate}
\item $L$ is reducing for $\S$;
\item the restriction map of $\S$ to $L$ is completely isometric.
\end{enumerate}
We say that $\S \subseteq M_d$ \emph{acts irreducibly} if there exists no such non-trivial subspace $L \subseteq \bC^d$; in this case we see that $\ca(\S) = \oplus_{j=1}^n M_{d_j}$ for $d_j, n \in \bN$ with $d_1 + \cdots + d_n = d$.

This irreducibility condition is equivalent to the restriction maps on the components $\bC^{d_j}$ having the strong separation property of Arveson \cite{Arv10}.
It is easily verified that the restriction maps satisfy also the faithfulness and the irreducibility condition of \cite{Arv10}, and thus $\cenv(\S) \simeq \ca(\S)$ by an application \cite[Theorem 1.2]{Arv10}.
Arveson's proof uses the Choquet boundary machinery.
Let us provide an elementary proof here.

\begin{proposition}\label{p_ncenv}
Let $\S \subseteq M_d$ be an non-commutative graph.
If $\S$ acts irreducibly then 
\[
\cenv(\S) \simeq \ca(\S) = \oplus_{j = 1}^n M_{d_j}, \ d_j, n \in \bN, d_1 + \cdots + d_n = d,
\]
for some $d_j, n \in \bN$.
Moreover, the identity representation on $\S$ is a Choquet representation.
\end{proposition}

\begin{proof}
Suppose that $\ca(\S) = \oplus_{j=1}^n M_{d_j}$ and let $\Phi \colon \ca(\S) \to \cenv(\S)$ be the canonical $*$-epimorphism.
Up to a re-arranging assume that
\[
\ker \Phi = \oplus_{j=1}^{n'} M_{d_j}, \quad n' < n.
\]
Let $L := \sum_{j=n' + 1}^{n} \bC_{d_j}$ and for every $s \in \cl S$ we write $s = s|_{L} \oplus s|_{L^\perp}$, so that $s|_{L^\perp} \in \ker \Phi$ for all $s \in \S$.
Since $\ker\Phi$ is the \v{S}ilov ideal, and thus boundary, we have that
\[
\| s \| = \| s + \ker \Phi \| = \inf\{ s|_{L} + s|_{L^\perp} + x \mid x \in \ker\Phi \} = \| s|_{L} \|
\]
for all $s \in \S$.
Similar arguments work for the matrix levels giving that the restriction map on $L$ is completely isometric.
By definition $L$ is reducing for $\S$ which leads to a contradiction, unless $n' = 0$.
Therefore $\Phi$ is a $*$-isomorphism.

The second part follows from Arveson's Boundary Theorem because $\ca(\S)$ is a finite dimensional C*-algebra.
\end{proof}

Let $H$ and $K$ be finite dimensional Hilbert spaces and $\cl S\subseteq \cl B(H)$ and $\cl T\subseteq \cl B(K)$ be non-commutative graphs. 
A \emph{cohomomorphism} from $\cl T$ to $\cl S$ is a unital completely positive map $\Phi \colon \cl B(K)\to \cl B(H)$, which admits a Kraus representation 
\begin{equation}\label{eq_ngh}
\Phi(T) = \sum_{i=1}^r A_i^* T A_i
\quad \text{such that} \quad
A_i^* \cl T A_j\subseteq \cl S \foral i,j\in [r];
\end{equation}
if such a map exists, we write $\cl T\to \cl S$. 
A \emph{bicohomomorphism between $\S$ and $\T$} consists of two cohomomorphisms $\cl T \to \cl S$ and $\cl S \to \cl T$ with Kraus representations induced by $\{A_i\}_{i \in [r]}$ and $\{B_j^*\}_{j \in [r']}$, respectively, such that
\[
\spn\{A_i \mid i \in [r]\} = \spn\{B_j \mid j \in [r']\}.
\]

The notion of cohomomorphism is equivalent to the notion of a non-commutative graph homomorphism, introduced by D. Stahlke in \cite{Sta16}.
Stahlke works with the orthogonal complements of operator systems as opposed to non-commutative graphs; in view of our emphasis on operator systems, it is natural to start with the concept defined here. 
By \cite[Theorem 8]{Sta16}, if $G$ and $H$ are classical graphs, the existence of a classical graph homomorphism from $G$ to $H$ is equivalent to the existence of a cohomomorphism from $\cl S_{\bar{H}}$ to $\cl S_{\bar{G}}$ (we denote by $\bar{G}$ the complement of a graph $G$).

\begin{proposition}\label{p_coho}
Let $\S$ and $\T$ be non-commutative graphs. 
There is a bicohomomorphism between $\S$ and $\T$ if and only if $\S$ and $\T$ are bihomomorphically equivalent, if and only if $\S$ and $\T$ are TRO-equivalent.
\end{proposition}

\begin{proof}
Assume that $\S \subseteq M_m$ and $\T \subseteq M_k$. 
If $\T \to \S$ and $\Phi \colon M_k \to M_m$ is a unital completely positive map with Kraus representation satisfying (\ref{eq_ngh}), then the subspace 
\[
\X := {\rm span}\{A_i \mid i \in [r]\}
\]
satisfies the desired conditions. 
Conversely, suppose that $\X \subseteq M_{k,m}$ is a subspace such that $I \in [\X^*\X]$ and $\X^* \T \X\subseteq \S$. 
By finite dimensionality, there exist operators $A_i, C_i \in \X$, $i = 1,\dots, p$, such that $I_m = \sum_{i=1}^p A_i^* C_i$. 
Let $A = (A_i)_{i=1}^p$ and $C = (C_i)_{i=1}^p$, viewed as column operators. 
Then $I_m = A^* C$, implying that ${\rm rank} (C) = m$. 
It follows that ${\rm rank} (C^* C) = m$, and hence $C^* C$ is invertible. 
By the remarks on \cite[p. 558]{Sta16} and using that $\S$ is a bimodule over $\ca(\X^*\X)$ as in Proposition \ref{P:tro os}, this yields $\T \to \S$.
By dual arguments we have that $\S \to \T$.

The second equivalence is proven in the same way as in the proof of Proposition \ref{P:tro os}.
\end{proof}

\begin{remark}
From the proof of Proposition \ref{p_coho}, we have that two non-commutative graphs are TRO-equivalent if and only if there are finitely many $A_i, B_j$ in a non-degenerate operator space $\X$ such that
\[
\spn\{A_i \mid i \in [r]\} = \X = \spn\{B_j \mid j \in [r']\},
\]
with
\[
A_{i_1}^* \T A_{i_2} \subseteq \S, \text{ for } i_1, i_2 \in [r]
\qand
B_{j_1} \S B_{j_2}^* \subseteq \T \text{ for } j_1, j_2 \in [r'].
\]
Indeed, these inclusions are the only requirement for a Kraus representation.
\end{remark}

We will show that $\Delta$-equivalence of non-commutative graphs can always be realised as a finite-dimensional TRO-equivalence.
We isolate the following idea from \cite{KP05, KT03}.

\begin{lemma}\label{l_decomp}
Suppose that $\M$ is a TRO such that
\[
[\M^* \M] = \oplus_{j=1}^k M_{r_j}
\qand
[\M \M^*] = \oplus_{i=1}^m M_{\ell_i},
\]
and set $r := \sum_{j=1}^k r_j$ and $\ell := \sum_{i=1}^m \ell_i$.
Then there exist $N \in \bN$ and surjective maps $g \colon [r] \to [N]$ and $f \colon [\ell] \to [N]$ such that 
\[
\M  = \{ (a_{i,j}) \mid a_{i,j} = 0 \text{ if } f(i) \neq g(j) \} \subseteq M_{\ell, r}.
\]
\end{lemma}

\begin{proof}
Since $\M$ is a left module over $[\M \M^*]$ and a right module over $[\M^* \M]$, there exist partitions 
\begin{equation}\label{eq_km}
[k] = \cup_{v=1}^N \alpha_v \ \mbox{ and } \ [m] = \cup_{v=1}^{N} \beta_v
\end{equation}
such that, if $[r^v] = \cup_{i\in \alpha_v} [r_i]$ and $[\ell^v] = \cup_{j\in \beta_v} [\ell_j]$, then 
\[
\cl M = \oplus_{v=1}^N M_{\ell^v, r^v}.
\]
In addition $\M$ is an equivalence bimodule and thus
\[
\oplus_{v=1}^N M_{r^v} = (\oplus_{v=1}^N M_{\ell^v, r^v})^* \cdot (\oplus_{v=1}^N M_{\ell^v, r^v}) = [\M^* \M] = \oplus_{i=1}^k M_{r_i}.
\]
From this we derive that $|\al_v| = 1$, that is, the decompositions (\ref{eq_km}) are trivial, and so $N = k$.
Likewise $N = m$ and therefore we have $k = N = m$.
Let 
\[
g \colon [r] \to [N] \qand f \colon [l]\to [N]
\]
be the natural surjections arising from the decompositions $[r] = \cup_{i=1}^N [r_i]$ and $[\ell] = \cup_{j=1}^N [\ell_j]$.
Then a direct verification (see e.g. \cite{KP05, KT03}) shows that
\[
\cl M = \left\{\left(a_{i,j}\right) \mid a_{i,j} = 0 \text{ if } f(i) \neq g(j) \right\},
\]
and the proof is complete.
\end{proof}

\begin{theorem}\label{t_ncg}
Let $\S \subseteq M_{d}$ and $\T \subseteq M_{d'}$ be irreducibly acting non-commutative graphs.
Then $\S \sim_\Delta \T$ if and only if $\S \sim_{\rm TRO} \T$.
\end{theorem}

\begin{proof}
It suffices to show that if $\S \sim_{\Delta} \T$ then we can visualise it through a TRO-equivalence on the embeddings $\S \subseteq M_{d}$ and $\T \subseteq M_{d'}$.
Using Proposition \ref{p_ncenv} let 
\[
\cenv(\S) \simeq \ca(\S) = \oplus_{j=1}^n M_{d_j}
\]
with $d_1 + \cdots + d_n = d$.
Likewise write
\[
\cenv(\T) \simeq \ca(\T) = \oplus_{i=1}^{n'} M_{d_i'}
\]
with $d_1' + \cdots + d_{n'}' = d'$.
By Proposition \ref{P:reduction} we have that for the embedding $\S \subseteq \ca(\S) \subseteq M_{d}$ 
there exists a (closed) TRO $\M$ such that
\[
[\M^* \M] = \A_\S
\qand
[\M \M^*] \simeq \A_\T,
\]
and moreover for the canonical representation
\[
\psi \colon \T \to \B(\M \otimes_{[\M^*\M]} \bC^d)
\]
we get that $\S \sim_{\rm TRO} \psi(\T)$ by $\M$, so that
\[
[\M^* \psi(\T) \M] = \S
\qand
[\M \S \M^*] = \psi(\T).
\]
In addition $\ca(\S) \sim_{\rm TRO} \ca(\psi(\T))$ and $\ca(\psi(\T)) \simeq \cenv(\T)$.

Since $\ca(\psi(\T))$ is an orthogonal sum of matrix C*-algebras, its representation on $\M \otimes_{[\M^*\M]} \bC^d$ is unitarily equivalent to the direct sum of identity representations on each summand with some multiplicity.
We will show that the multiplicities are equal to one.
To reach a contradiction, suppose that the multiplicity on a summand $M_{d_i'}$ is greater than one.
Then there is a subspace 
\[
K \subseteq \M \otimes_{[\M^*\M]} \bC^d
\]
that is reducing for $\ca(\psi(\T))$ and the restriction on $K^\perp$ is injective on $\ca(\psi(\T))$.
Hence the restriction on $K^\perp$ is completely isometric on $\ca(\psi(\T))$.
Recall that applying the functor $\M^* \otimes_{[\M\M^*]} -$ on $(\M \otimes_{[\M^*\M]} \bC^d, \ca(\psi(\T)))$ produces a $*$-representation on $\ca(\S)$ that is unitarily equivalent to the identity representation.
Hence 
\[
L:= \M^* \otimes_{[\M \M^*]} K
\]
induces a subspace in $\bC^d$, through the unitary equivalence, that is reducing for $\ca(\S)$ (and so for $\S$) and $\id|_L$ is completely isometric on $\ca(\S)$ (and so for $\S$).
This contradicts the assumption that $\S$ acts irreducibly.

Therefore the representation of $\cenv(\T) \simeq \ca(\psi^*(\T))$ on $\M \otimes_{[\M^*\M]} \bC^{d}$ is unitarily equivalent to the identity representation.
Hence there is a unitary
\[
U \colon \M \otimes_{[\M^* \M]} \bC^d \to \bC^{d'}
\]
such that $\ad_U \circ \psi(t) = t$ for all $t \in \T$.
We conclude that $\S \sim_{\rm TRO} \T$ by $U \M$.
Indeed, on one hand we have
\[
[U \M S \M^* U^*] = [U \psi(\T) U^*] = \T
\]
and on the other hand we have
\[
[\M^* U^* \T U \M] = [\M^* \psi(\T) \M] = \S,
\]
while 
\[
[U \M \M^* U^*U \M ] = [U \M \M^* \M] = U \M,  
\]
i.e., $U\M$ is a TRO, and the proof is complete.
 \end{proof}

Let $G$ and $H$ be graphs with finite vertex sets $[k]$ and $[m]$, respectively.
We say that $G$ is a \emph{pullback of $H$}, if there exists a surjective map $f \colon [k] \to [m]$ such that 
\[
x\simeq x' \textup{ in $G$} \qiff f(x)\simeq f(x') \textup{ in $H$.}
\] 

\begin{corollary}\label{c_graphs}
Let $G$ and $H$ be graphs.
The following are equivalent:
\begin{enumerate}
\item $\cl S_G \sim_{\rm TRO} \cl S_H$;
\item $\S_G \sim_{\rm \Delta} \S_H$;
\item $G$ and $H$ are pullbacks of isomorphic graphs. 
\end{enumerate}
\end{corollary}

\begin{proof} 
Suppose that $G$ and $H$ are graphs on $d$ and $d'$ vertices respectively.
The equivalence of the first two items is a consequence of Theorem \ref{t_ncg}.
For their equivalence with item (iii) we will use the following notation.
Given sets $\kappa \subseteq X_1\times X_2$ and $\kappa' \subseteq X_2\times X_3$, we define
\[
\kappa \circ \kappa' 
:= 
\left\{ (x_1,x_3) \in X_1 \times X_3 \mid
\exists x_2\in X_2 \textup{ s.t. } (x_1,x_2)\in \kappa \textup{ and } (x_2,x_3)\in \kappa' \right\}.
\]
We also write 
\[
\ol{\ka} := \{(i,j) \mid (j,i) \in \ka\}.
\]

\smallskip

For the implication [(i)$\Rightarrow$(iii)], let $\M \subseteq \B(\bC^d, \bC^{d'})$ such that 
\begin{equation}\label{eq_tromsg}
\M^* \S_H \M \subseteq \S_G
\qand
\M \S_G \M^* \subseteq \S_H.
\end{equation}
By Proposition \ref{P:reduction} and Theorem \ref{t_ncg} we can assume that 
\[
[\M^* \M] = \A_{\S_G} := \oplus_{j=1}^k M_{r_j}
\qand
[\M \M^*] = \A_{\S_H} := \oplus_{i=1}^m M_{\ell_i}.
\]
By Lemma \ref{l_decomp} we obtain an $N \in \bN$ and onto maps $g \colon [r] \to [N]$ and $f \colon [\ell] \to [N]$ for $r=d$ and $\ell = d'$ such that
\[
\M = \{ (a_{i,j}) \mid a_{i,j} = 0 \text{ if } f(i) \neq g(j) \} \subseteq M_{\ell, r} = M_{d', d}.
\]
By definition we see that if $g(j) = g(j')$ then $E_{j,j'} \in \A_{\S_G} \subseteq \S_{G}$ and thus $j \simeq j'$ in $G$; likewise if $f(i) = f(i')$ then $i \simeq i'$ in $H$.
Let 
\[
\kappa_G := \{(j,j')\in [r]\times [r] \mid j\simeq j'\}
\qand
\kappa_H := \{(i,i')\in [\ell]\times [\ell] \mid i\simeq i'\}.
\]
We also define
\[
\ka := \{(i,j) \mid f(i) = g(j)\},
\]
and note that
\[
\ol{\ka} \circ \ka \subseteq \ka_G
\qand
\ka \circ \ol{\ka} \subseteq \ka_H.
\]
Indeed for $(j', j) \in \ol{\ka} \circ \ka$ there are $(i,j) \in \ka$ and $(j', i) \in \ol{\ka}$; then $g(j) = f(i) = g(j')$ and thus $j \simeq j'$.
Moreover we have that 
\[
\de_G := \{ (j,j) \mid j \in G\} \subseteq \ol{\ka} \circ \ka
\qand
\de_H := \{ (i,i) \mid i \in H\} \subseteq \ka \circ \ol{\ka}.
\]
Indeed due to surjectivity of $f$ for $j$ in $G$ there exists $i$ in $H$ such that $g(j) = f(i)$ and then $(i,j) \in \ka$; likewise for $\de_H$.
Conditions (\ref{eq_tromsg}) translate into 
\begin{equation} 
\overline{\kappa}\circ \kappa_H \circ \kappa \subseteq \ka_G
\qand
\kappa\circ \kappa_G \circ \overline{\kappa} \subseteq \ka_H.
\end{equation}
Indeed for $(j_1, j_2) \in \overline{\kappa}\circ \kappa_H \circ \kappa$, there is $(i_1, i_2) \in \ka_H$ such that $(i_1, j_1), (i_1, j_2) \in \ka$.
By definition, then $E_{i_1, i_2} \in \S_H$ and $E_{i_1, j_1}, E_{i_2, j_2} \in \M$, and conditions (\ref{eq_tromsg}) yield
\[
E_{j_1, j_2} = E_{i_1, j_1}^* E_{i_1, i_2} E_{i_2, j_2} \in \M^* \S_H \M \subseteq \S_G.
\]
Thus $(j_1, j_2) \in \ka_G$.
By applying successively, and because $\de_G \circ \ka_G = \ka_G = \ka_G \circ \de_G$, we get that
\begin{align*}
\ka_G 
\subseteq
\de_G \circ \ka_G \circ \de_G
\subseteq 
\ol{\ka} \circ \ka \circ \ka_G \circ \ol{\ka} \circ \ka
\subseteq
\overline{\kappa}\circ \kappa_H \circ \kappa \subseteq \ka_G.
\end{align*}
We thus have equalities, and symmetrically we derive
\begin{equation}\label{eq_kappaGH}
\ol{\ka} \circ \ka_H \circ \ka = \ka_G
\qand
\ka \circ \ka_G \circ \ol{\ka} = \ka_H.
\end{equation}
Note that in passing we also obtain that
\begin{equation}
\ol{\ka} \circ \ka \circ \ka_G \circ \ol{\ka} \circ \ka = \ka_G
\qand
\ka \circ \ol{\ka} \circ \ka_H \circ \ka \circ \ol{\ka} = \ka_H.
\end{equation}

Next we define the common graph.
For $v\in [N]$, set
\[
X_v := \{j \in [r] \mid g(j) = v\},
\]
and define a graph $\wt{G}$ with vertex set $[N]$ by letting 
\[
\textup{$v \simeq v'$ if $j \simeq j'$ for some $j \in X_v$ and $j' \in X_{v'}$}.
\] 
Note that the map $g$ realises $G$ as a pullback of $\wt{G}$. 
Indeed, if $j\simeq j'$ in $G$ then $g(j)\simeq g(j')$ in $\wt{G}$ by definition. 
Conversely, if $g(j)\simeq g(j')$ in $\wt{G}$ then there exist $x, x' \in [r]$ such that $g(j) = g(x)$, $g(j') = g(x')$ and $x \simeq x'$. 
We have that $(j, x), (j', x') \in \ol{\kappa} \circ \kappa$ and, since $(x, x')\in \kappa_G$, we get 
\[
\{(j,j')\} = \{(j,x)\}\circ \{(x,x')\}\circ \{(x',j')\}\subseteq \ol{\ka} \circ \ka \circ \ka_G \circ \ol{\ka} \circ \ka = \ka_G,
\]
giving the required $j \simeq j'$.
Similarly, for $v\in [N]$, set
\[
Y_v = \{i \in [l] \mid f(i) = v\},
\]
and define a graph $\wt{H}$ with vertex set $[N]$ by letting 
\[
\textup{$v \simeq v'$ if $i \simeq i'$ for some $i \in Y_v$ and $i' \in Y_{v'}$}. 
\]
By the previous paragraph, $f$ realises $H$ as a pullback of $\wt{H}$. 

We finally show that $\wt{G} = \wt{H}$.
By construction, $\wt{G}$ and $\wt{H}$ have the same vertex set.
Hence we need to show that $v \simeq v'$ in $\wt{G}$ if and only if $v \simeq v'$ in $\wt{H}$.
It suffices to prove only one direction, and then symmetry establishes the claim.
Suppose that $v \simeq v'$ in $\wt{G}$; then there exist $x, x' \in [r]$ such that $g(x) = v$, $g(x') = v'$ and $(x, x')\in \kappa_G$. 
By equation (\ref{eq_kappaGH}), there exist $i, i'\in [l]$ such that $(i, x), (i', x')\in \kappa$ and $(i, i')\in \kappa_H$. 
Since $f(i) = v$ and $f(i') = v'$, we have that $v \simeq v'$ in $\wt{H}$. 

\smallskip

For the implication [(iii) $\Rightarrow$ (i)] let $g \colon [r]\to [N]$ and $f \colon [l]\to [N]$ be maps that realise $G$ and $H$ as pullbacks of a graph on $N$ verties. 
A straightforward verification shows that the TRO
\[
\cl M = \{(a_{i,j})\in M_{l,r} \mid a_{i,j} = 0 \textup{ if } g(j) \neq f(i) \},
\]
implements a TRO-equivalence between $\cl S_G$ and $\cl S_H$. 
\end{proof}

\section{$\Delta$-embeddings of operator systems} \label{s_emb}

Let $\cl S$ and $\cl T$ be operator systems. 
Recall that if $\S \sim_{\Delta} \T$ then there exist unital complete order embeddings $\phi$ and $\psi$ of $\cl S$ and $\cl T$, respectively, such that $\phi(\S) \sim_{\rm TRO} \psi(\T)$, satisfying the additional property that $\ca(\phi(\S)) \simeq \cenv(\S)$ and $\ca(\psi(\T)) \simeq \cenv(\T)$.
In addition, we have that $\S \otimes \fK \simeq \T \otimes \fK$ up to a complete order isomorphism, and thus by equation (\ref{eq_fK}), or alternatively by \cite[Proposition 2.37]{CS20}, we obtain an extension $*$-isomorphism $\cenv(\S) \otimes \fK \to \cenv(\T) \otimes \fK$.
In this section, we investigate the case where only ``half'' of these relations are satisfied; 
namely, we are interested in characterising the existence of a $*$-epimorphism $\cenv(\T) \otimes \fK \to \cenv(\S) \otimes \fK$ that maps $\T \otimes \fK$ onto $\S \otimes \fK$.

In order to keep track of the Hilbert space representations, we will be using the following notation.
For a complete order embedding $\psi$ of $\S$ we write
\[
\psi^{**} \colon \S \to \ca(\psi(\S)) \hookrightarrow \ca(\psi(\S))^{**}.
\]
Recall that we write $\iota_{\env}$ for the embedding $\S \hookrightarrow \cenv(\S)$, and $\iota_{\max}$ for the embedding $\S \hookrightarrow \ca_{\max}(\S)$.

\begin{proposition}\label{P:local embed}
Let $\S$ and $\T$ be operator systems and $\psi \colon \T \to \B(K)$ be a complete order embedding.
The following are equivalent:
\begin{enumerate}
\item there is a $*$-epimorphism $\ca(\psi(\T)) \otimes \fK \to \cenv(\S) \otimes \fK$ that maps $\psi(\T) \otimes \fK$ onto $\iota_{\env}(\S) \otimes \fK$;

\item there is a projection $p$ in the commutant of $\ca(\psi(\T))^{**}$ such that $\S \sim_{\Delta} p [\psi^{**}(\T)]$; 

\item there is a faithful $*$-representation $\rho$ of $\ca(\psi(\T))$ and a projection $p \in \ca(\rho \circ \psi(\T))'$ such that $\S \sim_{\Delta} p [\rho \circ \psi(\T)]$.
\end{enumerate}
\end{proposition}

\begin{proof}
\noindent
[(i) $\Rightarrow$ (ii)]: Let $\Phi \colon \ca(\psi(\T)) \otimes \fK \to \cenv(\S) \otimes \fK$ be the given $*$-epimorphism, which extends to a weak*-continuous $*$-epimorphism on the second duals
\[
\Phi^{**} \colon \left(\ca(\psi(\T)) \otimes \fK \right)^{**} \to \left( \cenv(\S) \otimes \fK \right)^{**}.
\]
However by \cite[p. 44, equation (1.62)]{blm} we have that
\[
\left(\fC \otimes \fK \right)^{**} \simeq \fC^{**} \bar\otimes \B(\ell^2)
\textup{ for any C*-algebra $\fC$},
\]
which extends the canonical inclusion $\fC \otimes \fK \hookrightarrow \fC^{**} \bar\otimes \B(\ell^2)$.
Therefore we obtain a weak*-continuous $*$-epimorphism (denoted in the same fashion)
\[
\Phi^{**} \colon \ca(\psi(\T))^{**} \bar\otimes \B(\ell^2) \to \cenv(\S)^{**} \bar\otimes \B(\ell^2).
\]
Let $P$ be the projection in the commutant of $\ca(\psi(\T))^{**} \bar\otimes \B(\ell^2)$ corresponding to $\ker \Phi^{**}$.
Since
\[
\left(\ca(\psi(\T))^{**} \bar\otimes \B(\ell^2) \right)' = \left[ \ca(\psi(\T))^{**} \right]' \otimes \bC I,
\]
we deduce that $P = p \otimes I$ for some projection $p$ in the commutant of $\ca(\psi(\T))^{**}$.
Moreover, by construction, the induced $*$-isomorphism
\[
\Phi^{**} (p \otimes I) \colon p [\ca(\psi(\T))^{**}] \bar\otimes \B(\ell^2) \to \cenv(\S)^{**} \bar\otimes \B(\ell^2); \ 
p x \otimes y \equiv P (x \otimes y) \mapsto \Phi^{**}(x \otimes y),
\]
maps $p [\psi^{**}(\T)] \otimes \fK$ onto $\iota_{\env}(\S) \otimes \fK$.
We derive that $p [\psi^{**}(\T)] \otimes \fK \simeq \iota_{\env}(\S) \otimes \fK$, and thus $p [\psi^{**}(\T)] \sim_{\Delta} \S$.

\smallskip

\noindent
[(ii) $\Rightarrow$ (iii)]: This follows by letting $\rho$ be the universal representation of $\ca(\psi(\T))$.

\smallskip

\noindent
[(iii) $\Rightarrow$ (i)]: Suppose that $p [\rho \circ \psi(\T)] \sim_{\Delta} \S$ for a projection $p$ in the commutant of $\ca(\rho \circ \psi(\T))$.
Then we get a complete order isomorphism $p [\rho \circ \psi(\T)] \otimes \fK \to \S \otimes \fK$, that lifts to a $*$-isomorphism of their C*-envelopes.
Using \cite[Proposition 2.37]{CS20}, we obtain a $*$-isomorphism
\[
\Phi_0 \colon \cenv( p [\rho \circ \psi(\T)] ) \otimes \fK \to \cenv(\S) \otimes \fK
\]
that maps $p [\rho \circ \psi(\T)] \otimes \fK$ onto $\S \otimes \fK$.
Since $p$ is in the commutant of $\ca(\rho \circ \psi(\T))$, and thus in the commutant of $\rho \circ \psi(\T)$, we can consider the $*$-epimorphism
\[
\Phi_1 \colon \ca(\rho \circ \psi(\T)) \to p \ca(\rho \circ \psi(\T)) = \ca( p [\rho \circ \psi(\T)] ) \to \cenv( p [\rho \circ \psi(\T)] ).
\]
The required $*$-epimorphism is given by 
\[
\Phi_0 \circ (\Phi_1 \otimes \id) \circ (\id \otimes \rho) \colon \ca(\psi(\T)) \otimes \fK \to \ca(\rho \circ \psi(\T)) \otimes \fK \to \cenv( p [\rho \circ \psi(\T)] ) \otimes \fK \to \cenv(\S) \otimes \fK,
\]
and the proof is complete.
\end{proof}

\begin{definition}
Let $\S$ and $\T$ be operator systems.

\begin{enumerate}
\item We say that $\S$ \emph{$\Delta$-embeds} in $\T$ (denoted $\S \subset_{\Delta} \T$) if $\S \sim_{\Delta} p [\psi(\T)]$ for a complete order embedding $\psi$ of $\T$ and a projection $p \in \psi(\T)'$.

\item We say that $\S$ \emph{$\Delta_{\env}$-embeds} in $\T$ (denoted $\S \subset_{\Delta_{\env}} \T$) if $\S \sim_\Delta p [\iota_{\env}^{**}(\T )] $ for the complete order embedding $\iota_{\env}^{**} \colon \T \to \cenv(\T)^{**}$ and a projection $p \in \iota_{\env}^{**}(\T)'$.
\end{enumerate}
\end{definition}

\begin{proposition}\label{p_epimm}
Let $\S$ and $\T$ be operator systems.
Then 
\begin{itemize}
\item[(i)]
$\S \subset_{\Delta} \T$ if and only if there is a $*$-epimorphism $\ca_{\max}(\T) \otimes \fK \to \cenv(\S) \otimes \fK$ that maps $\iota_{\max}(\T) \otimes \fK$ onto $\iota_{\env}(\S) \otimes \fK$;

\item[(ii)]
$\S \subset_{\Delta_{\env}} \T$ if and only if there is a $*$-epimorphism $\ca_{\env}(\T) \otimes \fK \to \cenv(\S) \otimes \fK$ that maps $\iota_{\env}(\T) \otimes \fK$ onto $\iota_{\env}(\S) \otimes \fK$.
\end{itemize}
\end{proposition}

\begin{proof}
(i) 
Let $\psi$ be a complete order embedding of $\T$ so that $\S \sim_{\Delta} p [\psi(\T)]$ for a projection in the commutant of $\psi(\T)$.
By the implication [(iii)$\Rightarrow$(i)] of Proposition \ref{P:local embed}, there exists a $*$-epimorphism $\ca(\psi(\T)) \otimes \fK \to \cenv(\S) \otimes \fK$ that maps $\psi(\T) \otimes \fK$ onto $\iota_{\env}(\S) \otimes \fK$.
By the universal property of $\ca_{\max}(\T)$, there exists a $*$-epimorphism
\[
\ca_{\max}(\T) \otimes \fK \to \ca(\psi(\T)) \otimes \fK \to \cenv(\S) \otimes \fK,
\]
which maps $\iota_{\max}(\T) \otimes \fK$ onto $\iota_{\env}(\S) \otimes \fK$.

Conversely, suppose that there exists a $*$-epimorphism $\ca_{\max}(\T) \otimes \fK \to \cenv(\S) \otimes \fK$ that maps $\iota_{\max}(\T) \otimes \fK$ onto $\iota_{\env}(\S) \otimes \fK$.
By the implication [(i)$\Rightarrow$(ii)] of Proposition \ref{P:local embed}, $\S \sim_{\Delta} p[\iota_{\max}^{**}(\T)]$ for a projection $p$ in the commutant of $\iota_{\max}^{**}(\T)$.

\smallskip

\noindent
(ii) 
This is the equivalence of items (i) and (ii) of Proposition \ref{P:local embed} for $\psi = \iota_{\env}$.
\end{proof}

Proposition \ref{p_epimm} has the following immediate corollary.

\begin{corollary}
$\Delta_{\env}$-embedding of operator systems is a transitive relation.
\end{corollary}

\begin{remark}\label{R:abelian}
In the case of commutative C*-algebras we have that the relations $\Delta$-embeding and $\Delta_{\env}$-embedding both coincide with the existence of surjective $*$-homomorphisms.
Indeed, $\Delta$-equivalence for C*-algebras is Rieffel's strong Morita equivalence, and thus corresponds to $*$-isomorphisms for commutative C*-algebras.

We also note that $\Delta_{\env}$-embedding is not anti-symmetric modulo $\Delta$-equivalence.
An example to this end is given in \cite[Example 4.9]{Ele19}.
Let $X$ and $Y$ be non-homeomorphic compact Hausdorff spaces such that there are continuous injections
\[
\vartheta \colon X \to Y
\qand
\varphi \colon Y \to X.
\]
For example, one can consider a disc and a torus in $\bC$, and use the fact that the disc can be homeomorphically mapped onto a disc inside the torus.
Then there are $*$-epimorphisms
\[
\wt{\vartheta} \colon {\rm C}(Y) \to {\rm C}(X)
\qand
\wt{\varphi} \colon {\rm C}(X) \to {\rm C}(Y).
\]
Hence
\[
{\rm C}(X) \subset_{\Delta_{\env}} {\rm C}(Y)
\qand
{\rm C}(Y) \subset_{\Delta_{\env}} {\rm C}(X).
\]
However, ${\rm C}(X)$ is not $\Delta$-equivalent to ${\rm C}(Y)$, as $X$ and $Y$ are not 
homeomoprhic.
\end{remark}

\begin{remark}
In contrast with the case of commutative C*-algebras, $\Delta_{\env}$-embedding is rigid for operator systems of graphs.
Indeed, the C*-envelope of an operator system of a graph is a finite-dimensional C*-algebra, with the summands corresponding to the disjoint components of the graph \cite{op}.
Therefore $\S_G \subset_{\Delta_{\env}} \S_H$ if and only if $\S_G \sim_{\Delta} \S_{H'}$ for $H'$ a disjoint union of connected components of $H$; equivalently, by Corollary \ref{c_graphs}, if and only if $G$ and $H'$ are pullbacks of isomorphic graphs for $H'$ a disjoint union of connected components of $H$.

In more detail, recall the characterization obtained in item (ii) of Proposition \ref{p_epimm}.
Suppose that $\cenv(\S_G) \simeq \oplus_{i=1}^k M_{r_i}$ and $\cenv(\S_H) \simeq \oplus_{j=1}^m M_{\ell_j}$.
For a $*$-epimorphism
\[
\Phi \colon \oplus_{j=1}^m M_{\ell_j} \otimes \fK \to  \oplus_{i=1}^k M_{r_i} \otimes \fK,
\]
we have that its kernel (up to a rearrangement of the indices) is of the form $\oplus_{j=k+1}^{m} M_{\ell_j} \otimes \fK$.
Therefore by restricting to $\oplus_{j=1}^k \bC^{\ell_j} \otimes \ell^2$ we obtain a $*$-isomorphism 
\[
\oplus_{j=1}^k M_{\ell_j} \otimes \fK \to  \oplus_{i=1}^k M_{r_i} \otimes \fK.
\]
Its restriction to $\S_H$ annihilates the elements of $H$ appearing in the components for $j=k+1, \dots, m$, and therefore we have a unital complete order isomorphism $\S_{H'} \to \S_{G}$, where $H'$ is the graph from the generators in the first $m$ components of $H$.
By Corollary \ref{c_graphs} we have that $H'$ and $G$ are pullbacks of isomorphic graphs.
\end{remark}

The following result will be used subsequently, but it may also be interesting in its own right. 

\begin{lemma}\label{L:split}
Let $\phi \colon \A \to \B$ and $\psi \colon \B \to \A$ be $*$-epimorphisms.
Let $\phi\colon \A^{**} \to \B^{**}$ and $\psi^{**} \colon \A^{**} \to \B^{**}$ be the w*-continuous $*$-epimorphic extensions to their double duals.
Then there exist central projections $p \in \A^{**}$ and $q \in \B^{**}$ such that
\[
p \A^{**} \simeq q \B^{**}
\qand
(1_{\A^{**}} - p) \A^{**} \simeq (1_{\B^{**}} - q) \B^{**}.
\]
Moreover, if $\X \subseteq \A$ and $\Y \subseteq \B$ are closed subspaces with $\phi(\X) = \Y$ and $\psi(\Y) = \X$ then 
\[
p \X \simeq q \Y
\qand
(1_{\A^{**}} - p) \X \simeq (1_{\B^{**}} - q) \Y.
\]
\end{lemma}

\begin{proof}
As usual, we denote by $\cl Z(\cl M)$ the centre of a von Neumann algebra $\cl M$. 
Let $e_1 \in \Z(\A^{**})$ and $f_1 \in \Z(\B^{**})$ be such that the restrictions
\[
\phi^{**}|_{e_1 \A^{**}} \colon e_1 \A^{**} \to \B^{**}
\qand
\psi^{**}|_{f_1 \B^{**}} \colon f_1 \B^{**} \to \A^{**}
\] 
are weak*-isomorphisms.
Let
\[
e_2 := e_1 (\phi^{**}|_{e_1 \A^{**}})^{-1}(f_1) \leq e_1;
\]
thus, $e_2\in \cl Z(\cl A^{**})$ and the map
\[
\rho := \psi^{**}|_{f_1 \B^{**}} \circ \phi^{**}|_{e_2 \A^{**}} \colon e_2 \A^{**} \to f_1 \B^{**} \to \A^{**}
\]
is a weak*-isomorphism.

We define a family of central projections in $\A^{**}$ by letting
\[
e_{n+2} = \rho^{-1}(e_n) \foral n \in \bb{N}.
\]
By definition, $\rho(e_2) = 1$ and so 
\[
e_3 = \rho^{-1}(e_1) \leq \rho^{-1}(1) = e_2.
\]
Continuing inductively, 
\[
e_{n+2} = \rho^{-1}(e_n) \leq \rho^{-1}(e_{n-1}) = e_{n+1} \foral n \geq 2,
\]
and hence $\{e_n\}_{n\in \bb{N}}$ is a decreasing sequence of central projections.
Let 
\[
p := \sum_{n=0}^\infty e_{2n} - e_{2n+1};
\]
then 
\[
1_{\A^{**}} - p = \left( \sum_{n=0}^\infty e_{2n + 1} - e_{2n + 2} \right) \oplus \wedge_{n=0}^\infty e_n \leq e_1.
\]
Furthermore, let 
\[
r := \rho^{-1}(p) = \sum_{n=1}^{\infty} e_{2n} - e_{2n+1} \qand q := \phi^{**}(r).
\]
By the definition of the projections $e_n$, $n\in \bb{N}$, the projection $r$ satisfies the conditions
\[
r = r e_2 \leq p \qand e_1 = r \oplus (1_{\A^{**}} - p).
\]

To conclude the proof we claim that the maps
\[
\phi^{**}|_{r \A^{**}} \circ \rho^{-1}|_{p \A^{**}} \colon p \A^{**} \to q \B^{**}
\]
and
\[
\phi^{**}|_{(1_{\A^{**}} - p) \A^{**}} \colon (1_{\A^{**}} - p) \A^{**} \to (1_{\B^{**}} - q) \B^{**}
\]
are weak*-isomorphisms.
Indeed, first note that, since $r \leq e_1$, the composition $\phi^{**}|_{r \A^{**}} \circ \rho^{-1}$ is faithful.
For the surjectivity, we note that 
\[
\left( \phi^{**}|_{r \A^{**}} \circ \rho^{-1} \right) (p \A^{**}) = \phi^{**}(r \A^{**}) = q \B^{**}.
\]
On the other hand, $(1_{\A^{**}} - p) \leq e_1$ and hence $\phi^{**}|_{(1_{\A^{**}} - p) \A^{**}}$ is faithful.
For the surjectivity, note that
\[
\phi^{**}(r \oplus (1_{\A^{**}} - p)) = \phi^{**}(e_1) = 1_{\B^{**}},
\]
and so
\[
\phi^{**}(1_{\A^{**}} - p) = 1_{\B^{**}} - \phi^{**}(r) = 1_{\B^{**}} - q.
\]
Thus, 
\[
\phi^{**}|_{(1_{\A^{**}} - p) \A^{**}} \left( (1_{\A^{**}} - p) \A^{**} \right) = \phi^{**}(1_{\A^{**}} - p) \B^{**} = (1_{\B^{**}} - q) \B^{**},
\]
which completes the proof of the first statement.

Let $\X \subseteq \A$ and $\Y \subseteq \B$ are closed subspaces such that $\phi(\X) = \Y$ and $\psi(\Y) = \X$.
It suffices to check that the restrictions of the constructed maps satisfy the required properties. 
However, this follows directly by surjectivity of the restrictions of $\phi$ and $\psi$; indeed, 
\[
\left(\phi^{**}|_{r \A^{**}} \circ \rho^{-1}|_{p \A^{**}}\right)(p \X)
=
\left(\phi^{**}|_{r \A^{**}} \circ \rho^{-1}|_{p \A^{**}}\right)(p) \left(\phi^{**}\circ \psi^{**}\circ \phi^{**}\right)(\X) = q \Y,
\]
and
\[
\phi^{**}|_{(1_{\A^{**}} - p) \A^{**}} ((1_{\A^{**}} - p) \X)
=
\phi^{**}|_{(1_{\A^{**}} - p) \A^{**}} (1_{\A^{**}} - p) \phi^{**}(\X)
=
(1_{\B^{**}} - q) \Y,
\]
as required.
\end{proof}

\begin{remark}
The following chain of isomorphisms provides a short visualisation of the proof:
\begin{align*}
p\A^{**} \oplus (1 - p) \A^{**} 
& \simeq^{\rho^{-1} \oplus \id} r \A^{**} \oplus (1-p) \A^{**} 
\simeq^{\phi^{**}|_{e_1 \cdot} \oplus \phi^{**}|_{e_1 \cdot }} q \B^{**} \oplus (1-q) \B^{**}.
\end{align*}
\end{remark}

Lemma \ref{L:split} and Proposition \ref{p_epimm} have the following consequence.

\begin{theorem}\label{T:anti-sym} 
Let $\S$ and $\T$ be operator systems.
If $\S \subset_{\Delta_{\env}} \T $ and $\T \subset_{\Delta_{\env}} \S$ then there exist projections  $p$ in the commutant of $\cenv(\T)^{**}$ and $q$ in the commutant of $\cenv(\S)^{**}$ such that 
\[
p [\iota_{\env}^{**}(\T)] \sim_{\Delta} q [\iota_{\env}^{**}(\S)]
\qand
(1 - p) [\iota_{\env}^{**}(\T)] \sim_{\Delta} (1 - q) [\iota_{\env}^{**}(\S)].
\]
\end{theorem}

\begin{proof}
By Proposition \ref{p_epimm} we have a $*$-epimorphism $\cenv(\T) \otimes \fK \to \cenv(\S) \otimes \fK$ that maps $\iota_{\env}(\T) \otimes \fK$ onto $\iota_{\env}(\S) \otimes \fK$, and a $*$-epimorphism $\cenv(\S) \otimes \fK \to \cenv(\T) \otimes \fK$ that maps $\iota_{\env}(\S) \otimes \fK$ onto $\iota_{\env}(\T) \otimes \fK$.
By Lemma \ref{L:split} we obtain central projections $P \in (\cenv(\T) \otimes \fK)^{**}$ and $Q \in (\cenv(\S) \otimes \fK)^{**}$, such that
\[
P (\cenv(\T) \otimes \fK)^{**} \simeq Q (\cenv(\S) \otimes \fK)^{**}
\qand
(1 - P) (\cenv(\T) \otimes \fK)^{**} \simeq (1 - Q) (\cenv(\S) \otimes \fK)^{**}.
\]
Recall that
\[
(\fC \otimes \fK)^{**} \simeq \fC^{**} \otimes \B(\ell^2)
\]
by a $*$-isomorphism that fixes the embedding of $\fC \otimes \fK$ for any C*-algebra $\fC$.
Since the center of $\fC^{**} \otimes \B(\ell^2)$ is $\Z(\fC^{**}) \otimes I$, we have that $P = p \otimes I$ and $q = q \otimes I$ for central projections in $\cenv(\T)^{**}$ and $\cenv(\S)^{**}$, respectively.
Therefore, after applying the $*$-isomorphism of the double duals, we can restrict to $\iota_{\env}^{**}(\T) \otimes \fK$ and $\iota_{\env}^{**}(\S) \otimes \fK$, Lemma \ref{L:split} yields
\[
\left( p [\iota_{\env}^{**}(\T)] \right) \otimes \fK \simeq \left( q [\iota_{\env}^{**}(\S)] \right) \otimes \fK
\qand
\left( (1 - p) [\iota_{\env}^{**}(\T)] \right) \otimes \fK \simeq \left( (1 - q) [\iota_{\env}^{**}(\S)] \right) \otimes \fK,
\]
which completes the proof.
\end{proof}

\begin{example}
The converse of Theorem  \ref{T:anti-sym} does not hold.
Indeed, let 
\[
\A := \{(x_n) \mid  \lim_{n\to\infty} x_n \mbox{ exists}\}
\]
as a C*-subalgebra of $\ell^\infty$.
The embedding
\[
\ell^1 \hookrightarrow \A; y \mapsto L_y;
\ 
L_y(x) = \sum_n x_n y_n
\]
gives a $*$-isomorphism $\A^{**} \simeq \ell^\infty$.
We can also see $\A \oplus \A$ as a C*-subalgebra of $\ell^\infty \oplus \ell^\infty$; thus we also have $(\A \oplus \A)^{**} \simeq \ell^{\infty} \oplus \ell^\infty$.
By setting $p = 1 \oplus 0$ we see that
\[
p (\A \oplus \A) \simeq \A \qand (1-p)(\A \oplus \A) \simeq \A.
\]

Recall by Remark \ref{R:abelian} that the $\Delta_{\env}$-embedding refers to $*$-epimorphisms in the case of commutative C*-algebras.
Although we have a $*$-epimorphism $\A \oplus \A \to \A$, there does not exist a $*$-epimorphism $\A \to \A \oplus \A$ (and thus the converse of Theorem  \ref{T:anti-sym} does not hold).
Towards this end, let a $*$-epimorphism $\phi \colon \A \to \A \oplus \A$, yielding a $*$-epimorphism of the second duals. 
Let $r \in \ell^\infty$ be a projection such that the induced map 
\[
\phi' \colon \ell^\infty \cdot r \to \ell^\infty \oplus \ell^\infty; \ x r \mapsto \phi(x),
\]
is a $*$-isomorphism.
Let $x, y \in \A$ be such that $\phi(x) = p$ and $\phi(y) = 1-p$, for $p = 1 \oplus 0$.
Since $\phi'(xy) = \phi(x y) = 0$, we have that $x y r = 0$.
On the other hand, since $r_n \neq 0$ for infinitely many $n \in \bN$, we have that $x_n = 0$ for infinitely many $n \in \bN$, or $y_n = 0$ for infinitely many $n \in \bN$.
Without loss of generality, assume that $x_n = 0$ for infinitely many $n \in \bN$, implying that $x \in c_0$ (since $x \in \A$).
Thus $x r \in c_0$.
Since $\phi'$ is a (unitary implemented) $*$-isomorphism, we have that $\phi(c_0 \cdot r) = \phi'(c_0 \cdot r) = c_0 \oplus c_0$ as fixing the compact operators, thus $\phi(x r) \in c_0 \oplus c_0$.
This leads to the contradiction that $p \in c_0 \oplus c_0$.
\end{example}  

\smallskip

\noindent {\bf Data availability statement.}
For the purposes of publication of this article, we note that data sharing is not applicable as no datasets were generated or analysed during the underlying research.

\smallskip

\noindent {\bf Conflict of interest statement.}
On behalf of all authors, the corresponding author states that there is no conflict of interest.

\smallskip

\noindent {\bf Open access statement.}
For the purpose of open access, the second author has applied a Creative Commons Attribution (CC BY) license to any Author Accepted Manuscript version arising.


\end{document}